\ifdefined\pdfoutput\pdfoutput=1\fi
\documentclass[11pt]{amsart} 
\usepackage{amsmath,amssymb,amsthm} 
\usepackage[all]{xy} 
\usepackage{iftex}
\ifpdf\else
  
\fi
\usepackage{tikz} 
\usetikzlibrary{arrows.meta} 
\usepackage{comment} 

\newtheorem{theorem}{Theorem}[section] 
\newtheorem{lemma}[theorem]{Lemma} 
\newtheorem{proposition}[theorem]{Proposition} 
\newtheorem{corollary}[theorem]{Corollary} 

\newtheorem{maintheorem}{Theorem} 
 
\newtheorem{maincorollary}[maintheorem]{Corollary} 

\theoremstyle{definition} 
\newtheorem{definition}[theorem]{Definition} 
\newtheorem{remark}[theorem]{Remark} 
\newtheorem{example}[theorem]{Example}

\numberwithin{equation}{section}

\renewcommand{\phi}{\varphi}

\newcommand{\Eff}{\operatorname{Eff}} 
\newcommand{\Coker}{\operatorname{Coker}} 
\newcommand{\Fix}{\operatorname{Fix}} 
 
\newcommand{\Homeo}{\operatorname{Homeo}} 
\newcommand{\id}{\operatorname{id}} 
 
\newcommand{\Iso}{\operatorname{Iso}} 
\newcommand{\Ker}{\operatorname{Ker}} 
\newcommand{\Nor}{\operatorname{N}} 
\newcommand{\Out}{\operatorname{Out}} 
\newcommand{\St}{\operatorname{St}} 
\newcommand{\supp}{\operatorname{supp}} 

\newcommand{\N}{\mathbb{N}} 
\newcommand{\Z}{\mathbb{Z}}

\newcommand{\cCO}{\mathcal{CO}} 
\newcommand{\cG}{\mathcal{G}} 
\newcommand{\cH}{\mathcal{H}} 
\newcommand{\cK}{\mathcal{K}} 
\newcommand{\cN}{\mathcal{N}} 
\newcommand{\cS}{\mathcal{S}} 

\newcommand{\sD}{\mathsf{D}} 
\newcommand{\sF}{\mathsf{F}} 
\newcommand{\sA}{\mathsf{A}}

\title[Maximal subgroups of topological full groups of shifts of finite type] 
{Maximal subgroups of topological full groups of \\ 
shifts of finite type} 
\author{Hiroki Matui} 
\thanks{The author was supported by JSPS KAKENHI Grant Number JP23K22397.} 
\address{Graduate School of Mathematical Sciences, 
The University of Tokyo, Japan} 
\email{hiroki@ms.u-tokyo.ac.jp} 
\date{}


\begin{document} 
\begin{abstract} 
We study maximal subgroups of topological full groups of irreducible 
one-sided shifts of finite type whose associated groupoids have 
Cantor unit space. 
For each wide open subgroupoid of such a groupoid, we construct 
a subgroupoid graph and identify the subgroupoid with a full compact 
open reduction of the corresponding countable graph groupoid. 
Every minimal wide open subgroupoid is either AF or purely infinite. 
A proper wide open subgroupoid is prime if it admits no proper 
intermediate wide open subgroupoid. 
We prove a trichotomy for prime subgroupoids: they are minimal, arise 
from a clopen bipartition, or have a unique nonempty proper closed 
invariant set, which is nowhere dense. 
For a prime subgroupoid, the normalizer of its topological full group 
is maximal in the ambient topological full group precisely when it 
maps onto the ambient abelianization. 
An analogous criterion holds in every subgroup containing the 
ambient commutator subgroup. 
For a minimal prime subgroupoid, its topological full group is maximal 
precisely when it equals its normalizer and the inclusion induces 
a surjection on first groupoid homology. 
We also describe all fully amplified maximal subgroups, namely 
those closed under local agreement. 
For Thompson's group $V$, the stabilizer of every simple type system 
is maximal, extending results of Belk, Bleak, Quick and Skipper 
for finite type systems and certain classes of infinite ones. 
Finally, every subgroup containing the ambient commutator subgroup 
has continuum many pairwise nonisomorphic maximal subgroups 
acting minimally on the unit space. 
\end{abstract} 

\maketitle

%%%%%%%%%%%%%%%%%%%%%%%%%%%%%%%%%%%%%%%%%%%%%%%%%%%%%%%%%%%%
\section{Introduction} 
%%%%%%%%%%%%%%%%%%%%%%%%%%%%%%%%%%%%%%%%%%%%%%%%%%%%%%%%%%%%

Topological full groups of irreducible one-sided shifts of finite type 
generalize Thompson's group $V$ and the Higman--Thompson groups 
\cite{Ma15crelle}. 
We study their maximal subgroups through the ample groupoids 
associated with these shifts, which we call SFT groupoids. 
A wide open subgroupoid of an SFT groupoid gives rise to two natural 
subgroups of the ambient topological full group: 
its own topological full group and the normalizer of that group. 
We investigate when these subgroups are maximal and how maximality is 
reflected in the structure of the subgroupoid. 
Our approach combines graph models, the dynamics of prime subgroupoids 
and the maps induced on abelianizations and groupoid homology. 

Belk, Bleak, Quick and Skipper \cite{BBQS25TAMS} use type systems 
to construct maximal subgroups of Thompson's group $V$, 
including an uncountable family of pairwise nonisomorphic examples. 
Their type systems are equivalence relations on finite binary words 
satisfying compatibility conditions with passage to children. 
Grigorchuk and Vorobets \cite{GV26ETDS} study maximal subgroups 
arising from closed-set stabilizers and clopen partitions 
in a broader dynamical setting. 
Here we interpret type systems in terms of wide open subgroupoids 
and develop the resulting theory for general irreducible SFT groupoids. 
This gives structural results for subgroupoids themselves, 
as well as criteria and constructions for maximal subgroups. 

Throughout the introduction, let $(V,E)$ be a finite directed graph 
whose adjacency matrix is irreducible and is not a permutation matrix. 
Let $\cH:=\cG_{(V,E)}$ be its SFT groupoid, and put $X:=\cH^{(0)}$. 
Thus $X$ is a Cantor set, and $\cH$ is minimal, effective and 
purely infinite. 
All groupoids in this paper are second countable and Hausdorff. 
We write $\sF(\cH)$ for the topological full group of $\cH$ and 
$\sD(\cH)$ for its commutator subgroup. 
Our maximality criteria apply to every group $\Gamma$ with 
$\sD(\cH)\subset\Gamma\subset\sF(\cH)$. 
These are precisely the nontrivial normal subgroups of $\sF(\cH)$ 
\cite[Theorem 4.16]{Ma15crelle}. 
A subgroupoid is wide if its unit space is $X$. 
A proper wide open subgroupoid $\cG\subsetneq\cH$ is called prime 
if there is no proper intermediate wide open subgroupoid 
\cite[Definition 5.1]{Ma26maximal1}. 

We first associate a subgroupoid graph $\Delta_{\cG}$ to every 
wide open subgroupoid $\cG\subset\cH$. 
Its vertices record which finite paths can be interchanged by 
prefix replacements belonging to $\cG$. 
The graph has a natural projection to $(V,E)$ and a root section. 
Wide open subgroupoids between $\cG$ and $\cH$ correspond to 
quotients of $\Delta_{\cG}$ as a rooted graph over $(V,E)$ 
(Propositions~\ref{prop:subgroupoidgraphconverse} 
and~\ref{prop:subgroupoidgraphquotient}). 
In particular, primeness has a graph-theoretic characterization 
(Corollary~\ref{cor:primegraph}). 
For the full $2$-shift groupoid $\cH_2$, these subgroupoid graphs 
are the type graphs of \cite{BBQS25TAMS}, and primeness corresponds 
to simplicity of type systems 
(Remark~\ref{rem:subgroupoidgraphBBQS}). 

The graph also gives a concrete groupoid model: 
\[ 
\cG\cong\cG_{\Delta_{\cG}}|S, 
\] 
where $S$ is a full compact open subset of the unit space of the 
countable graph groupoid $\cG_{\Delta_{\cG}}$ 
(Proposition~\ref{prop:subgroupoidgraphmodel}). 
This requires neither minimality nor primeness of $\cG$, 
nor finiteness of $\Delta_{\cG}$. 
It allows the homology of $\cG$ to be computed from the adjacency 
matrix of its subgroupoid graph. 
It also shows that every minimal wide open subgroupoid of $\cH$ 
is either AF or purely infinite 
(Corollary~\ref{cor:minimalsubgroupoiddichotomy}). 

To state the structure theorem for prime subgroupoids, we recall 
the regular reduction and the $Y$-preserving subgroupoid from 
\cite{Ma26maximal2}. 
For a closed set $Y\subset X$, the regular reduction $\cH|Y$ 
consists of those arrows with source and range in $Y$ which lie 
in a compact open bisection $U$ satisfying $YU=UY$. 
It has the relative topology and may be smaller than the full 
restriction $r^{-1}(Y)\cap s^{-1}(Y)$. 
The union of all such $Y$-preserving bisections is a wide open 
subgroupoid $\cS_Y$ of $\cH$, with 
\[ 
\cS_Y|Y=\cH|Y,\qquad 
\sF(\cS_Y)=\St(\cH,Y), 
\] 
where $\St(\cH,Y)$ denotes the setwise stabilizer of $Y$ 
in $\sF(\cH)$. 

\begin{maintheorem}\label{intro:trichotomy} 
Let $\cG\subsetneq\cH$ be a prime wide open subgroupoid. 
Then exactly one of the following holds. 
\begin{enumerate} 
\item $\cG$ is minimal. 
In this case, $\cG$ is either AF or purely infinite. 
\item There is a partition $X=X_1\sqcup X_2$ into nonempty clopen sets 
such that $\cG=\cH|X_1\sqcup\cH|X_2$. 
\item There is a unique nonempty proper closed $\cG$-invariant set 
$Y\subset X$. 
It is nowhere dense and either finite or a Cantor set, 
the regular reduction $\cH|Y$ is minimal, and $\cG=\cS_Y$. 
In particular, $\sF(\cG)=\St(\cH,Y)$. 
\end{enumerate} 
\end{maintheorem} 

This combines Theorem~\ref{thm:primetrichotomy}, 
Corollary~\ref{cor:minimalsubgroupoiddichotomy} 
and the identification at the beginning of 
Section~\ref{sec:exceptional}. 
The trichotomy gives a dynamical interpretation of the finite 
simple type systems classified in \cite[Theorem 4.8]{BBQS25TAMS}, 
and applies without a finiteness assumption on the subgroupoid graph. 
The underlying trichotomy holds for general minimal ample groupoids 
with Cantor unit space (Remark~\ref{rem:primetrichotomy}). 
In the SFT setting, it remains open whether a proper prime minimal 
wide open subgroupoid can be AF. 
Case \textup{(3)} connects the present theory with the closed-set 
stabilizers studied in \cite{Ma26maximal2}. 

The three cases admit a common maximality criterion. 
For groups $K\subset L$, write $\Nor(L,K)$ for the normalizer 
of $K$ in $L$, and write $H_1(L)$ for the abelianization of $L$. 

\begin{maintheorem}\label{intro:maximality} 
Let $\cG\subsetneq\cH$ be a prime wide open subgroupoid, put 
$N:=\Nor\left(\sF(\cH),\sF(\cG)\right)$, and let 
$\pi:\sF(\cH)\to H_1(\sF(\cH))$ be the abelianization map. 
For every group $\Gamma$ with 
$\sD(\cH)\subset\Gamma\subset\sF(\cH)$, the subgroup 
$\Gamma\cap N$ is proper in $\Gamma$, and 
\[ 
\Gamma\cap N\text{ is maximal in }\Gamma 
\iff \pi(\Gamma\cap N)=\pi(\Gamma). 
\] 
The condition on abelianizations holds automatically if $\cG$ 
is not minimal. 
In particular, $N\cap\sD(\cH)$ is maximal in $\sD(\cH)$. 
\end{maintheorem} 

This is Corollary~\ref{cor:primenormalizermaximal}. 
The nonminimal cases follow from 
Proposition~\ref{prop:clopenpartitionnormalizer} on clopen partitions 
and Theorem~\ref{thm:exceptionalstabilizer} on closed-set stabilizers. 
The main step is the minimal case. 
For a prime minimal $\cG\subsetneq\cH$, we prove that every subgroup 
of $\sF(\cH)$ containing $\sD(\cG)$ and an element outside $N$ 
must contain $\sD(\cH)$ 
(Theorem~\ref{thm:minimalovergroup}). 
Primeness forces an equivalence relation built from local 
three-cycles to identify all finite paths with a common terminal 
vertex, allowing us to recover the generators of $\sD(\cH)$. 
Once the ambient commutator subgroup has been obtained, 
maximality reduces to the image in the ambient abelianization. 
This is a generation strategy also used for the factor-map 
and subgroupoid constructions in \cite{Ma26maximal1}. 

When $\cG\subsetneq\cH$ is prime and minimal, Li's exact sequence 
\cite[Corollary E]{Li25ForumPi} relates this criterion to groupoid 
homology. 
The topological full group $\sF(\cG)$ itself is maximal in 
$\sF(\cH)$ if and only if it equals its normalizer and the map 
$H_1(\cG)\to H_1(\cH)$ induced by inclusion is surjective 
(Corollary~\ref{cor:minimalfullmaximal} (2)). 
The two conditions are independent, as illustrated in 
Section~\ref{subsec:normalizerexamples}. 
An analogous criterion for maximality of $\sD(\cG)$ in 
$\sD(\cH)$ is expressed by injectivity of the map 
$N/\sD(\cG)\to H_1(\sF(\cH))$ induced by $\pi$ 
(Corollary~\ref{cor:minimalcommutatormaximal}). 

We next ask which maximal subgroups can be described directly 
as topological full groups of wide open subgroupoids. 
For a subgroup $M\subset\sF(\cH)$, 
let $\cG_M$ be the subgroupoid of germs of its action on $X$, 
viewed as a wide open subgroupoid of $\cH$. 
Then $\sF(\cG_M)$ consists of the homeomorphisms of $X$ 
which locally agree with elements of $M$. 
Following \cite[Definition 2.5]{GV26ETDS}, this is the full 
amplification of $M$, and $M$ is fully amplified if 
$M=\sF(\cG_M)$. 
The same reference also calls such groups ample; 
we use fully amplified for groups and reserve ample for groupoids. 
Proposition~\ref{prop:maximalgerms} shows that every fully amplified 
maximal subgroup of $\sF(\cH)$ arises from a proper prime 
wide open subgroupoid. 
This observation is not specific to SFT groupoids. 
Combining it with the structure and maximality results above 
gives the following description. 

\begin{maincorollary}\label{intro:fullyamplified} 
A subgroup $M\subset\sF(\cH)$ is fully amplified and maximal 
in $\sF(\cH)$ if and only if it has one of the following forms. 
\begin{enumerate} 
\renewcommand{\theenumi}{\alph{enumi}} 
\item $M=\St(\cH,A)$ for a nonempty proper clopen set $A\subset X$ 
such that $2[1_A]\neq[1_X]$ in $H_0(\cH)$. 
\item $M=\St(\cH,Y)$ for a nonempty closed nowhere dense set 
$Y\subset X$ such that the regular reduction $\cH|Y$ is minimal. 
\item $M=\sF(\cG)$ for a minimal prime wide open subgroupoid 
$\cG\subsetneq\cH$ such that 
$\Nor\left(\sF(\cH),\sF(\cG)\right)=\sF(\cG)$ and the map 
$H_1(\cG)\to H_1(\cH)$ induced by inclusion is surjective. 
\end{enumerate} 
These three alternatives are mutually exclusive. 
\end{maincorollary} 

This is Corollary~\ref{cor:amplemaximal}. 
In case \textup{(a)}, the homological condition means that no element 
of $\sF(\cH)$ exchanges $A$ and $X\setminus A$. 
Cases \textup{(a)} and \textup{(b)} are exactly the maximal subgroups 
which do not act minimally on $X$, in accordance with 
\cite[Theorems 6.6 and 6.19]{GV26ETDS} and \cite{Ma26maximal2}. 
A maximal subgroup which is not fully amplified has full 
amplification $\sF(\cH)$ and therefore acts minimally on $X$, 
since full amplification does not change orbits. 
The normalizer criterion in Theorem~\ref{intro:maximality} 
also covers maximal subgroups which are not fully amplified; 
Example~\ref{ex:paritynormalizers} (2) gives such a normalizer. 

For a simple type system $\mathcal P$ and its corresponding 
subgroupoid $\cG\subsetneq\cH_2$, the dictionary in 
Remark~\ref{rem:subgroupoidgraphBBQS}, together with 
Lemma~\ref{lem:groupnormalizers} and 
Theorem~\ref{thm:exceptionalstabilizer}, gives 
\[ 
\Fix(\mathcal P)=\sF(\cG),\qquad 
\operatorname{Stab}(\mathcal P) 
=\Nor\left(\sF(\cH_2),\sF(\cG)\right). 
\] 
Thus our criteria give the following consequences for $V$. 

\begin{maincorollary}\label{intro:typesystems} 
Let $\cH_2$ be the full $2$-shift groupoid, so that 
$\sF(\cH_2)=V_2$ is Thompson's group $V$. 
\begin{enumerate} 
\item For every simple type system $\mathcal P$, finite or infinite, 
$\operatorname{Stab}(\mathcal P)$ is a maximal subgroup of $V_2$. 
\item For a simple type system $\mathcal P$, the subgroup 
$\Fix(\mathcal P)$ is maximal in $V_2$ if and only if 
$\Fix(\mathcal P)=\operatorname{Stab}(\mathcal P)$. 
\item A maximal subgroup $M$ of $V_2$ is fully amplified 
if and only if 
$M=\Fix(\mathcal P)=\operatorname{Stab}(\mathcal P)$ 
for some simple type system $\mathcal P$. 
\end{enumerate} 
\end{maincorollary} 

This is Corollary~\ref{cor:typesystems}. 
Since $V_2$ is perfect, the abelianization condition in 
Theorem~\ref{intro:maximality} is automatic. 
The first assertion includes 
\cite[Theorems 5.12, 5.13 and 5.15]{BBQS25TAMS} 
and treats all simple type systems uniformly. 
The second and third assertions describe precisely when the 
subgroup $\Fix(\mathcal P)$ itself is maximal and which maximal 
subgroups are fully amplified. 
The detailed comparison with \cite{BBQS25TAMS} is given in 
Section~\ref{subsec:BBQScomparison}. 

Our final main result constructs maximal subgroups acting 
minimally on the whole unit space. 

\begin{maintheorem}\label{intro:uncountable} 
There exists a family $(\cG_t)_{t\in T}$, with $\#T=2^{\aleph_0}$, 
of proper prime minimal purely infinite wide open subgroupoids 
of $\cH$ such that, for every group $\Gamma$ with 
$\sD(\cH)\subset\Gamma\subset\sF(\cH)$, the groups 
$\Gamma\cap\sF(\cG_t)$, $t\in T$, are pairwise nonisomorphic 
maximal subgroups of $\Gamma$ acting minimally on $X$. 
\end{maintheorem} 

This is Theorem~\ref{thm:uncountableminimal}; 
the actions are minimal because $\Gamma\cap\sF(\cG_t)$ contains 
$\sD(\cG_t)$, which acts minimally on $X$. 
The construction starts with the infinite graphs used in 
\cite[Theorem 7.5]{BBQS25TAMS}. 
We modify the ambient graph and use the classification of 
irreducible SFT groupoids by Matsumoto and the author 
\cite[Theorem 3.6]{MM14Kyoto} to obtain the prescribed $\cH$. 
The resulting subgroupoids have zeroth homology groups with different 
prime torsion. 
By the isomorphism theorem \cite[Theorem 3.9]{Ma15crelle}, this 
distinguishes the maximal subgroups up to abstract isomorphism. 
A single family of subgroupoids works for every $\Gamma$ in the 
statement, so the construction extends the family of 
\cite{BBQS25TAMS} both to arbitrary irreducible SFT groupoids 
with Cantor unit space and to all the intermediate groups 
considered here. 

The maximal stabilizers constructed in 
\cite[Proposition 4.16]{Ma26maximal2} preserve nowhere dense Cantor 
sets, whereas the groups in Theorem~\ref{intro:uncountable} 
act minimally on $X$. 
For $\Gamma=\sF(\cH)$, the latter groups are fully amplified 
and belong to case \textup{(c)} of 
Corollary~\ref{intro:fullyamplified}. 

We also discuss embeddings, realizations and examples. 
In Proposition~\ref{prop:binarygraphrealization}, we give a direct 
proof of a special case of Nyland and Ortega's embedding theorem 
\cite[Theorem 11.16]{NO19IJM}, as a complement to our graph model. 
We also recall its consequences for embeddings of SFT and AF 
groupoids with Cantor unit space into $\cH_2$ with wide open images. 
In a different direction, every SFT groupoid with Cantor unit space 
is strongly realizable in $\cH$, that is, occurs as its full restriction 
to a nowhere dense Cantor set, with the relative topology 
(Proposition~\ref{prop:SFTstrongrealization}). 
This complements the AF realization theorem of \cite{Ma26maximal2}. 
The examples also explain how outer symmetries and homological 
obstructions affect maximality of a subgroupoid's topological full 
group and its normalizer. 

Section~2 recalls the groupoid, graph and homological preliminaries. 
Section~\ref{sec:subgroupoidgraph} develops subgroupoid graphs, 
their graph-groupoid models and normalizers. 
Section~\ref{sec:trichotomy} proves the prime trichotomy and treats 
clopen partitions, while Section~\ref{sec:exceptional} studies 
the case of a nowhere dense invariant set. 
Section~\ref{sec:minimal} proves the generation and maximality 
theorems for minimal subgroupoids and derives the unified criteria. 
Section~\ref{sec:uncountable} constructs the uncountable families. 
Section~\ref{sec:BBQS} compares our results with those on type systems 
and discusses embeddings, realizations and examples.

%%%%%%%%%%%%%%%%%%%%%%%%%%%%%%%%%%%%%%%%%%%%%%%%%%%%%%%%%%%%
\section{Preliminaries} 
%%%%%%%%%%%%%%%%%%%%%%%%%%%%%%%%%%%%%%%%%%%%%%%%%%%%%%%%%%%%

%%%%%%%%%%%%%%%%%%%%%%%%%%%%%%%%%%%%%%%%%%%%%%%%%%%%%%%%%%%%
\subsection{Ample groupoids}\label{subsec:amplegroupoids} 
%%%%%%%%%%%%%%%%%%%%%%%%%%%%%%%%%%%%%%%%%%%%%%%%%%%%%%%%%%%%

We fix notation and recall the basic terminology for ample groupoids. 

The cardinality of a set $A$ is written $\#A$ and 
the characteristic function of $A$ is written $1_A$. 
The finite cyclic group of order $p$ is denoted by $\Z/p$. 
We say that a subset of a topological space $X$ is clopen 
if it is both closed and open. 
We let $\cCO(X)$ denote 
the set of all nonempty compact open subsets of $X$. 
A topological space is said to be totally disconnected 
if its connected components are singletons. 
By a Cantor set, 
we mean a nonempty compact, metrizable, totally disconnected space 
with no isolated points. 
It is known that any two such spaces are homeomorphic. 

In this paper, by an \'etale groupoid 
we mean a second countable locally compact Hausdorff groupoid 
such that the range map is a local homeomorphism. 
(We emphasize that our \'etale groupoids are always assumed to be Hausdorff, 
while non-Hausdorff groupoids are also studied actively.) 
We refer the reader to \cite{Re_text,Re08Irish} 
for background material on \'etale groupoids. 
For an \'etale groupoid $\cG$, 
we let $\cG^{(0)}$ denote the unit space and 
let $s$ and $r$ denote the source and range maps, 
i.e.\ $s(g)=g^{-1}g$, $r(g)=gg^{-1}$. 
A subset $U\subset\cG$ is called a bisection 
if $r|U$ and $s|U$ are injective. 
The \'etale groupoid $\cG$ has a basis for its topology 
consisting of open bisections. 
For $x\in\cG^{(0)}$, 
the set $r(s^{-1}(x))$ is called the orbit 
(or $\cG$-orbit) of $x$. 
When every orbit is dense in $\cG^{(0)}$, $\cG$ is said to be minimal. 
An \'etale groupoid $\cG$ is called ample 
if its unit space $\cG^{(0)}$ is totally disconnected. 
An \'etale groupoid is ample if and only if 
it has a basis for its topology consisting of compact open bisections. 

For $x\in\cG^{(0)}$, 
we write $\cG_x=r^{-1}(x)\cap s^{-1}(x)$ and 
call it the isotropy group of $x$. 
The isotropy subgroupoid of $\cG$ is 
$\Iso(\cG):=\{g\in\cG\mid r(g)=s(g)\}=\bigsqcup_{x\in\cG^{(0)}}\cG_x$. 
The interior $\Iso(\cG)^\circ$ is an open subgroupoid of $\cG$. 
We say that $\cG$ is principal if $\Iso(\cG)=\cG^{(0)}$. 
We say that $\cG$ is effective 
if the interior of $\Iso(\cG)$ is $\cG^{(0)}$. 
Note that the term essentially principal is also used 
in the literature. 
In this paper, unless specified otherwise, 
an ample groupoid is always assumed to be effective. 
A subgroupoid $\cH\subset\cG$ is said to be wide 
if $\cH^{(0)}=\cG^{(0)}$. 
A wide open subgroupoid of an ample groupoid is again ample, 
and is effective whenever the ambient groupoid is effective. 

For an ample groupoid $\cG$ and an open subset $Y\subset\cG^{(0)}$, 
we write $\cG|Y:=r^{-1}(Y)\cap s^{-1}(Y)$, 
which is again an ample groupoid, with unit space $Y$. 
Since $\cG$ is effective, so is $\cG|Y$. 
We say that $Y$ is full if $Y$ meets every $\cG$-orbit. 
Two ample groupoids $\cG_1$ and $\cG_2$ are said to be 
Kakutani equivalent 
if there exist full clopen subsets $Y_i\subset\cG_i^{(0)}$ 
such that $\cG_1|Y_1$ and $\cG_2|Y_2$ are isomorphic. 
In particular, $\cG$ and $\cG|Y$ are Kakutani equivalent 
for every full clopen subset $Y\subset\cG^{(0)}$. 

For closed subsets $Y\subset\cG^{(0)}$, we use the conventions of 
\cite[Section 3.1]{Ma26maximal2}. 
An arrow $g\in r^{-1}(Y)\cap s^{-1}(Y)$ is regular for $Y$ 
if it belongs to a compact open bisection $U$ with $YU=UY$. 
The regular reduction $\cG|Y$ consists of these arrows, 
with the relative topology; it is ample but need not be effective. 
When the quotient 
\[ 
\Eff(\cG|Y):=(\cG|Y)/\Iso(\cG|Y)^\circ 
\] 
is Hausdorff, we write $\cG{\downarrow}Y:=\Eff(\cG|Y)$ 
and call it the effective reduction. 
A closed set $Y$ is $\cG$-\'etale if the full restriction 
$r^{-1}(Y)\cap s^{-1}(Y)$ is \'etale with the relative topology. 
This holds if and only if every arrow of the full restriction is regular, 
in which case $\cG|Y=r^{-1}(Y)\cap s^{-1}(Y)$. 

We use the following conventions for semidirect products and skew products 
(see \cite{Re_text}). 
Let $\Gamma$ be a countable discrete group. 
For an action $\lambda:\Gamma\curvearrowright\cG$ by automorphisms, 
the semidirect product $\cG\rtimes_\lambda\Gamma$ is $\cG\times\Gamma$ 
with the product topology and unit space identified with $\cG^{(0)}$. 
Its range, source and multiplication are given by 
\[ 
r(g,\gamma)=r(g),\qquad s(g,\gamma)=\lambda_\gamma^{-1}(s(g)),\qquad 
(g,\gamma)(h,\delta)=(g\lambda_\gamma(h),\gamma\delta). 
\] 
When $\cG$ is a unit groupoid $X$, this is the transformation groupoid 
$X\rtimes_\lambda\Gamma$. 

For a continuous homomorphism $\xi:\cG\to\Gamma$, 
the skew product $\cG\times_\xi\Gamma$ is $\cG\times\Gamma$ 
with the product topology and unit space $\cG^{(0)}\times\Gamma$, where 
\[ 
r(g,\gamma)=(r(g),\gamma),\qquad 
s(g,\gamma)=(s(g),\gamma\xi(g)),\qquad 
(g,\gamma)(h,\gamma\xi(g))=(gh,\gamma). 
\] 
The action $\delta\cdot(g,\gamma)=(g,\delta\gamma)$ is by automorphisms. 
When $\Gamma$ is finite, the projection $\cG\times_\xi\Gamma\to\cG$ 
is a covering map, and these automorphisms are called deck transformations.

%%%%%%%%%%%%%%%%%%%%%%%%%%%%%%%%%%%%%%%%%%%%%%%%%%%%%%%%%%%%
\subsection{Topological full groups, homology and the index map} 
\label{subsec:homology} 
%%%%%%%%%%%%%%%%%%%%%%%%%%%%%%%%%%%%%%%%%%%%%%%%%%%%%%%%%%%%

For a bisection $U\in\cCO(\cG)$, we write 
$\theta_U:=r\circ(s|U)^{-1}:s(U)\to r(U)$ 
for its associated homeomorphism. 
The bisection $U$ is said to be full when 
$r(U)=\cG^{(0)}=s(U)$; in this case 
$\theta_U\in\Homeo(\cG^{(0)})$. 

\begin{definition}[{\cite[Definition 2.3]{Ma12PLMS}}] 
\label{def:topologicalfullgroup} 
Let $\cG$ be an effective ample groupoid 
whose unit space $\cG^{(0)}$ is compact. 
The set of all $\alpha\in\Homeo(\cG^{(0)})$, 
for which there exists a full bisection $U\in\cCO(\cG)$ 
such that $\alpha=\theta_U$, 
is called the topological full group of $\cG$ 
and is denoted by $\sF(\cG)$. 
\end{definition} 

For $\alpha\in\sF(\cG)$ 
the compact open bisection $U$ as above uniquely exists, 
because $\cG$ is effective. 
Obviously $\sF(\cG)$ is a subgroup of $\Homeo(\cG^{(0)})$. 
Since $\cG$ is second countable, it has countably many compact open subsets, 
and so $\sF(\cG)$ is at most countable. 
A homeomorphism $\alpha:\cG^{(0)}\to\cG^{(0)}$ belongs to $\sF(\cG)$ 
if and only if for any $x\in\cG^{(0)}$ 
there exists a bisection $V\in\cCO(\cG)$ 
such that $x$ is in $s(V)$ and 
$\alpha$ equals $r\circ(s|V)^{-1}$ on a neighborhood of $x$. 
Thus, $\alpha$ is in $\sF(\cG)$ if and only if 
the `graph' of $\alpha$ is a clopen subset of $\cG$. 
We let $\sD(\cG)$ denote the derived subgroup (commutator subgroup) 
of $\sF(\cG)$.

For $\alpha\in\Homeo(\cG^{(0)})$ we write 
\[ 
\supp(\alpha):=\overline{\{x\in\cG^{(0)}\mid\alpha(x)\neq x\}}. 
\] 
When $\alpha=\theta_U$ belongs to $\sF(\cG)$, 
one has $\supp(\alpha)=s(U\setminus\cG^{(0)})$, 
and so $\supp(\alpha)$ is a clopen subset of $\cG^{(0)}$ 
by \cite[Lemma 2.2]{Ma15crelle}. 

\begin{definition}\label{def:localsubgroup} 
Let $\Gamma$ be a subgroup of $\sF(\cG)$ and 
let $V\subset\cG^{(0)}$ be a clopen subset. 
We call 
\[ 
\Gamma_V:=\{\alpha\in\Gamma\mid\supp(\alpha)\subset V\} 
\] 
the local subgroup of $\Gamma$ on $V$. 
\end{definition} 

Extending homeomorphisms by the identity outside $V$, we identify 
$\sF(\cG|V)$ with $\sF(\cG)_V$. 
Under this identification, 
$\sD(\cG|V)\subset\sD(\cG)_V$. 

We note that $\Gamma_{\beta(V)}=\beta\Gamma_V\beta^{-1}$ 
for any $\beta\in\sF(\cG)$ normalizing $\Gamma$ and 
any clopen set $V$, because 
$\supp(\beta\alpha\beta^{-1})=\beta(\supp(\alpha))$. 

For a subgroup $\Gamma\subset\sF(\cG)$ and a subset 
$Y\subset\cG^{(0)}$, we write 
\[ 
\St(\Gamma,Y):=\{\alpha\in\Gamma\mid\alpha(Y)=Y\},\qquad 
\St(\cG,Y):=\St(\sF(\cG),Y). 
\] 

We write $H_n(\cG,A)$ for groupoid homology with coefficients in 
an abelian group $A$, and put $H_n(\cG):=H_n(\cG,\Z)$. 
We refer to \cite[Section 3]{Ma12PLMS} for the definition. 
The group $H_0(\cG)$ is generated by the classes $[1_B]$ 
of compact open subsets $B\subset\cG^{(0)}$, subject to the relations 
\[ 
[1_{B\sqcup C}]=[1_B]+[1_C],\qquad 
[1_{r(U)}]=[1_{s(U)}] 
\] 
for disjoint compact open sets $B,C$ and compact open bisections $U$. 
The same presentation over $\Z/2$ gives $H_0(\cG,\Z/2)$; 
equivalently, $H_0(\cG,\Z/2)\cong H_0(\cG)\otimes\Z/2$. 

A groupoid homomorphism $\rho:\cG\to\cH$ which is a local homeomorphism 
induces covariant homomorphisms 
\[ 
H_n(\rho):H_n(\cG,A)\longrightarrow H_n(\cH,A); 
\] 
see the discussion preceding \cite[Definition 3.4]{Ma12PLMS}. 
In particular, inclusion of an open subgroupoid induces such maps, 
sending each compact open unit-space class $[1_B]$ to $[1_B]$. 
For a wide open inclusion, the map on $H_0$ is therefore surjective, 
both with integer and with $\Z/2$ coefficients. 
Inclusion of a full clopen reduction induces an isomorphism on homology 
by \cite[Proposition 3.5 and Theorem 3.6]{Ma12PLMS}. 
Thus these groups are invariants of Kakutani equivalence. 
For an AF groupoid, $H_0$ is its dimension group; 
see \cite[Theorem 4.10]{Ma12PLMS}. 

For $\alpha\in\sF(\cG)$, let $U$ be its unique compact open full bisection. 
The assignment 
\[ 
I:\sF(\cG)\longrightarrow H_1(\cG),\qquad I(\alpha):=[1_U] 
\] 
is a homomorphism, called the index map 
\cite[Section 7]{Ma12PLMS}. 
It is natural for open subgroupoid inclusions, using extension by 
the identity for the corresponding inclusions of topological full groups. 
In particular, $\sD(\cG)\subset\Ker I$. 
For an abstract group $\Lambda$, we use 
$H_1(\Lambda)=\Lambda/[\Lambda,\Lambda]$ for its abelianization.

%%%%%%%%%%%%%%%%%%%%%%%%%%%%%%%%%%%%%%%%%%%%%%%%%%%%%%%%%%%%
\subsection{Graph groupoids}\label{subsec:graphgroupoid} 
%%%%%%%%%%%%%%%%%%%%%%%%%%%%%%%%%%%%%%%%%%%%%%%%%%%%%%%%%%%%

We recall the construction of graph groupoids and their basic bisections. 

Let $(W,F)$ be a countable directed graph, 
where $W$ is the set of vertices and $F$ is the set of edges, 
and where $i(f)$ and $t(f)$ denote 
the initial and the terminal vertex of $f\in F$. 
For $w\in W$, we write $wF:=\{f\in F\mid i(f)=w\}$. 
Throughout we assume that $(W,F)$ is row-finite and has no sinks, 
that is, 
\[ 
0<\#(wF)<\infty\qquad\text{for every }w\in W. 
\] 
A finite path is a word $\mu=\mu_1\mu_2\cdots\mu_n$ in $F$ 
with $t(\mu_j)=i(\mu_{j+1})$ for all $j$; 
we allow $n=0$ and identify the paths of length $0$ with the vertices. 
We write $\lvert\mu\rvert:=n$ for the length of $\mu$, 
and $i(\mu):=i(\mu_1)$ and $t(\mu):=t(\mu_n)$ 
for its initial and terminal vertices. 
For a vertex $w$, regarded as a path of length zero, 
we set $i(w)=t(w)=w$. 
The concatenation $\mu\nu$ is defined when $t(\mu)=i(\nu)$. 
We say that $\mu$ is a prefix of $\nu$, and write $\mu\preceq\nu$, 
when $\nu=\mu\eta$ for some finite path $\eta$; 
the same notation is used when $\nu$ is an infinite path. 
Two finite paths are said to be incomparable 
when neither is a prefix of the other. 
By a closed path we mean a finite path $\mu$ 
with $\lvert\mu\rvert\geq1$ and $i(\mu)=t(\mu)$. 
The graph is strongly connected if every vertex can reach every 
other vertex by a finite path. 
A vertex is called branching if it has at least two outgoing edges. 

Let $(W',F')$ be another countable row-finite graph without sinks. 
A graph homomorphism $\rho\colon(W,F)\to(W',F')$ consists of 
maps $\rho\colon W\to W'$ and $\rho\colon F\to F'$ 
such that $i(\rho(f))=\rho(i(f))$ and $t(\rho(f))=\rho(t(f))$. 
We say that $\rho$ is right-covering 
if $\rho\colon wF\to\rho(w)F'$ is bijective for every $w\in W$. 

Let 
\[ 
X_{(W,F)}:=\left\{(x_k)_{k\in\N}\in F^\N 
\mid t(x_j)=i(x_{j+1})\quad\forall j\in\N\right\} 
\] 
be the space of infinite paths with the relative product topology 
and let $\sigma$ be the shift on it. 
For a finite path $\mu$, we put 
$Z(\mu):=\{x\in X_{(W,F)}\mid \mu\preceq x\}$; 
in particular $Z(w)=\{x\in X_{(W,F)}\mid i(x_1)=w\}$ for $w\in W$. 
A finite set of pairwise incomparable finite paths is called 
a complete prefix set if every $x\in X_{(W,F)}$ has a prefix in it. 
Since $(W,F)$ is row-finite and has no sinks, 
each $Z(\mu)$ is a nonempty compact open set, 
and these sets form a basis for the topology of $X_{(W,F)}$. 
Hence $X_{(W,F)}$ is a second countable, locally compact, Hausdorff, 
totally disconnected space, 
and it is compact if and only if $W$ is finite. 

The shift $\sigma$ is a local homeomorphism, 
and we let 
\[ 
\cG_{(W,F)}:=\left\{(x,n,y)\in X_{(W,F)}\times\Z\times X_{(W,F)}\mid 
\exists k,l\geq0,\ n=k{-}l,\ \sigma^k(x)=\sigma^l(y)\right\} 
\] 
denote the associated groupoid, 
which we call the graph groupoid of $(W,F)$. 
For finite paths $\mu,\nu$ with $t(\mu)=t(\nu)$, the set 
\[ 
Z(\nu,\mu):=\left\{(\nu z,\,\lvert\nu\rvert-\lvert\mu\rvert,\,\mu z) 
\mid z\in Z(t(\mu))\right\} 
\] 
is a compact open bisection with range $Z(\nu)$ and source $Z(\mu)$, 
it implements the prefix replacement homeomorphism $Z(\mu)\to Z(\nu)$, 
$\mu z\mapsto\nu z$, 
and these sets form a basis for the topology of $\cG_{(W,F)}$. 
Thus $\cG_{(W,F)}$ is a second countable, locally compact, Hausdorff, 
\'etale groupoid whose unit space $X_{(W,F)}$ is totally disconnected. 
We say that $(W,F)$ is cofinal if, for every $w\in W$ and every 
$x\in X_{(W,F)}$, there is a finite path from $w$ to $i(x_n)$ 
for some $n\geq1$. 
The graph satisfies Condition~\textup{(L)} if every closed path 
$\mu=\mu_1\cdots\mu_n$ has an exit, that is, there are $j$ and 
$f\in i(\mu_j)F$ with $f\neq\mu_j$. 
Equivalently, every closed path passes through a vertex 
with at least two outgoing edges. 
Then $\cG_{(W,F)}$ is minimal if and only if $(W,F)$ is cofinal, 
and it is effective if and only if Condition~\textup{(L)} holds; 
see \cite[Propositions 8.2 and 8.3]{NO19IJM}. 
In particular, strong connectivity implies minimality. 

The graph $C^*$-algebra $C^*(W,F)$ is canonically isomorphic to 
$C_r^*(\cG_{(W,F)})$; graph groupoids are amenable, so the full 
and reduced groupoid $C^*$-algebras agree \cite{Pa02JOT}. 
We record the standard criteria and computations used below. 
Let $B=(B(u,v))_{u,v\in W}$ be the adjacency matrix, where 
$B(u,v):=\#\{f\in F\mid i(f)=u,\ t(f)=v\}$. 
Row-finiteness ensures that $I-B^t$ is an endomorphism of the 
direct sum $\Z^{(W)}$ of finitely supported vertex vectors. 

\begin{proposition}\label{prop:graphalgebras} 
Let $(W,F)$ be a nonempty countable row-finite graph without sinks. 
\begin{enumerate} 
\item $C^*(W,F)$ is simple if and only if $(W,F)$ is cofinal 
and satisfies Condition~\textup{(L)}. 
\item $C^*(W,F)$ is AF if and only if $(W,F)$ has no closed paths. 
Consequently, it is simple AF if and only if $(W,F)$ is cofinal 
and has no closed paths. 
\item $C^*(W,F)$ is purely infinite simple if and only if $(W,F)$ 
is cofinal, satisfies Condition~\textup{(L)}, and has a closed path. 
\item There are isomorphisms 
\begin{equation}\label{eq:countablegraphhomology} 
\begin{aligned} 
K_0(C^*(W,F))&\cong H_0(\cG_{(W,F)})\cong\Coker(I-B^t),\\ 
K_1(C^*(W,F))&\cong H_1(\cG_{(W,F)})\cong\Ker(I-B^t). 
\end{aligned} 
\end{equation} 
Moreover, $H_n(\cG_{(W,F)})=0$ for $n\geq2$. 
For $w\in W$, the class of the vertex projection in $K_0$ 
and the class $[1_{Z(w)}]$ in $H_0$ correspond to the class of 
the standard basis vector $\epsilon_w$ in the cokernel. 
More generally, $[1_{Z(\mu)}]=[\epsilon_{t(\mu)}]$. 
If $W$ is finite, the unit class corresponds to 
$[\sum_{w\in W}\epsilon_w]$. 
\end{enumerate} 
\end{proposition} 

\begin{proof} 
(1) is \cite[Proposition 5.1]{BPRS00NYJM}, and 
(2) follows from \cite[Theorem 2.4]{KPR98PJM} and (1). 
For (3), cofinality and the existence of a closed path imply that 
every vertex can reach that closed path. 
Now apply \cite[Proposition 5.3 and Remark 5.6]{BPRS00NYJM}. 
The $K$-theory formulas in (4) are 
\cite[Theorem 3.1]{DT02IJM}; the homology formulas follow from 
\cite[Theorem 4.6]{NO21DM}. 
The assertion about cylinder classes also follows from the 
bisection $Z(\mu,t(\mu))$. 
\end{proof} 

Thus a simple graph $C^*$-algebra is either AF or purely infinite. 
The homology identifications in \eqref{eq:countablegraphhomology} 
are natural for right-covering graph maps: the induced maps on 
the kernels and cokernels are given by summing vertex coordinates 
over the fibers. 
Indeed, for such a map $\rho\colon(W,F)\to(W',F')$, 
the map $\rho_*\colon\Z^{(W)}\to\Z^{(W')}$, 
$\epsilon_w\mapsto\epsilon_{\rho(w)}$, satisfies 
$\rho_*B^t=(B')^t\rho_*$. 
The AF skew-product computation of homology respects these maps; 
see \cite[Theorem 4.14]{Ma12PLMS} and 
\cite[Proposition 7.1 and Remark 7.2]{NO21DM}. 

The following elementary facts will be used without further comment. 

\begin{lemma}\label{lem:basicbisection} 
Let $\mu,\nu,\omega,\lambda$ be finite paths in $(W,F)$. 
\begin{enumerate} 
\item $Z(\mu,\mu)=Z(\mu)$, $Z(\nu,\mu)^{-1}=Z(\mu,\nu)$ and 
$Z(\omega,\nu)Z(\nu,\mu)=Z(\omega,\mu)$. 
\item $Z(\nu,\mu)=\bigsqcup_{f\in t(\mu)F}Z(\nu f,\mu f)$. 
\item If $Z(\nu,\mu)\cap Z(\omega,\lambda)\neq\emptyset$ and 
$\lvert\lambda\rvert\geq\lvert\mu\rvert$, 
then there is a unique finite path $\eta$ 
with $\lambda=\mu\eta$ and $\omega=\nu\eta$; 
in particular $Z(\omega,\lambda)\subset Z(\nu,\mu)$. 
\end{enumerate} 
\end{lemma} 

\begin{proof} 
(1) and (2) are immediate from the definitions. 
For (3), let 
$(\nu z,\lvert\nu\rvert-\lvert\mu\rvert,\mu z) 
=(\omega w,\lvert\omega\rvert-\lvert\lambda\rvert,\lambda w)$. 
From $\mu z=\lambda w$ and $\lvert\lambda\rvert\geq\lvert\mu\rvert$, 
we obtain $\eta$ with $\lambda=\mu\eta$ and $z=\eta w$. 
Then $\nu\eta w=\omega w$, and 
$\lvert\omega\rvert=\lvert\nu\rvert+\lvert\eta\rvert$ 
by comparing the middle coordinates, whence $\omega=\nu\eta$. 
\end{proof} 

When $W$ is finite, the dynamical system 
$(X_{(W,F)},\sigma)$ is a one-sided shift of finite type, 
and $\cG_{(W,F)}$ is called its SFT groupoid. 
The reader is referred to \cite[Section 6]{Ma15crelle} 
or \cite{Ma17Abel} for its basic properties. 
In this case, we let $A=(A(u,v))_{u,v\in W}$ 
denote the adjacency matrix of $(W,F)$, that is, 
\[ 
A(u,v):=\#\{f\in F\mid i(f)=u,\ t(f)=v\}, 
\] 
and we say that $A$ is irreducible if for all $u,v\in W$ 
there exists $n\in\N$ with $A^n(u,v)>0$. 
Thus $(W,F)$ is strongly connected if and only if $A$ is irreducible. 
Suppose that $A$ is irreducible and is not a permutation matrix. 
In this case, 
$X_{(W,F)}$ has no isolated points and hence is a Cantor set. 
We also note that every closed path in $(W,F)$ passes through a vertex $w$ 
with $\#(wF)\geq2$. 
Indeed, if $\mu_1\mu_2\cdots\mu_n$ were a closed path 
with $\#(i(\mu_j)F)=1$ for every $j$, 
then only the vertices $i(\mu_1),\dots,i(\mu_n)$ would be 
reachable from $i(\mu_1)$. 
Irreducibility would therefore force 
$W=\{i(\mu_1),\dots,i(\mu_n)\}$ and 
$F=\{\mu_1,\dots,\mu_n\}$, 
so that $A$ would be a permutation matrix. 
Moreover, $\cG_{(W,F)}$ is a minimal effective ample groupoid 
by \cite[Lemma 6.1]{Ma15crelle}.

%%%%%%%%%%%%%%%%%%%%%%%%%%%%%%%%%%%%%%%%%%%%%%%%%%%%%%%%%%%%
\subsection{AF and purely infinite groupoids} 
%%%%%%%%%%%%%%%%%%%%%%%%%%%%%%%%%%%%%%%%%%%%%%%%%%%%%%%%%%%%

We recall the two classes of groupoids arising from the graph model 
and the properties of their topological full groups used below. 

For an ample groupoid $\cG$ with compact unit space, 
a compact open principal subgroupoid $\cK\subset\cG$ 
with $\cK^{(0)}=\cG^{(0)}$ is called an elementary subgroupoid. 
We say that $\cG$ is AF if it is the union of an increasing sequence 
of elementary subgroupoids. 
Every AF groupoid is principal and is almost finite in the sense of 
\cite[Definition 6.2]{Ma12PLMS}. 

\begin{definition}[{\cite[Definition 4.9]{Ma15crelle}}]\label{pi/def} 
Let $\cG$ be an ample groupoid with $\cG^{(0)}$ compact. 
We say that $\cG$ is purely infinite 
if for every clopen set $A\subset\cG^{(0)}$, 
there exist compact open bisections $U,V\subset\cG$ 
such that $s(U)=s(V)=A$, $r(U)\sqcup r(V)\subset A$. 
\end{definition} 

For the graph groupoids used below, pure infiniteness follows 
from the following cylinder argument. 

\begin{lemma}\label{lem:graphpureinfiniteness} 
Let $(W,F)$ be a countable row-finite strongly connected graph 
with a vertex having at least two outgoing edges. 
For every nonempty compact open $S\subset X_{(W,F)}$, 
the reduction $\cG_{(W,F)}|S$ is minimal and purely infinite, 
and its unit space is a Cantor set. 
\end{lemma} 

\begin{proof} 
Strong connectivity implies cofinality. 
It also implies Condition~\textup{(L)}: a closed path without an exit 
would be the only path leaving its vertices, so no branching vertex 
could be reached from it. 
Hence the graph groupoid is minimal and effective. 
Every vertex $w$ admits two incomparable closed paths 
$\alpha_w,\beta_w$: follow a path from $w$ to a branching vertex, 
choose two distinct outgoing edges, and return to $w$. 
For every finite path $\mu$ with $t(\mu)=w$, the basic bisections 
$Z(\mu\alpha_w,\mu)$ and $Z(\mu\beta_w,\mu)$ have common 
source $Z(\mu)$ and disjoint ranges contained in $Z(\mu)$. 
Partitioning a compact open set into finitely many cylinders 
therefore proves pure infiniteness of the reduction to $S$. 
The same branching argument shows that no cylinder is a singleton, 
so $S$ is a Cantor set. 
Minimality passes to every nonempty open reduction. 
\end{proof} 

In particular, this recovers pure infiniteness of an irreducible 
SFT groupoid, as in \cite[Lemma 6.1]{Ma15crelle}. 

If $\cG$ is minimal and either AF or purely infinite, 
then so is $\cG|W$ for every $W\in\cCO(\cG^{(0)})$. 

\begin{theorem}\label{niceproperties} 
Let $\cG$ be a minimal ample groupoid with Cantor unit space 
which is either AF or purely infinite. 
\begin{enumerate} 
\item The commutator subgroup $\sD(\cG)$ is simple. 
\item The groupoid $\cG$ has cancellation: for nonempty clopen sets 
$A,B\subset\cG^{(0)}$ with $[1_A]=[1_B]$ in $H_0(\cG)$, 
there exists a compact open bisection $U$ with $r(U)=A$ and $s(U)=B$. 
\end{enumerate} 
\end{theorem} 

\begin{proof} 
The first assertion follows from 
\cite[Theorems 4.7 and 4.16]{Ma15crelle}. 
The second follows from \cite[Theorem 6.12]{Ma12PLMS} and 
\cite[Theorem 4.2 (3)]{Ma16Adv}. 
\end{proof} 

For a compact open bisection $U$ with $r(U)\cap s(U)=\emptyset$, 
define the transposition $\tau_U:=\theta_W\in\sF(\cG)$, where 
\[ 
W:=U\sqcup U^{-1}\sqcup 
\left(\cG^{(0)}\setminus(r(U)\cup s(U))\right). 
\] 
Then $\tau_U^2=1$ and $\supp(\tau_U)=r(U)\sqcup s(U)$. 
We use the following part of the AH exact sequence. 

\begin{theorem}[{\cite[Corollary E]{Li25ForumPi}}]\label{Li} 
Let $\cG$ be a minimal ample groupoid with Cantor unit space 
which is either AF or purely infinite. 
Then there is an exact sequence 
\[ 
\xymatrix@M=8pt{ 
H_0(\cG,\Z/2) \ar[r]^-\zeta & 
H_1(\sF(\cG)) \ar[r]^-\eta & H_1(\cG) \ar[r] & 0. 
} 
\] 
The map $\eta$ is induced by the index map, 
and $\zeta$ sends $[1_{r(U)}]$ to the class of $\tau_U$ 
whenever $U$ is a compact open bisection with disjoint source and range. 
\end{theorem} 

This description of $\zeta$, together with naturality of the index map, 
shows that the sequence is natural for wide open subgroupoid inclusions 
between groupoids satisfying the hypotheses. 

The following local properties will be used repeatedly. 
We recall their statements and refer to the first two papers for proofs. 
The three-cycle construction requires only minimality. 

\begin{lemma}[{\cite[Lemma 2.6]{Ma26maximal2}}] 
\label{threecycle} 
Let $\cH$ be a minimal ample groupoid with $\cH^{(0)}$ a Cantor set 
and let $U,V\in\cCO(\cH)$ be bisections with $s(U)=s(V)$ 
such that $s(U)$, $r(U)$ and $r(V)$ are mutually disjoint. 
\begin{enumerate} 
\item The element $\tau_V\tau_U$ equals the commutator $[\tau_U,\tau_V]$ 
and has order three, with 
\[ 
\supp(\tau_V\tau_U)=s(U)\cup r(U)\cup r(V). 
\] 
In particular $\tau_V\tau_U\in\sD(\cH)$. 
\item The element $\tau_V\tau_U$ agrees with $\theta_U$ on $s(U)$. 
\item For any $x\in\cH^{(0)}$ and any $A\in\cCO(\cH^{(0)})$ 
there exists $\sigma\in\sD(\cH)$ such that $\sigma(x)\in A$. 
Consequently, for any $B\in\cCO(\cH^{(0)})$, 
the local subgroup $\sD(\cH)_B$ acts minimally on $B$, 
and $\sD(\cH)$ acts minimally on $\cH^{(0)}$. 
\end{enumerate} 
\end{lemma} 

\begin{lemma}[{\cite[Lemma 2.9]{Ma26maximal1}}] 
\label{localabelianization} 
Let $\cG$ be a minimal ample groupoid with Cantor unit space 
which is either AF or purely infinite. 
For every nonempty clopen set $W\subset\cG^{(0)}$, one has 
\[ 
\sF(\cG)=\sF(\cG)_W\cdot\sD(\cG). 
\] 
Equivalently, the inclusion $\sF(\cG|W)\to\sF(\cG)$ induces 
a surjection on abelianizations. 
\end{lemma} 

\begin{lemma}[{\cite[Lemma 2.10 (2)]{Ma26maximal1}}] 
\label{elementofDG} 
Let $\cG$ be a minimal ample groupoid with Cantor unit space 
which is either AF or purely infinite. 
If a bisection $U\in\cCO(\cG)$ satisfies 
$r(U)\neq\cG^{(0)}$ and $s(U)\neq\cG^{(0)}$, then 
there exists a full bisection $V$ such that 
$U\subset V$ and $\theta_V\in\sD(\cG)$. 
\end{lemma}

%%%%%%%%%%%%%%%%%%%%%%%%%%%%%%%%%%%%%%%%%%%%%%%%%%%%%%%%%%%%
\section{Subgroupoid graphs of SFT groupoids}\label{sec:subgroupoidgraph} 
%%%%%%%%%%%%%%%%%%%%%%%%%%%%%%%%%%%%%%%%%%%%%%%%%%%%%%%%%%%%

Throughout this section, 
let $(V,E)$ be a finite directed graph with adjacency matrix $A$, 
and assume that $A$ is irreducible and is not a permutation matrix. 
We write 
\[ 
\cH:=\cG_{(V,E)},\qquad X:=\cH^{(0)}=X_{(V,E)}, 
\] 
and let $\Omega=\Omega(V,E)$ denote the set of all finite paths in $(V,E)$. 
We freely use the notation introduced in Section~\ref{subsec:graphgroupoid}. 

The subgroupoid graph records which prefix replacements belong to 
a wide open subgroupoid of $\cH$. 
We first characterize these graphs and describe intermediate 
subgroupoids by quotient maps. 
We then identify each subgroupoid with a full compact open reduction 
of a countable graph groupoid and use prefix replacements to describe 
its topological full group and normalizer. 

The following notion of a prime inclusion will be used in 
Corollary~\ref{cor:primegraph} and subsequent sections. 

\begin{definition}[{\cite[Definition 5.1]{Ma26maximal1}}] 
\label{def:prime} 
Let $\cH$ be an ample groupoid and 
let $\cG\subset\cH$ be a wide open subgroupoid. 
We say that $\cG\subset\cH$ is prime 
if there is no proper intermediate wide open subgroupoid; 
more precisely, if $\cK\subset\cH$ is a wide open subgroupoid 
containing $\cG$, then $\cK$ is equal to either $\cH$ or $\cG$. 
\end{definition}

%%%%%%%%%%%%%%%%%%%%%%%%%%%%%%%%%%%%%%%%%%%%%%%%%%%%%%%%%%%%
\subsection{The subgroupoid graph}\label{subsec:subgroupoidgraph} 
%%%%%%%%%%%%%%%%%%%%%%%%%%%%%%%%%%%%%%%%%%%%%%%%%%%%%%%%%%%%

We associate a rooted graph to a wide open subgroupoid of an SFT groupoid 
and describe how inclusions of subgroupoids correspond to graph quotients. 

Let $\cG\subset\cH$ be a wide open subgroupoid. 
For $\mu,\nu\in\Omega$, write 
\[ 
\mu\sim_{\cG}\nu 
\quad:\iff\quad 
t(\mu)=t(\nu)\ \text{ and }\ Z(\nu,\mu)\subset\cG. 
\] 
By Lemma~\ref{lem:basicbisection} (1), this is an equivalence relation. 
We write $[\mu]_{\cG}$ for the equivalence class of $\mu$. 
The following two properties will be used repeatedly. 

\begin{lemma}\label{lem:subgroupoidrelation} 
Let $\mu,\nu\in\Omega$ with $t(\mu)=t(\nu)$. 
\begin{enumerate} 
\item If $\mu\sim_{\cG}\nu$, then 
$\mu e\sim_{\cG}\nu e$ for every $e\in t(\mu)E$. 
\item If $\mu e\sim_{\cG}\nu e$ for every $e\in t(\mu)E$, then 
$\mu\sim_{\cG}\nu$. 
\end{enumerate} 
\end{lemma} 

\begin{proof} 
The first assertion follows from $Z(\nu e,\mu e)\subset Z(\nu,\mu)$. 
For the second, Lemma~\ref{lem:basicbisection} (2) gives 
\[ 
Z(\nu,\mu)=\bigsqcup_{e\in t(\mu)E}Z(\nu e,\mu e)\subset\cG. 
\] 
\end{proof} 

Let $\Delta_{\cG}^0:=\Omega/{\sim_{\cG}}$. 
There is a well-defined map 
\[ 
\rho_{\cG}:\Delta_{\cG}^0\longrightarrow V, 
\qquad 
\rho_{\cG}([\mu]_{\cG})=t(\mu). 
\] 
By Lemma~\ref{lem:subgroupoidrelation} (1), 
if $Q=[\mu]_{\cG}\in\Delta_{\cG}^0$ and $e\in\rho_{\cG}(Q)E$, 
then $Qe:=[\mu e]_{\cG}$ is independent of the choice of $\mu$. 

\begin{definition}\label{def:subgroupoidgraph} 
The subgroupoid graph of $\cG\subset\cH$ is the directed graph 
\[ 
\Delta_{\cG}=(\Delta_{\cG}^0,\Delta_{\cG}^1), 
\qquad 
\Delta_{\cG}^1:= 
\{(Q,e)\mid Q\in\Delta_{\cG}^0,\ e\in\rho_{\cG}(Q)E\}, 
\] 
where $(Q,e)$ goes from $Q$ to $Qe$. 
We also write 
\[ 
\rho_{\cG}:\Delta_{\cG}\longrightarrow(V,E) 
\] 
for the graph homomorphism given by $\rho_{\cG}(Q,e)=e$ on edges, 
and define the root section by 
\[ 
\varsigma_{\cG}:V\longrightarrow\Delta_{\cG}^0, 
\qquad 
\varsigma_{\cG}(v)=[v]_{\cG}. 
\] 
\end{definition} 

By construction, 
$\rho_{\cG}$ is right-covering and $\rho_{\cG}\circ\varsigma_{\cG}=\id_V$. 
Every vertex of $\Delta_{\cG}$ is reachable from $\varsigma_{\cG}(V)$: 
indeed, the terminal vertex of the lift of $\mu\in\Omega$ 
starting at $\varsigma_{\cG}(i(\mu))$ is precisely $[\mu]_{\cG}$. 
Moreover Lemma~\ref{lem:subgroupoidrelation} (2) says that 
\begin{equation}\label{eq:subgroupoidgraphreduced} 
Q\neq Q',\ \rho_{\cG}(Q)=\rho_{\cG}(Q') 
\implies Qe\neq Q'e\ \text{ for some }e\in\rho_{\cG}(Q)E. 
\end{equation} 
We refer to \eqref{eq:subgroupoidgraphreduced} 
as the reducedness condition. 
For the full two-shift, $\sim_{\cG}$ and $\Delta_{\cG}$ are the type 
system and type graph of \cite{BBQS25TAMS}; 
see Remark~\ref{rem:subgroupoidgraphBBQS}. 

The graph $\Delta_{\cG}$ is countable, row-finite and has no sinks. 
Indeed, its vertex set is a quotient of the countable set $\Omega$, 
and $\#(Q\Delta_{\cG}^1)=\#(\rho_{\cG}(Q)E)$ for every $Q\in\Delta_{\cG}^0$. 
Furthermore, $\rho_{\cG}$ carries closed paths in $\Delta_{\cG}$ 
to closed paths in $(V,E)$. 
Since $(V,E)$ is finite and irreducible and 
its adjacency matrix is not a permutation matrix, 
every closed path in $(V,E)$ passes through a vertex $v$ with $\#(vE)\geq2$; 
see Section~\ref{subsec:graphgroupoid}. 
If $Q$ is a vertex of $\Delta_{\cG}$ lying over such a vertex $v$, 
then the right-covering property gives $\#(Q\Delta_{\cG}^1)=\#(vE)\geq2$. 
It follows that every closed path 
in the generally infinite graph $\Delta_{\cG}$ 
passes through a vertex of outdegree at least two. 
Thus the graph groupoid $\cG_{\Delta_{\cG}}$ is 
an effective ample groupoid. 

\begin{example}\label{ex:AFcoresubgroupoidgraph} 
Let $c:\cH\to\Z$ be the canonical cocycle $c(x,k,y)=k$, 
and let $\cG:=\Ker c$ be the canonical AF core. 
Indeed, $\cG$ is the increasing union of the compact open principal 
subgroupoids 
\[ 
\{(x,0,y)\in\cH\mid\sigma^n(x)=\sigma^n(y)\},\qquad n\geq0. 
\] 
For finite paths $\mu,\nu\in\Omega$, 
$\mu\sim_{\cG}\nu$ if and only if 
$\lvert\mu\rvert=\lvert\nu\rvert$ and $t(\mu)=t(\nu)$. 
Consequently $[\mu]_{\cG}\mapsto(\lvert\mu\rvert,t(\mu))$ 
identifies the vertex set of $\Delta_{\cG}$ with 
$\Z_{\geq0}\times V$, and the edges are 
\[ 
(n,e):(n,i(e))\longrightarrow(n{+}1,t(e)), 
\qquad n\geq0,\ e\in E. 
\] 
Here $\rho_{\cG}(n,v)=v$ and $\varsigma_{\cG}(v)=(0,v)$. 
Thus $\Delta_{\cG}$ is the stationary Bratteli diagram obtained 
by repeating $(V,E)$ at successive levels, with one root for 
each $v\in V$. 
For the full $d$-shift, it has one vertex at each level and 
$d$ edges between successive levels. 
\end{example}

We next characterize the graphs which arise in this way. 

\begin{proposition}\label{prop:subgroupoidgraphconverse} 
Let $(W,F)$ be a countable directed graph, 
let $\rho:(W,F)\to(V,E)$ be a right-covering graph homomorphism 
and let $\varsigma:V\to W$ satisfy $\rho\circ\varsigma=\id_V$. 
Assume that every vertex of $W$ is reachable from $\varsigma(V)$ 
and that 
\begin{equation}\label{eq:reducedrootedgraph} 
w\neq w',\ \rho(w)=\rho(w') 
\implies we\neq w'e\ \text{ for some }e\in\rho(w)E, 
\end{equation} 
where $we$ in the right-hand side denotes 
the terminal vertex of the unique edge with initial vertex $w$ 
which is mapped to $e$ by $\rho$. 
For $\mu\in\Omega$, let $\ell(\mu)\in W$ be the terminal vertex of 
the lift of $\mu$ starting at $\varsigma(i(\mu))$, and put 
\[ 
\cG(W,F,\rho,\varsigma) 
:=\bigcup\left\{Z(\nu,\mu)\mid 
\ell(\mu)=\ell(\nu)\right\}. 
\] 
Then $\cG(W,F,\rho,\varsigma)$ is a wide open subgroupoid of $\cH$, 
and its subgroupoid graph is canonically isomorphic to $(W,F)$ 
over $(V,E)$, with the root section corresponding to $\varsigma$. 
\end{proposition} 

\begin{proof} 
Define an equivalence relation on $\Omega$ by 
\[ 
\mu\sim\nu\quad:\iff\quad 
\ell(\mu)=\ell(\nu). 
\] 
If $\mu\sim\nu$, then necessarily $t(\mu)=t(\nu)$, 
and the right-covering property implies $\mu e\sim\nu e$ 
for every $e\in t(\mu)E$. 
Conversely, if $\mu e\sim\nu e$ for every $e\in t(\mu)E$, then 
\eqref{eq:reducedrootedgraph} implies $\ell(\mu)=\ell(\nu)$. 

Let $\cK:=\cG(W,F,\rho,\varsigma)$. 
It is open and contains $\cH^{(0)}$. 
It is closed under inverses by symmetry of $\sim$. 
To see that it is closed under products, 
let $g\in Z(\nu,\mu)$ and $h\in Z(\omega,\lambda)$ be composable, 
where $\mu\sim\nu$ and $\lambda\sim\omega$. 
By refining the two basic bisections around $g$ and $h$ 
using Lemma~\ref{lem:basicbisection} (2), (3), 
we may assume that $\mu=\omega$. 
The compatibility of $\sim$ with right extensions then 
gives $\lambda\sim\omega=\mu\sim\nu$, 
and hence $gh\in Z(\nu,\lambda)\subset\cK$. 
Thus $\cK$ is a wide open subgroupoid. 

It remains to identify its subgroupoid graph. 
If $Z(\nu,\mu)\subset\cK$, 
compactness and Lemma~\ref{lem:basicbisection} (2), (3) show that 
there exists $N\geq0$ such that $\ell(\mu\eta)=\ell(\nu\eta)$ 
for every path $\eta$ of length $N$ with $i(\eta)=t(\mu)$. 
Applying \eqref{eq:reducedrootedgraph} successively $N$ times yields 
$\ell(\mu)=\ell(\nu)$. 
Consequently the equivalence relation obtained from $\cK$ is exactly $\sim$. 
Since every vertex of $W$ is reachable from $\varsigma(V)$, 
the map $[\mu]_{\cK}\mapsto\ell(\mu)$ is a bijection on vertices and 
extends to the required graph isomorphism. 
\end{proof} 

It is useful to record 
how inclusions of subgroupoids appear on the graph side. 

\begin{proposition}\label{prop:subgroupoidgraphquotient} 
Let $\cG\subset\cH$ be a wide open subgroupoid. 
For every intermediate wide open subgroupoid $\cG\subset\cK\subset\cH$, 
the map 
\[ 
q_{\cG,\cK}:\Delta_{\cG}\longrightarrow\Delta_{\cK}, 
\qquad 
q_{\cG,\cK}([\mu]_{\cG})=[\mu]_{\cK}, 
\] 
is a surjective right-covering graph homomorphism satisfying 
\[ 
\rho_{\cK}\circ q_{\cG,\cK}=\rho_{\cG}, 
\qquad 
q_{\cG,\cK}\circ\varsigma_{\cG}=\varsigma_{\cK}. 
\] 

Conversely, suppose that 
$q:\Delta_{\cG}\to(W,F)$ is a surjective graph homomorphism and 
that there are maps 
\[ 
\rho:(W,F)\longrightarrow(V,E), 
\qquad 
\varsigma:V\longrightarrow W 
\] 
such that $(W,F,\rho,\varsigma)$ satisfies the hypotheses of 
Proposition~\ref{prop:subgroupoidgraphconverse} and 
\[ 
\rho\circ q=\rho_{\cG}, 
\qquad 
q\circ\varsigma_{\cG}=\varsigma. 
\] 
Then 
$\cK:=\cG(W,F,\rho,\varsigma)$ is an intermediate wide open subgroupoid, 
its subgroupoid graph is canonically isomorphic to $(W,F,\rho,\varsigma)$, 
and, under this identification, $q_{\cG,\cK}$ is equal to $q$. 

Consequently, intermediate wide open subgroupoids are in bijection 
with equivalence classes of such quotient maps. 
Here two quotient maps are called equivalent if an isomorphism of 
rooted graphs over $(V,E)$ between their targets intertwines the maps. 
\end{proposition} 

\begin{proof} 
The first assertion is immediate 
from $\mu\sim_{\cG}\nu\implies\mu\sim_{\cK}\nu$. 

Conversely, let $q:\Delta_{\cG}\to(W,F)$ be as in the statement 
and put $\cK:=\cG(W,F,\rho,\varsigma)$. 
If $Z(\nu,\mu)\subset\cG$, then $[\mu]_{\cG}=[\nu]_{\cG}$, 
and hence $q([\mu]_{\cG})=q([\nu]_{\cG})$. 
The latter two vertices are precisely the terminal vertices 
of the lifts of $\mu$ and $\nu$ to $(W,F)$ 
starting at the corresponding roots. 
Thus $Z(\nu,\mu)\subset\cK$, proving $\cG\subset\cK$. 
Proposition~\ref{prop:subgroupoidgraphconverse} identifies 
$\Delta_{\cK}$ canonically with $(W,F)$, 
and the same description of terminal vertices shows that 
this identification carries $q_{\cG,\cK}$ to $q$. 

Finally, reconstructing the subgroupoid from $q_{\cG,\cK}$ gives 
\[ 
\bigcup\{Z(\nu,\mu)\mid [\mu]_{\cK}=[\nu]_{\cK}\}=\cK, 
\] 
since these basic bisections form a basis for the topology of $\cK$. 
Together with the canonical identification of $q$ with $q_{\cG,\cK}$ 
above, this shows that the two constructions are inverse 
up to the stated equivalence. 
\end{proof} 

The quotient corresponding to the ambient groupoid $\cH$ is 
$\rho_{\cG}:\Delta_{\cG}\to(V,E)$. 
Proposition~\ref{prop:subgroupoidgraphquotient} therefore gives 
the following graph-theoretic formulation of primeness. 

\begin{corollary}\label{cor:primegraph} 
A proper wide open subgroupoid $\cG\subset\cH$ is prime 
in the sense of Definition~\ref{def:prime} 
if and only if every quotient map of 
$(\Delta_{\cG},\rho_{\cG},\varsigma_{\cG})$ 
in the sense of Proposition~\ref{prop:subgroupoidgraphquotient} 
is equivalent either to 
$\id_{\Delta_{\cG}}:\Delta_{\cG}\to\Delta_{\cG}$ 
or to $\rho_{\cG}:\Delta_{\cG}\to(V,E)$. 
\end{corollary}

%%%%%%%%%%%%%%%%%%%%%%%%%%%%%%%%%%%%%%%%%%%%%%%%%%%%%%%%%%%%
\subsection{The graph-groupoid model} 
%%%%%%%%%%%%%%%%%%%%%%%%%%%%%%%%%%%%%%%%%%%%%%%%%%%%%%%%%%%%

The subgroupoid graph gives a concrete model for $\cG$. 
This elementary observation is of independent interest: 
every wide open subgroupoid of an irreducible SFT groupoid is 
isomorphic to a full compact open reduction of a countable graph 
groupoid, without any minimality or primeness assumption. 
For the full two-shift, the finite nuclear case is closely related 
to the topological full group model in \cite[Theorem 6.14]{BBQS25TAMS}. 
A related graph model for AF groupoids is given in 
\cite[Proposition 11.25]{NO19IJM}. 

\begin{proposition}\label{prop:subgroupoidgraphmodel} 
Let $\cG\subset\cH$ be a wide open subgroupoid and 
put $\Delta:=\Delta_{\cG}$. 
Define 
\[ 
S:=\bigsqcup_{v\in V}Z(\varsigma_{\cG}(v)) 
\subset X_{\Delta}. 
\] 
Then $S$ is a full compact open subset of $X_{\Delta}$, 
and $\rho_{\cG}$ induces an isomorphism 
$\cG_{\Delta}|S\longrightarrow\cG$, 
where $\cG_{\Delta}$ is the graph groupoid of $\Delta$. 
In particular, $\cG$ is Kakutani equivalent to $\cG_{\Delta}$. 
\end{proposition} 

\begin{proof} 
As $V$ is finite, $S$ is compact open. 
By a slight abuse of notation, 
the maps induced by $\rho_{\cG}$ on infinite path spaces and 
graph groupoids are also denoted by $\rho_{\cG}$. 

Since $\rho_{\cG}$ is right-covering, 
every $x\in X_{(V,E)}$ admits a unique lift $\widehat{x}\in X_{\Delta}$ 
with $i(\widehat{x}_1)=\varsigma_{\cG}(i(x_1))$, 
and this lift belongs to $S$. 
Hence $\rho_{\cG}|S:S\to X_{(V,E)}$ is a homeomorphism. 
To see that $S$ is full, 
let $\tilde x\in X_{\Delta}$ and put $Q:=i(\tilde x_1)$. 
Since every vertex of $\Delta$ is reachable from $\varsigma_{\cG}(V)$, 
there exists a finite path $\tilde\mu$ 
from a vertex in $\varsigma_{\cG}(V)$ to $Q$. 
Then $\tilde\mu\tilde x\in S$ 
and $(\tilde\mu\tilde x,\lvert\tilde\mu\rvert,\tilde x) \in\cG_{\Delta}$, 
so every $\cG_{\Delta}$-orbit meets $S$. 

Let $(\tilde x,n,\tilde y)\in\cG_{\Delta}|S$ and write 
$\tilde x=\tilde\nu\tilde z$, $\tilde y=\tilde\mu\tilde z$, 
$n=\lvert\tilde\nu\rvert-\lvert\tilde\mu\rvert$. 
Put $\nu:=\rho_{\cG}(\tilde\nu)$ and $\mu:=\rho_{\cG}(\tilde\mu)$. 
Since $\tilde x,\tilde y\in S$, 
the paths $\tilde\nu$ and $\tilde\mu$ are the lifts of $\nu$ and $\mu$ 
starting at the corresponding roots. 
Their terminal vertices both equal $i(\tilde z_1)$, 
and hence $[\mu]_{\cG}=[\nu]_{\cG}$. 
Therefore $Z(\nu,\mu)\subset\cG$, 
and the image of $(\tilde x,n,\tilde y)$ belongs to $\cG$. 

Conversely, every element of $\cG$ lies in some $Z(\nu,\mu)\subset\cG$. 
Then $[\mu]_{\cG}=[\nu]_{\cG}$, 
so the lifts of $\mu$ and $\nu$ starting at the corresponding roots have 
the same terminal vertex and give a preimage in $\cG_{\Delta}|S$. 
Uniqueness follows from the injectivity of $\rho_{\cG}|S$. 
The resulting groupoid homomorphism carries 
basic compact open bisections onto basic compact open bisections and 
is therefore an isomorphism. 
\end{proof} 

The graph model also gives a dichotomy at the groupoid level, 
corresponding to the familiar dichotomy for simple graph 
$C^*$-algebras in Proposition~\ref{prop:graphalgebras}. 

\begin{corollary}\label{cor:minimalsubgroupoiddichotomy} 
Let $\cG\subset\cH$ be a minimal wide open subgroupoid. 
Then $\cG$ is AF if $\Delta_{\cG}$ has no closed paths, 
and is purely infinite otherwise. 
\end{corollary} 

\begin{proof} 
Write $\cG\cong\cG_{\Delta}|S$ as in 
Proposition~\ref{prop:subgroupoidgraphmodel}. 
Since $S$ is full and $\cG$ is minimal and effective, 
$\Delta$ is cofinal and satisfies Condition~\textup{(L)}. 

If $\Delta$ has no closed paths, exhaust its edges by finite sets. 
For each such set, take $S$ together with all $Z(\nu,\mu)$ 
where $\mu,\nu$ start at roots, end at the same vertex, and 
use only edges in that set. 
Acyclicity makes these finite unions of compact open bisections 
with trivial isotropy. Prefix cancellation shows that they are subgroupoids, 
and they increase to $\cG_{\Delta}|S$. Thus $\cG$ is AF. 

Otherwise, let $C$ be the strongly connected component containing 
a closed path. Cofinality implies that every vertex can reach $C$, 
by applying it to the infinite repetition of that path. 
Every vertex reachable from $C$ can return to $C$ and thus belongs 
to $C$; hence no edge leaves $C$. 
Applying cofinality from a vertex of $C$ to any infinite path 
shows that the path eventually enters $C$. 
Condition~\textup{(L)} ensures that $\Delta|C$ has a branching vertex. 
Every compact open subset of $S$ can therefore be partitioned 
into finitely many cylinders whose terminal vertices lie in $C$. 
The cylinder argument of Lemma~\ref{lem:graphpureinfiniteness} 
then proves pure infiniteness of $\cG_{\Delta}|S$. 
\end{proof} 

In Section~\ref{sec:minimal}, this dichotomy will allow us to apply 
the general results for AF and purely infinite groupoids to minimal 
wide open subgroupoids of $\cH$.

\begin{example}\label{ex:cocyclesubgroupoidgraph} 
Let $m\geq2$ and assign a label $\ell(e)\in\Z/m$ to each $e\in E$. 
For a finite path $\mu$, put $\ell(\mu):=\sum_j\ell(\mu_j)$, 
with $\ell(v)=0$ for $v\in V$. 
Define a continuous homomorphism $\xi:\cH\to\Z/m$ by 
\[ 
\xi(x,k-l,y):=\ell(x_1\cdots x_k)-\ell(y_1\cdots y_l) 
\quad\text{when }\sigma^k(x)=\sigma^l(y). 
\] 
The value is unchanged when $k$ and $l$ are increased equally. 
Put $\cK:=\Ker\xi$ for this example. 
For $\mu,\nu\in\Omega$ with $t(\mu)=t(\nu)$, 
$\xi$ is constant on $Z(\nu,\mu)$, with value 
$\ell(\nu)-\ell(\mu)$. Hence 
\[ 
\mu\sim_{\cK}\nu 
\iff t(\mu)=t(\nu)\ \text{ and }\ 
\ell(\mu)=\ell(\nu). 
\] 
Define a graph $(W,F)$ by 
\[ 
W:=V\times\Z/m,\qquad F:=E\times\Z/m, 
\qquad 
i(e,b):=(i(e),b),\quad t(e,b):=(t(e),b+\ell(e)). 
\] 
Then $\Delta_{\cK}$ is the subgraph consisting of the vertices 
reachable from $V\times\{0\}$ and their outgoing edges, 
with $\varsigma_{\cK}(v)=(v,0)$ and with $(e,b)$ labelled by $e$. 
In particular, it is the whole graph $(W,F)$ when that graph 
is strongly connected. 

Sending a path in $(W,F)$ to its projected path and its initial 
layer identifies its path space with $X_{(V,E)}\times\Z/m$ 
and its groupoid with $\cH\times_\xi\Z/m$. 
Under this identification, the root cylinders correspond to 
$X_{(V,E)}\times\{0\}$. 
Proposition~\ref{prop:subgroupoidgraphmodel} therefore gives 
\[ 
\Ker\xi\cong 
(\cH\times_\xi\Z/m)|\left(X_{(V,E)}\times\{0\}\right). 
\] 
For $m=2$ and $E=E_0\sqcup E_1$, take $\ell(e)=i$ for $e\in E_i$. 
This is the cocycle used for the two-to-one covers in 
\cite[Section 4.3]{Ma26maximal1}. 
The factor map $\cH\times_\xi\Z/m\to\cH$ and the wide open 
inclusion $\Ker\xi\subset\cH$ thus arise from the same cocycle; 
see \cite[Section 2.5]{Ma26maximal1} for factor maps and their pullbacks. 
Their associated subgroups are different: the factor pullback 
acts on the covering space, whereas $\sF(\Ker\xi)$ acts on 
$X_{(V,E)}$. 
The inclusion of the covering groupoid into its semidirect product 
by the deck group is prime when $m$ is prime and the covering 
groupoid is minimal; 
see Example~\ref{ex:cycliccovers}. 
\end{example}

%%%%%%%%%%%%%%%%%%%%%%%%%%%%%%%%%%%%%%%%%%%%%%%%%%%%%%%%%%%%
\subsection{Topological full groups and normalizers} 
%%%%%%%%%%%%%%%%%%%%%%%%%%%%%%%%%%%%%%%%%%%%%%%%%%%%%%%%%%%%

We use the subgroupoid graph to describe the topological full group 
of a wide open subgroupoid and its normalizer in the ambient 
topological full group. 

For groups $K\subset L$, we use the notation 
\[ 
\Nor(L,K):=\{h\in L\mid hKh^{-1}=K\} 
\] 
for the usual group normalizer, as in \cite{Ma26maximal1}. 
For a wide open subgroupoid $\cG\subset\cH$, define 
\[ 
\cN_{\cH}(\cG):= 
\{\alpha\in\sF(\cH)\mid 
U_\alpha\cG U_\alpha^{-1}=\cG\}, 
\] 
where $U_\alpha\in\cCO(\cH)$ is the unique full bisection 
satisfying $\alpha=\theta_{U_\alpha}$. 
It is immediate that $\cN_{\cH}(\cG)$ is a subgroup of $\sF(\cH)$, 
and $\sF(\cG)$ is a normal subgroup of $\cN_{\cH}(\cG)$. 

The two normalizers agree whenever the subgroupoid has sufficiently 
large orbits; in particular, they agree when $\cG$ is minimal. 

\begin{lemma}\label{lem:groupnormalizers} 
Let $\cG\subset\cH$ be a wide open subgroupoid such that every 
$\cG$-orbit contains at least three points. Then 
\[ 
\cN_{\cH}(\cG) 
=\Nor\left(\sF(\cH),\sF(\cG)\right) 
=\Nor\left(\sF(\cH),\sD(\cG)\right). 
\] 
\end{lemma} 

\begin{proof} 
By \cite[Lemma 3.7]{Ma15crelle}, the full bisections associated 
with $\sD(\cG)$ cover $\cG$; see also 
\cite[Lemma 4.9 (2)]{NO19IJM}. 
Thus normalizing $\sD(\cG)$ implies normalizing $\cG$. 
The reverse inclusions follow because $\sD(\cG)$ is 
characteristic in $\sF(\cG)$. 
\end{proof} 

We first record the partial action of $\sF(\cH)$ on finite paths. 
For $\alpha\in\sF(\cH)$, set 
\[ 
\Omega(\alpha):= 
\{\mu\in\Omega\mid 
Z(\nu,\mu)\subset U_\alpha\ \text{for some }\nu\in\Omega\}. 
\] 

\begin{lemma}\label{lem:subgroupoidpartialaction} 
Let $\alpha,\beta\in\sF(\cH)$. 
\begin{enumerate} 
\item For $\mu\in\Omega(\alpha)$, the path $\nu$ above is unique; 
we write it as $\alpha(\mu)$. 
One has $t(\alpha(\mu))=t(\mu)$ and $\alpha(Z(\mu))=Z(\alpha(\mu))$. 
\item $\Omega\setminus\Omega(\alpha)$ is finite, 
and $\mu\in\Omega(\alpha)$ implies $\mu\eta\in\Omega(\alpha)$ and 
$\alpha(\mu\eta)=\alpha(\mu)\eta$ for every finite path $\eta$ 
which can be concatenated to $\mu$. 
\item The map $\mu\mapsto\alpha(\mu)$ is a bijection 
from $\Omega(\alpha)$ onto $\Omega(\alpha^{-1})$. 
\item If $\mu\in\Omega(\alpha)$ and $\alpha(\mu)\in\Omega(\beta)$, 
then $\mu\in\Omega(\beta\alpha)$ and 
$(\beta\alpha)(\mu)=\beta(\alpha(\mu))$. 
\end{enumerate} 
\end{lemma} 

\begin{proof} 
For (1), any two candidate bisections have the same source and 
are contained in $U_\alpha$, so they coincide. 
Lemma~\ref{lem:basicbisection} (3) then gives uniqueness of $\nu$. 
For (2), write 
\[ 
U_\alpha=\bigsqcup_{k=1}^n Z(\nu_k,\mu_k), 
\] 
where $\{\mu_k\}$ and $\{\nu_k\}$ are complete prefix sets. 
Then $\mu_k\eta\in\Omega(\alpha)$ and $\alpha(\mu_k\eta)=\nu_k\eta$ 
by Lemma~\ref{lem:basicbisection} (2). 
Since $\{\mu_k\}_{k=1}^n$ is a complete prefix set, 
every finite path $\mu$ which is not a proper prefix of any $\mu_k$ 
extends some $\mu_k$; 
thus $\mu=\mu_k\eta$ for some $k$ and some finite path $\eta$. 
Hence $\Omega\setminus\Omega(\alpha)$ is contained 
in the finite set of proper prefixes of the $\mu_k$'s. 
Assertions (3) and (4) follow from Lemma~\ref{lem:basicbisection} (1). 
\end{proof}

The subgroupoid relation $\sim_{\cG}$ detects 
both the topological full group of $\cG$ and its normalizer in $\sF(\cH)$. 

\begin{proposition}\label{prop:subgroupoidnormalizer} 
Let $\cG\subset\cH$ be a wide open subgroupoid and let $\alpha\in\sF(\cH)$. 
\begin{enumerate} 
\item The following are equivalent: 
\begin{enumerate} 
\item $\alpha\in\sF(\cG)$; 
\item $\alpha(\mu)\sim_{\cG}\mu$ for every $\mu\in\Omega(\alpha)$; 
\item $\alpha(\mu)\sim_{\cG}\mu$ for all but finitely many $\mu\in\Omega$ 
for which $\alpha(\mu)$ is defined. 
\end{enumerate} 
\item The following are equivalent: 
\begin{enumerate} 
\item $\alpha\in\cN_{\cH}(\cG)$; 
\item there exists a cofinite subset 
$\Omega_0\subset\Omega(\alpha)$, closed under right extensions, 
such that for all $\mu,\nu\in\Omega_0$, 
\[ 
\mu\sim_{\cG}\nu 
\iff \alpha(\mu)\sim_{\cG}\alpha(\nu). 
\] 
\end{enumerate} 
\end{enumerate} 
\end{proposition} 

\begin{proof} 
(1)\: 
If $\alpha\in\sF(\cG)$, then $U_\alpha\subset\cG$, 
and hence $Z(\alpha(\mu),\mu)\subset\cG$ for every $\mu\in\Omega(\alpha)$. 
Thus (i) implies (ii), and (ii) implies (iii). 
Suppose (iii) holds. 
Write $U_\alpha=\bigsqcup_k Z(\nu_k,\mu_k)$ 
as in the proof of Lemma~\ref{lem:subgroupoidpartialaction}. 
Refining the partition by right extensions, 
we may assume that every $\mu_k$ belongs to the cofinite set 
on which (iii) holds. 
Then $Z(\nu_k,\mu_k)\subset\cG$ for all $k$, and hence $U_\alpha\subset\cG$. 
Thus $\alpha\in\sF(\cG)$. 

(2)\: 
Suppose first that (ii) holds. Let $g\in\cG$. 
By openness, 
after refining a basic neighborhood of $g$ and using the fact that 
$\Omega_0$ is cofinite and closed under right extensions, 
we can find $\mu,\nu\in\Omega_0$ such that $g\in Z(\nu,\mu)\subset\cG$. 
Then 
\[ 
U_\alpha Z(\nu,\mu)U_\alpha^{-1} 
=Z(\alpha(\nu),\nu)Z(\nu,\mu)Z(\mu,\alpha(\mu)) 
=Z(\alpha(\nu),\alpha(\mu))\subset\cG. 
\] 
Hence $U_\alpha\cG U_\alpha^{-1}\subset\cG$. 
By Lemma~\ref{lem:subgroupoidpartialaction} (2), (3), the set 
$\Omega_1:=\alpha(\Omega_0)\subset\Omega(\alpha^{-1})$ 
is cofinite in $\Omega$ and closed under right extensions. 
The equivalence in (ii) gives the same condition for $\alpha^{-1}$ 
on $\Omega_1$. Applying the preceding argument to $\alpha^{-1}$ 
gives the reverse inclusion, and therefore equality holds. 
Thus $\alpha\in\cN_{\cH}(\cG)$. 

Conversely, suppose $U_\alpha\cG U_\alpha^{-1}=\cG$ and 
take $\Omega_0:=\Omega(\alpha)$. 
If $\mu,\nu\in\Omega_0$ and $t(\mu)=t(\nu)$, then 
\[ 
Z(\nu,\mu)\subset\cG 
\iff Z(\alpha(\nu),\alpha(\mu))\subset\cG, 
\] 
by conjugating with $U_\alpha$. 
If $t(\mu)\neq t(\nu)$, both relations fail, 
since $t(\alpha(\mu))=t(\mu)$ and $t(\alpha(\nu))=t(\nu)$. 
This proves (ii). 
\end{proof}

\begin{remark}\label{rem:subgroupoidgraphBBQS} 
When $V$ consists of a single vertex and $E$ consists of two loops, 
$\sF(\cH)$ is the Higman--Thompson group $V_2$. 
In this case, the equivalence relation $\sim_{\cG}$ and 
the graph $\Delta_{\cG}$ are exactly the type system and the type graph, 
respectively, associated to $\cG$ in the terminology of \cite{BBQS25TAMS}. 
Writing $\mathcal P$ for the corresponding type system, 
$\sF(\cG)$ is their $\Fix(\mathcal P)$ and 
$\cN_{\cH}(\cG)$ is their $\operatorname{Stab}(\mathcal P)$. 
Moreover, Corollary~\ref{cor:primegraph} translates their simplicity 
condition into primeness of the wide open subgroupoid. 
The type semigroup and the main results are compared in 
Section~\ref{subsec:BBQScomparison}. 
\end{remark} 

We conclude by passing in the opposite direction, from subgroups 
to subgroupoids. 
For a subgroup $M\subset\sF(\cH)$, put 
\[ 
\cG_M:=\bigcup_{\alpha\in M}U_\alpha. 
\] 
Since $U_\alpha U_\beta=U_{\alpha\beta}$, this is a wide open 
subgroupoid of $\cH$, normalized by $M$. 
By \cite[Proposition 4.10]{NO19IJM}, it identifies with the 
groupoid of germs of the action of $M$ on $X$. 
By \cite[Proposition 4.6]{NO19IJM}, $\sF(\cG_M)$ consists of 
those homeomorphisms of $X$ which locally agree with elements 
of $M$; in particular, $M\subset\sF(\cG_M)$. 

Following \cite[Definition 2.5]{GV26ETDS}, we call $\sF(\cG_M)$ the 
full amplification of $M$ and say that $M$ is fully amplified 
if $M=\sF(\cG_M)$. 
The same reference also calls such groups ample. 
We use fully amplified for groups and reserve ample for groupoids. 
Since the topological full group of a wide open subgroupoid is 
closed under local agreement, the fully amplified subgroups of 
$\sF(\cH)$ are precisely the topological full groups of wide open 
subgroupoids. 
For a maximal subgroup of $\sF(\cH)$, its full amplification is 
either the subgroup itself or $\sF(\cH)$. 
The next proposition shows that every fully amplified maximal 
subgroup of $\sF(\cH)$ arises from a proper prime wide open subgroupoid. 
Corollary~\ref{cor:amplemaximal} will describe these subgroups 
by combining the trichotomy and maximality criteria below. 
The construction and proposition also hold for any effective ample 
groupoid with Cantor unit space whose orbits contain at least 
three points. 

\begin{proposition}\label{prop:maximalgerms} 
Let $\sD(\cH)\subset\Gamma\subset\sF(\cH)$, and let $M$ be 
a maximal subgroup of $\Gamma$. 
If $\cG_M\neq\cH$, then 
\[ 
M=\Gamma\cap\sF(\cG_M). 
\] 
If in addition $\Gamma=\sF(\cH)$, then 
$\cG_M\subset\cH$ is prime. 
\end{proposition} 

\begin{proof} 
Every arrow of $\cH$ belongs to $U_\alpha$ for some 
$\alpha\in\sD(\cH)$ by \cite[Lemma 3.7]{Ma15crelle}. 
Consequently $\Gamma\cap\sF(\cG_M)$ is a proper subgroup 
of $\Gamma$ when $\cG_M\neq\cH$. 
It contains $M$, so maximality gives the first assertion. 

Suppose that $\Gamma=\sF(\cH)$, and let 
$\cG_M\subset\cK\subsetneq\cH$ be an intermediate wide open 
subgroupoid. 
The same observation shows that $\sF(\cK)\neq\sF(\cH)$, 
and hence maximality gives $\sF(\cK)=M$. 
If $g\in\cK$ has distinct source and range, choose a compact 
open bisection $U\subset\cK$ containing $g$ with 
$s(U)\cap r(U)=\emptyset$. 
Then $\tau_U\in\sF(\cK)=M$, so $g\in\cG_M$. 
If $g\in\cK$ is isotropy at $x$ and the $\cK$-orbit of $x$ 
contains another point, write $g$ as a product of two arrows 
of $\cK$ with distinct source and range. 
The preceding argument again gives $g\in\cG_M$. 

It remains to consider a singleton $\cK$-orbit $\{x\}$. 
Every element of $M=\sF(\cK)$ fixes $x$, so 
$M\subset\St(\cH,\{x\})\subsetneq\sF(\cH)$. 
Maximality gives $M=\St(\cH,\{x\})$. 
For any isotropy arrow $g\in\cH_x$, choose 
$\alpha\in\sD(\cH)$ with $g\in U_\alpha$, as above. 
Then $\alpha(x)=x$, so $\alpha\in M$ and $g\in\cG_M$. 
Thus every arrow of $\cK$ belongs to $\cG_M$, proving 
$\cK=\cG_M$. 
\end{proof}

%%%%%%%%%%%%%%%%%%%%%%%%%%%%%%%%%%%%%%%%%%%%%%%%%%%%%%%%%%%%
\section{A trichotomy for prime subgroupoids}\label{sec:trichotomy} 
%%%%%%%%%%%%%%%%%%%%%%%%%%%%%%%%%%%%%%%%%%%%%%%%%%%%%%%%%%%%

We retain the notation of Section~\ref{sec:subgroupoidgraph}. 
Thus $\cH=\cG_{(V,E)}$ is the SFT groupoid of a finite irreducible graph 
whose adjacency matrix is not a permutation matrix, and $X=\cH^{(0)}$. 
By Proposition~\ref{prop:maximalgerms}, every fully amplified maximal 
subgroup of $\sF(\cH)$ is the topological full group of a proper 
prime wide open subgroupoid. 
Grigorchuk and Vorobets \cite[Theorem 6.6 and Lemma 6.7]{GV26ETDS} 
show that, for a fully amplified group acting minimally 
on a Cantor set, 
a maximal subgroup which fails to act minimally is the stabilizer 
of a nonempty proper closed set which is either clopen or nowhere dense. 
Using their observation about boundaries of minimal invariant sets, 
we obtain the following trichotomy for prime wide open subgroupoids. 

\begin{theorem}\label{thm:primetrichotomy} 
Let $\cG\subsetneq\cH$ be a prime wide open subgroupoid. 
Then exactly one of the following holds. 
\begin{enumerate} 
\item $\cG$ is minimal. 
\item There is a partition $X=X_1\sqcup X_2$ into two nonempty clopen sets 
such that 
\[ 
\cG=\cH|X_1\sqcup\cH|X_2. 
\] 
The two reductions are minimal and purely infinite, 
and the partition is unique up to interchanging $X_1$ and $X_2$. 
\item There is a unique nonempty proper closed $\cG$-invariant subset 
$Y\subset X$. 
It is nowhere dense, $\cG|Y$ is minimal, and 
\[ 
\cG|(X\setminus Y)=\cH|(X\setminus Y). 
\] 
Moreover, $Y$ is either finite or a Cantor set, 
and every $\cG$-orbit in $X\setminus Y$ is dense in $X$. 
\end{enumerate} 
\end{theorem} 

\begin{proof} 
Suppose that $\cG$ is not minimal. 
By compactness and Zorn's lemma, there is a nonempty proper closed 
$\cG$-invariant subset $Y\subset X$ which is minimal under inclusion. 
As in \cite[Lemma 6.7]{GV26ETDS}, its boundary $\partial Y$ is 
a closed $\cG$-invariant subset of $Y$: 
the partial homeomorphisms defined by open bisections of $\cG$ 
preserve $Y$ and its interior. 
Minimality of $Y$ therefore gives $\partial Y=\emptyset$ or $\partial Y=Y$. 
Thus $Y$ is either clopen or nowhere dense. 

Suppose first that $Y$ is clopen. 
Since $Y$ is $\cG$-invariant, 
\[ 
\cG\subset\cK:=\cH|Y\sqcup\cH|(X\setminus Y)\subsetneq\cH. 
\] 
Here $\cK$ is wide and open, and the last inclusion is strict 
because $\cH$ is minimal and $Y$ is nonempty and proper. 
Primeness gives $\cG=\cK$, which is \textup{(2)} with 
$X_1=Y$ and $X_2=X\setminus Y$. 
Each $\cH|X_i$ is minimal and purely infinite, since it is a reduction 
of the minimal purely infinite groupoid $\cH$ to a nonempty clopen set. 
The $\cG$-orbit closures are precisely $X_1$ and $X_2$, 
so the partition is unique up to interchange. 

Suppose next that $Y$ is nowhere dense, and put $U:=X\setminus Y$. 
Let $\cK$ be the subgroupoid of $\cH$ generated by $\cG$ and $\cH|U$. 
It is wide and open, since products of open subsets of 
an \'etale groupoid are open. 
Both generating subgroupoids preserve $Y$, so $Y$ is $\cK$-invariant. 
Consequently $\cK\neq\cH$ by minimality of $\cH$, 
and primeness yields $\cK=\cG$. 
It follows that $\cH|U\subset\cG$, and hence $\cG|U=\cH|U$. 

For every $x\in U$, its $\cG$-orbit is the intersection of 
its $\cH$-orbit with $U$. 
This intersection is dense in $U$ by minimality of $\cH$, 
and hence dense in $X$ because $U$ is open and dense. 
Every proper closed $\cG$-invariant subset is therefore contained in $Y$. 
Minimality of $Y$ now shows that it is the unique nonempty such subset. 
Since $Y$ is $\cG$-invariant, every bisection of $\cG$ preserves $Y$ 
on its domain. Thus the regular reduction $\cG|Y$ coincides with 
the full restriction to $Y$ and has precisely the $\cG$-orbits 
contained in $Y$. It is therefore minimal. 
Finally, if $Y$ is infinite, its set of nonisolated points is 
nonempty, closed and invariant under this reduction. 
It equals $Y$ by minimality, so $Y$ is a Cantor set. 
This proves \textup{(3)}. 

The three alternatives are mutually exclusive: 
in \textup{(1)} there is no nonempty proper closed invariant subset, 
in \textup{(2)} there are exactly two, 
and in \textup{(3)} there is exactly one. 
\end{proof} 

\begin{remark}\label{rem:primetrichotomy} 
The same argument gives this trichotomy for any minimal ample groupoid 
$\cH$ with Cantor unit space, with the assertion of pure infiniteness 
in \textup{(2)} omitted. 
In the SFT setting, Corollary~\ref{cor:minimalsubgroupoiddichotomy} 
shows that $\cG$ in \textup{(1)} is either AF or purely infinite. 
We leave open whether the AF alternative can occur for a prime inclusion. 
\end{remark}

We treat case \textup{(2)} in the proposition below, 
case \textup{(3)} in Section~\ref{sec:exceptional}, 
and case \textup{(1)} in Section~\ref{sec:minimal}. 
The proposition also covers partitions into more than two clopen sets: 
these can yield maximal normalizers even when the subgroupoid is not prime. 

For $\Gamma=\sF(\cH)$, the maximality criterion follows from 
\cite[Theorems 6.19 and 6.20]{GV26ETDS}, whose condition for 
partitions into at least three sets is transitivity of the stabilizer 
on the partition. 
Here we express this condition in terms of $H_0(\cH)$ and give 
a direct proof for all groups between $\sD(\cH)$ and $\sF(\cH)$. 

\begin{proposition}\label{prop:clopenpartitionnormalizer} 
Let $\mathcal Q=\{X_1,\dots,X_n\}$ be a partition of $X$ 
into nonempty clopen sets, where $n\geq2$, and put 
\[ 
\cG:=\bigsqcup_{i=1}^n\cH|X_i. 
\] 
Then 
\[ 
\Nor\left(\sF(\cH),\sF(\cG)\right) 
=\{\alpha\in\sF(\cH)\mid\alpha(\mathcal Q)=\mathcal Q\}. 
\] 
For every group $\Gamma$ with 
$\sD(\cH)\subset\Gamma\subset\sF(\cH)$, the subgroup 
$K:=\Gamma\cap\Nor\left(\sF(\cH),\sF(\cG)\right)$ is maximal in $\Gamma$ 
if and only if 
\[ 
n=2\quad\text{or}\quad 
[1_{X_1}]=\cdots=[1_{X_n}]\text{ in }H_0(\cH). 
\] 
\end{proposition} 

\begin{proof} 
By Lemma~\ref{lem:groupnormalizers}, the group normalizer equals 
$\cN_{\cH}(\cG)$. 
The $\cG$-orbit closures are precisely $X_1,\dots,X_n$, 
so every element of $\Nor\left(\sF(\cH),\sF(\cG)\right)$ permutes these sets. 
Conversely, any element of $\sF(\cH)$ permuting them normalizes 
$\cG=\bigsqcup_i\cH|X_i$. 
This proves the asserted description of the normalizer. 

First, $K\neq\Gamma$. 
Indeed, choose a nonempty proper clopen subset $A\subset X_1$. 
By Lemma~\ref{threecycle} (3), 
$\sD(\cH)_{A\cup X_2}$ acts minimally on $A\cup X_2$. 
Hence it contains an element sending a point of $A$ into $X_2$ 
and fixing $X_1\setminus A$ pointwise. 
This element belongs to $\Gamma\setminus K$. 

Suppose that $n=2$ or that all the classes $[1_{X_i}]$ are equal. 
Let $K\subsetneq L\subset\Gamma$ and choose $\alpha\in L\setminus K$. 
Since $\alpha$ does not preserve $\mathcal Q$, 
some $B:=\alpha(X_i)$ meets two distinct members $X_j,X_k$. 
The group $L$ contains $\Gamma_{X_j}$ and $\Gamma_{X_k}$, as well as 
$\Gamma_B=\alpha\Gamma_{X_i}\alpha^{-1}$. 
By \cite[Proposition 3.11]{Ma26maximal2}, $\Gamma$ has Property E: 
\[ 
\Gamma_{D_1\cup D_2} 
=\left\langle\Gamma_{D_1}\cup\Gamma_{D_2}\right\rangle 
\] 
for clopen sets $D_1,D_2\subset X$ with $D_1\cap D_2\neq\emptyset$. 
Applying it first to $X_j,B$ and then to $X_j\cup B,X_k$ gives 
\[ 
\Gamma_{X_j\cup X_k}\subset\Gamma_{X_j\cup B\cup X_k}\subset L. 
\] 
If $n=2$, this gives $L=\Gamma$. 

Now assume that $n\geq3$ and that all the classes $[1_{X_i}]$ are equal. 
We first show that $K$ induces every permutation of $\mathcal Q$. 
For every permutation $\sigma$ of $\{1,\dots,n\}$, 
Theorem~\ref{niceproperties} (2) provides compact open bisections 
$U_i\subset\cH$ with $s(U_i)=X_i$ and $r(U_i)=X_{\sigma(i)}$. 
Their union is a full bisection, so 
$h:=\theta_{\bigsqcup_iU_i}$ induces $\sigma$ on $\mathcal Q$. 
By Lemma~\ref{localabelianization}, there exists 
$\beta\in\sF(\cH)_{X_1}$ such that $h\beta\in\sD(\cH)$. 
Thus $h\beta\in K$ and still induces $\sigma$ on $\mathcal Q$. 
Conjugating $\Gamma_{X_j\cup X_k}\subset L$ by these elements of $K$ 
gives $\Gamma_{X_r\cup X_s}\subset L$ for every $r\neq s$. 
Here conjugation carries local subgroups of $\Gamma$ to local 
subgroups of $\Gamma$, since the conjugating elements belong to $\Gamma$. 
Repeated applications of Property E to the sets $X_1\cup X_s$, 
$2\leq s\leq n$, now give $\Gamma_X\subset L$. 
Hence $L=\Gamma$, proving maximality. 

Conversely, suppose $n\geq3$ and the classes $[1_{X_i}]$ 
are not all equal. 
Choose one of these classes $a$ and put 
\[ 
Y:=\bigcup_{\{i\mid[1_{X_i}]=a\}}X_i. 
\] 
Every element of $\sF(\cH)$ preserves the homology class 
of a clopen set, so $K\subset\St(\Gamma,Y)$. 
Since $Y$ is nonempty and proper and $\Gamma$ acts minimally on $X$, 
$\St(\Gamma,Y)\neq\Gamma$. 
Moreover, one of $Y$ and $X\setminus Y$ contains two distinct 
members $X_i,X_j$ of $\mathcal Q$. 
The same local construction used above to show $K\neq\Gamma$, 
applied to $X_i,X_j$, gives an element of $\sD(\cH)$ 
supported in $X_i\cup X_j$ which does not preserve $\mathcal Q$. 
It belongs to $\St(\Gamma,Y)\setminus K$. 
Thus $K\subsetneq\St(\Gamma,Y)\subsetneq\Gamma$, 
and $K$ is not maximal. 
\end{proof} 

\begin{remark}\label{rem:finitepartitionBBQS} 
The case $n=2$ applies to case \textup{(2)} of 
Theorem~\ref{thm:primetrichotomy}. 
For $n\geq3$, the inclusion $\cG\subset\cH$ is not prime: 
merging $X_1$ and $X_2$ gives a proper intermediate wide open subgroupoid. 
Thus a maximal normalizer need not arise from a prime inclusion. 
For the full $2$-shift, $H_0(\cH)=0$, 
so every finite clopen partition has maximal stabilizer in $V_2$. 
Choosing bisections identifying its parts and using 
Lemma~\ref{lem:groupnormalizers} gives 
\[ 
\Nor\left(\sF(\cH),\sF(\cG)\right)\cong V_2^n\rtimes S_n, 
\] 
where $S_n$ permutes the factors; here each clopen reduction 
of the full $2$-shift groupoid has topological full group isomorphic 
to $V_2$. 
This recovers \cite[Proposition 6.3]{BBQS25TAMS}. 
\end{remark}

%%%%%%%%%%%%%%%%%%%%%%%%%%%%%%%%%%%%%%%%%%%%%%%%%%%%%%%%%%%%
\section{Subgroupoids with a nowhere dense invariant set} 
\label{sec:exceptional} 
%%%%%%%%%%%%%%%%%%%%%%%%%%%%%%%%%%%%%%%%%%%%%%%%%%%%%%%%%%%%

We retain the notation of Section~\ref{sec:subgroupoidgraph}, 
so $\cH=\cG_{(V,E)}$ and $X=\cH^{(0)}$. 
Throughout this section, $Y\subsetneq X$ is a nonempty nowhere dense 
closed set such that the regular reduction $\cH|Y$ is minimal. 
Let $\cG:=\cS_Y$ be the $Y$-preserving subgroupoid of $\cH$ 
from \cite[Definition 3.19]{Ma26maximal2}; explicitly, 
\begin{equation}\label{eq:exceptionalsubgroupoid} 
\cG=\bigcup\left\{U\in\cCO(\cH)\mid 
\text{$U$ is a bisection and }YU=UY\right\}. 
\end{equation} 
We first identify its normalizer and establish maximality for all groups 
between $\sD(\cH)$ and $\sF(\cH)$. 
We then model the regular reduction to $Y$ by a graph and classify 
it according to whether $Y$ is finite or Cantor and whether the 
graph has closed paths. 
For Cantor sets, this gives an AF/purely infinite dichotomy and 
invariants of the abstract isomorphism class of the stabilizer. 

The general results of 
\cite[Definition 3.19 and Proposition 3.20]{Ma26maximal2} give 
\[ 
\cG|Y=\cH|Y,\qquad 
\cG|(X\setminus Y)=\cH|(X\setminus Y),\qquad 
\sF(\cG)=\St(\cH,Y), 
\] 
and show that $\cG\subset\cH$ is prime. 
They also show that $Y$ is the unique nonempty proper closed 
$\cG$-invariant subset of $X$, that every $\cG$-orbit in 
$X\setminus Y$ is dense in $X$, and that $Y$ is finite or a Cantor set. 

This setting includes case \textup{(3)} of 
Theorem~\ref{thm:primetrichotomy}. 
Indeed, let $\cK\subsetneq\cH$ be a prime wide open subgroupoid 
in that case, and let $Y$ be its exceptional invariant set. 
Then $\cK\subset\cS_Y=\cG\subsetneq\cH$, and primeness gives 
$\cK=\cG$. Thus $\cH|Y=\cK|Y$ is minimal as required. 

To apply the isomorphism invariants of 
\cite[Theorem 4.3 (2)]{Ma26maximal2} when $Y$ is a Cantor set, 
we need the effective reduction to $Y$ to be Hausdorff. 
The next lemma shows that the regular reduction is already effective 
in this case. 
The maximality criterion used below requires no such assumption. 

\begin{lemma}\label{lem:SFTperfectreduction} 
Let $Z\subset X$ be a closed subset with no isolated points. 
Then $\cH|Z$ is effective, and $\cH{\downarrow}Z=\cH|Z$. 
\end{lemma} 

\begin{proof} 
A basic prefix replacement $\mu z\mapsto\nu z$ with 
$\mu\neq\nu$ has at most one fixed point. 
If $\lvert\mu\rvert=\lvert\nu\rvert$, it has no fixed point. 
Otherwise, equality requires one prefix to extend the other, 
and periodicity uniquely determines the remaining tail. 
If a nonunit arrow $g$ belonged to $\Iso(\cH|Z)^\circ$, we could 
choose a basic bisection $B=Z(\nu,\mu)$ containing $g$ with 
$B\cap(\cH|Z)\subset\Iso(\cH|Z)^\circ$. 
Then $\mu\neq\nu$, and the nonempty relatively open set 
$s(B\cap(\cH|Z))\subset Z$ would consist of the single fixed 
point $s(g)$, contradicting the absence of isolated points. 
\end{proof} 

\begin{theorem}\label{thm:exceptionalstabilizer} 
One has 
\[ 
\sF(\cG)=\Nor\left(\sF(\cH),\sF(\cG)\right)=\St(\cH,Y). 
\] 
Moreover, $\cN_{\cH}(\cG)=\sF(\cG)$. 
For every group $\Gamma$ with $\sD(\cH)\subset\Gamma\subset\sF(\cH)$, 
the subgroup $\Gamma\cap\sF(\cG)=\St(\Gamma,Y)$ is maximal in $\Gamma$. 
\end{theorem} 

\begin{proof} 
The equality $\sF(\cG)=\St(\cH,Y)$ was recalled above. 
Every element of $\cN_{\cH}(\cG)$ preserves $Y$, since it is the 
unique nonempty proper closed $\cG$-invariant subset of $X$. 
Moreover, $Y$ is the largest proper closed subset invariant under 
$\sF(\cG)$. 
Indeed, let $x\notin Y$ and let $O\subset X$ be nonempty and open. 
Choose a nonempty clopen set $B\subset O\setminus Y$ and 
a clopen set $D\subset X\setminus Y$ containing $x$ and $B$. 
Since $\sD(\cH)_D\subset\St(\cH,Y)=\sF(\cG)$, 
Lemma~\ref{threecycle} (3) gives an element of $\sF(\cG)$ 
sending $x$ into $B$. 
Thus every element of 
$\Nor\left(\sF(\cH),\sF(\cG)\right)$ also preserves $Y$. 
The reverse inclusions hold because $\sF(\cG)$ normalizes both 
itself and $\cG$. This proves the normalizer assertions. 

Since $\cH$ is minimal and purely infinite and $\cH|Y$ is minimal, 
\cite[Proposition 3.20]{Ma26maximal2} gives the maximality of 
$\St(\Gamma,Y)$ for every $\Gamma$ in the statement. 
The equality with $\Gamma\cap\sF(\cG)$ follows from the first part. 
\end{proof}

We now turn to the subgroupoid graph of $\cG$, which will provide 
a concrete model for the regular reduction to $Y$. 
For $\mu\in\Omega$, the set 
\[ 
\{z\in Z(t(\mu))\mid\mu z\in Y\} 
\] 
is called the follower set of $\mu$ in $Y$. 
The next lemma characterizes $\sim_{\cG}$ in terms of these sets. 

\begin{lemma}\label{lem:exceptionalprefix} 
For $\mu,\nu\in\Omega$ with $t(\mu)=t(\nu)$, one has 
\begin{equation}\label{eq:exceptionalprefix} 
\mu\sim_{\cG}\nu 
\iff 
\left(\mu z\in Y\iff\nu z\in Y\right) 
\quad\text{for every }z\in Z(t(\mu)). 
\end{equation} 
\end{lemma} 

\begin{proof} 
The forward implication follows from $\cG$-invariance of $Y$. 
Conversely, the condition on the right says that 
$U:=Z(\nu,\mu)$ satisfies $YU=UY$, so $U\subset\cG$ 
by \eqref{eq:exceptionalsubgroupoid}. 
\end{proof} 

Put $\Delta:=\Delta_{\cG}$, $\rho:=\rho_{\cG}$ and 
$\varsigma:=\varsigma_{\cG}$, and define 
\[ 
R:=\{[\mu]_{\cG}\mid Z(\mu)\cap Y=\emptyset\}, 
\qquad C:=\Delta^0\setminus R. 
\] 
Since $Y$ is $\cG$-invariant, membership in $R$ does not depend 
on the choice of the representative $\mu$. 
For a subset of vertices, we use $\Delta|R$ or $\Delta|C$ 
to denote the induced subgraph. 
A set of vertices is said to be forward closed 
if every edge starting in it also ends in it. 
The strongly connected components are the equivalence classes 
for mutual reachability, allowing paths of length zero. 
Such a component is terminal if it is forward closed. 

\begin{lemma}\label{lem:exceptionalgraph} 
The subgroupoid graph $\Delta$ has the following properties. 
\begin{enumerate} 
\item $R$ is the unique terminal strongly connected component of $\Delta$, 
and $\rho:\Delta|R\to(V,E)$ is an isomorphism. 
Every vertex of $\Delta$ can reach $R$. 
\item One has 
\[ 
Y=\{x\in X\mid [x_1\cdots x_n]_{\cG}\in C 
\text{ for every }n\geq1\}. 
\] 
The graph $\Delta|C$ is countable, row-finite and has no sinks. 
Moreover, every vertex of $C$ can reach a vertex on any given 
infinite path in $\Delta|C$. 

\end{enumerate} 
\end{lemma} 

\begin{proof} 
(1)\: 
If $Z(\mu)$ and $Z(\nu)$ are disjoint from $Y$ and 
$t(\mu)=t(\nu)$, then 
\[ 
Z(\nu,\mu)\subset\cH|(X\setminus Y)\subset\cG. 
\] 
Thus $\rho|R$ is injective. 
Choose a cylinder disjoint from $Y$ and extend its defining path 
to end at any prescribed vertex of $(V,E)$. 
Irreducibility shows that $\rho(R)=V$. 
The set $R$ is forward closed, and the right-covering property 
now gives the asserted graph isomorphism. 
In particular, $R$ is a terminal strongly connected component. 
Every cylinder has a subcylinder disjoint from $Y$, since $Y$ is nowhere 
dense, so every vertex of $\Delta$ reaches $R$. 
No other strongly connected component can therefore be terminal. 

(2)\: 
A point is outside $Y$ if and only if some cylinder containing it 
is disjoint from $Y$, which gives the displayed formula. 
If $[\mu]_{\cG}\in C$, choose $x\in Z(\mu)\cap Y$. 
Extending $\mu$ by the next edge of $x$ gives an edge in $\Delta|C$. 
Hence this graph has no sinks; countability and row-finiteness 
are inherited from $\Delta$. 

Let $P=[\mu]_{\cG}\in C$ and let $\widetilde z$ be an infinite path 
in $\Delta|C$. 
Choose $\nu\in\Omega$ whose lift ends at $i(\widetilde z_1)$, 
and put $z:=\rho(\widetilde z)$. 
Since $R$ is forward closed, the lift of $\nu$ also stays in $C$, 
and the displayed formula gives $\nu z\in Y$. 
By minimality of $\cG|Y$, some arrow of this reduction sends $\nu z$ 
into $Z(\mu)\cap Y$. 
Refining a basic bisection of $\cG$ containing that arrow, we obtain 
a finite path $\eta$ and a finite prefix $\zeta$ of $z$ such that 
\[ 
\mu\eta\sim_{\cG}\nu\zeta. 
\] 
The lift of $\eta$ from $P$ stays in $C$ and ends at the vertex 
$[\nu\zeta]_{\cG}$ on $\widetilde z$. 
This proves the last assertion. 

\end{proof}

The graph left after removing $R$ models the regular reduction to $Y$. 
When $Y$ is finite, the groupoids below need not be effective. 
We use Kakutani equivalence for them with the same definition in 
terms of isomorphic full clopen reductions. 

\begin{proposition}\label{prop:exceptionalgraphmodel} 
Put 
\[ 
S:=\bigsqcup_{\substack{v\in V\\\varsigma(v)\in C}} 
Z(\varsigma(v))\subset X_{\Delta|C}. 
\] 
Then $S$ is full and compact open, 
and $\rho$ induces an isomorphism $\cG_{\Delta|C}|S\to\cG|Y$. 
Consequently, $\cH|Y=\cG|Y$ is Kakutani equivalent to 
$\cG_{\Delta|C}$. 
\end{proposition} 

\begin{proof} 
The set $S$ is compact open because the root section is finite. 
Lemma~\ref{lem:exceptionalgraph} (2) identifies $Y$ with precisely 
the paths lifted from the root section which stay in $C$. 
Thus the restriction $\rho|S:S\to Y$ is a homeomorphism. 
Every vertex of $C$ is reachable from the root section by a path 
staying in $C$, since no path can leave $R$. 
Prepending such a path shows that every $\cG_{\Delta|C}$-orbit meets $S$. 
As $Y$ is $\cG$-invariant, all arrows of $\cG$ with source in $Y$ 
are regular. 
The computation in Proposition~\ref{prop:subgroupoidgraphmodel}, 
restricted to paths which stay in $C$, therefore gives the asserted 
groupoid isomorphism. 
In particular, $\cG_{\Delta|C}$ is minimal. 
\end{proof} 

For infinite $Y$, the graph model yields invariants of the 
stabilizer. Indeed, Lemma~\ref{lem:SFTperfectreduction} gives 
\[ 
\cH{\downarrow}Y=\cH|Y=\cG|Y. 
\] 
For every $\sD(\cH)\subset\Gamma\subset\sF(\cH)$, 
\cite[Theorem 4.3 (2)]{Ma26maximal2}, together with 
Theorem~\ref{thm:exceptionalstabilizer}, shows that the isomorphism 
class of $\cG|Y$ is an invariant of the abstract isomorphism class 
of $\Gamma\cap\sF(\cG)$. 
In particular, the Kakutani equivalence class of $\cG_{\Delta|C}$ 
is an invariant. 

The same model gives the following classification of the regular 
reduction. The absence of closed paths is equivalent to the absence 
of nontrivial isotropy in this reduction. 

\begin{corollary}\label{cor:exceptionalclassification} 
\begin{enumerate} 
\item If $Y$ is finite, it is contained in a single $\cH$-orbit. 
If its points are not eventually periodic, then 
$\cH|Y\cong Y\times Y$ and $\Delta|C$ has no closed paths. 
Otherwise, $\cH|Y\cong(Y\times Y)\times\Z$ and 
$\Delta|C$ contains a closed path. 
Here $Y\times Y$ denotes the pair groupoid. 
\item If $Y$ is a Cantor set, the following conditions are equivalent: 
$\cH|Y$ is AF; $\Delta|C$ has no closed paths; 
$\cH|Y$ is principal. 
If these conditions fail, $\cH|Y$ is purely infinite. 
\end{enumerate} 
\end{corollary} 

\begin{proof} 
A graph groupoid has nontrivial isotropy exactly when its graph 
has a closed path: its points with nontrivial isotropy are the 
eventually periodic infinite paths. 
Since $S$ in Proposition~\ref{prop:exceptionalgraphmodel} is full, 
$\cH|Y$ is principal if and only if $\Delta|C$ has no closed paths. 

If $Y$ is finite, every arrow with source and range in $Y$ is 
regular by \cite[Section 3.1]{Ma26maximal2}. 
Minimality therefore means that $Y$ lies in one $\cH$-orbit. 
The isotropy groups are trivial when its points are not eventually 
periodic, and are isomorphic to $\Z$ otherwise. 
A transitive \'etale groupoid with finite unit space is discrete 
and isomorphic to the product of the pair groupoid on its units 
and any one of its isotropy groups. This proves (1). 

If $Y$ is a Cantor set, $\cH|Y$ is effective by 
Lemma~\ref{lem:SFTperfectreduction}. 
Thus $\Delta|C$ is cofinal and satisfies Condition~\textup{(L)}. 
The proof of Corollary~\ref{cor:minimalsubgroupoiddichotomy}, 
applied to the full compact open reduction 
$\cG_{\Delta|C}|S$, shows that $\cH|Y$ is AF in the acyclic 
case and purely infinite otherwise. 
Since AF groupoids are principal, this proves (2). 
\end{proof} 

Table~\ref{tab:exceptionalconfigurations} summarizes the four cases. 

\begin{table}[htbp] 
\caption{Classification of the regular reduction to $Y$.} 
\label{tab:exceptionalconfigurations} 
\centering 
\begingroup 
\small 
\renewcommand{\arraystretch}{1.2} 
\begin{tabular}{@{}p{0.14\textwidth}p{0.30\textwidth}p{0.46\textwidth}@{}} 
\hline 
$Y$ & Closed paths in $\Delta|C$ & $\cH|Y$ \\ 
\hline 
Finite & None & $Y\times Y$ \\ 
Finite & Present & $(Y\times Y)\times\Z$ \\ 
Cantor & None & Minimal AF \\ 
Cantor & Present & Minimal purely infinite \\ 
\hline 
\end{tabular} 
\endgroup 
\end{table} 

\begin{example}\label{ex:exceptionalAF} 
By \cite[Proposition 4.12 and Corollary 4.13]{Ma26maximal2}, 
every minimal AF groupoid with Cantor unit space occurs as 
$\cH|Y$ in the present setting. 
The corresponding graph $\Delta|C$ is acyclic by 
Corollary~\ref{cor:exceptionalclassification}, and it is infinite 
because it has no sinks. 
These realizations therefore supply the Cantor AF case in 
Table~\ref{tab:exceptionalconfigurations}. 

For a concrete UHF example, let $\cH$ be the full $2$-shift groupoid 
and put 
\[ 
Y:=\left\{a_1\,0^2\,1\,a_2\,0^4\,1\,a_3\,0^6\,1\cdots 
\mid a_n\in\{0,1\}\right\}. 
\] 
This is a nowhere dense Cantor set. 
The maximal blocks of zeros have lengths $2n$ or $2n+1$, so equality 
of tails forces alignment of the block indices. 
Hence the full restriction to $Y$ is the usual tail equivalence 
groupoid on the free symbols $(a_n)$, of UHF type $2^\infty$. 
In $\Delta|C$, the $n$th stage consists of two parallel edges 
followed by a chain of $2n+1$ edges; the other edges at the vertices 
of this chain lead to $R$. 
Telescoping at the block boundaries gives the standard Bratteli 
diagram with one vertex at each level and two edges between 
successive levels. 
\end{example}

\begin{example}\label{ex:aperiodicpoint} 
An acyclic graph can also arise from a singleton invariant set. 
Let $\cH$ be the groupoid of the full $2$-shift, 
choose a point $x\in X$ which is not eventually periodic, 
and let $Y:=\{x\}$. 
The isotropy group $\cH_x$ is trivial, so $\cH|Y=Y$ is minimal. 
The subgroupoid in \eqref{eq:exceptionalsubgroupoid} is 
\[ 
\cG=\cH|(X\setminus\{x\})\cup X. 
\] 

The vertices outside $R$ are the classes of the prefixes of $x$, 
including the empty prefix. 
These classes are distinct. 
Indeed, if distinct prefixes $\mu,\nu$ had the same class, 
then $Z(\nu,\mu)\subset\cG$ would contain an arrow with source $x$ 
and nonzero middle coordinate $\lvert\nu\rvert-\lvert\mu\rvert$. 
This is impossible, since the only arrow of $\cG$ with source $x$ 
is the unit at $x$. 
Consequently $\Delta|C$ is an infinite directed ray and has no 
closed path. 
Theorem~\ref{thm:exceptionalstabilizer} gives the maximal point 
stabilizer $\sF(\cG)=\St(\cH,\{x\})$. 
\end{example}

%%%%%%%%%%%%%%%%%%%%%%%%%%%%%%%%%%%%%%%%%%%%%%%%%%%%%%%%%%%%
\section{Minimal subgroupoids}\label{sec:minimal} 
%%%%%%%%%%%%%%%%%%%%%%%%%%%%%%%%%%%%%%%%%%%%%%%%%%%%%%%%%%%%

Retaining the notation of Section~\ref{sec:subgroupoidgraph}, 
we study case \textup{(1)} of Theorem~\ref{thm:primetrichotomy}. 
Cases \textup{(2)} and \textup{(3)} were treated in 
Proposition~\ref{prop:clopenpartitionnormalizer} and 
Theorem~\ref{thm:exceptionalstabilizer}, respectively. 
Until Corollary~\ref{cor:primenormalizermaximal}, we assume that 
$\cG\subset\cH$ is a minimal wide open subgroupoid. 
By Corollary~\ref{cor:minimalsubgroupoiddichotomy}, 
it is either AF or purely infinite. 
Lemma~\ref{elementofDG} therefore applies in both cases. 
By Lemma~\ref{lem:groupnormalizers}, the groupoid normalizer agrees 
with the usual group normalizer. We write 
\[ 
N:=\Nor\left(\sF(\cH),\sF(\cG)\right)=\cN_{\cH}(\cG). 
\] 
Our main step is to prove that, when $\cG\subsetneq\cH$ is prime, 
any subgroup of $\sF(\cH)$ containing $\sD(\cG)$ and not contained 
in $N$ must contain $\sD(\cH)$ 
(Theorem~\ref{thm:minimalovergroup}). 
We associate an equivalence relation on $\Omega$ with such a subgroup. 
Primeness will force this relation to identify all paths with a common 
terminal vertex, allowing us to recover the three-cycles that generate 
$\sD(\cH)$. 
This follows the strategy of 
\cite[Proposition 5.10 and Theorem 5.12]{BBQS25TAMS}, 
where transpositions in an overgroup define a quotient type system 
and simplicity forces the overgroup to contain all transpositions. 
Since our subgroups are only assumed to contain $\sD(\cG)$, 
we use three-cycles instead, and we extend the argument to 
irreducible SFT groupoids. 
The construction and basic properties of the relation do not require 
primeness. 
We then characterize maximality of the normalizer by its image 
in the ambient abelianization. 
At the end of the section, we drop the minimality assumption, 
combine the three cases of Theorem~\ref{thm:primetrichotomy}, 
and describe the fully amplified maximal subgroups of $\sF(\cH)$. 

For $n\geq3$ and pairwise incomparable finite paths 
$\gamma_0,\dots,\gamma_{n-1}$ with a common terminal vertex, let 
$\sA(\gamma_0,\dots,\gamma_{n-1})$ be the subgroup of $\sF(\cH)$ 
consisting of the prefix permutations 
\[ 
\gamma_i z\longmapsto\gamma_{\pi(i)}z 
\qquad(0\leq i\leq n-1), 
\] 
where $\pi$ is an even permutation of $\{0,\dots,n-1\}$, 
extended by the identity outside $\bigsqcup_iZ(\gamma_i)$. 
It is naturally isomorphic to the alternating group of degree $n$ 
and is independent of the ordering of the paths. 
The group associated with a subfamily fixes all other cylinders. 
Every such even permutation is a product of three-cycles, each of 
which is a commutator of transpositions. 
Consequently, if $\gamma_i\sim_{\cG}\gamma_j$ for all $i,j$, then 
$\sA(\gamma_0,\dots,\gamma_{n-1})\subset\sD(\cG)$. 

The group $\sD(\cH)$ is generated by the groups 
$\sA(\gamma_0,\gamma_1,\gamma_2)$ with 
$\bigsqcup_{i=0}^2Z(\gamma_i)\neq X$. 
Indeed, Nekrashevych's alternating group equals $\sD(\cH)$ 
for the minimal purely infinite groupoid $\cH$; 
see \cite[Section 4]{Ne19ETDS} and \cite[Section 2.6]{Ma26maximal1}. 
Its generating three-cycles can be written as products of 
three-cycles on disjoint triples of cylinders by refining the 
compact open bisections to basic prefix replacements. 
Since the cylinders are Cantor sets, a further common suffix 
refinement splits each such three-cycle into factors with proper 
support.

\begin{lemma}\label{lem:prefixplacement} 
Let $U_1,\dots,U_k\in\cCO(X)$. 
There exists $L\geq0$ such that, for any 
$\beta_1,\dots,\beta_k\in\Omega$ with 
$\lvert\beta_i\rvert\geq L$, there are 
$\gamma_1,\dots,\gamma_k\in\Omega$ satisfying 
\[ 
\gamma_i\sim_{\cG}\beta_i,\qquad 
Z(\gamma_i)\subset U_i\quad(1\leq i\leq k). 
\] 
In particular, if the $U_i$ are mutually disjoint and 
$\bigcup_iU_i\neq X$, then the $\gamma_i$ are pairwise 
incomparable and $\bigsqcup_iZ(\gamma_i)\neq X$. 
\end{lemma} 

\begin{proof} 
It suffices to prove the first assertion for one set $U$. 
For every $x\in X$, minimality gives an arrow in $\cG$ 
with source $x$ and range in $U$. 
Since $\cG$ is open, we may choose a basic bisection 
$Z(\delta,\alpha)\subset\cG$ containing this arrow 
such that $Z(\delta)\subset U$. 
By compactness, there are finitely many such bisections 
$Z(\delta_j,\alpha_j)$ for which the sets $Z(\alpha_j)$ 
cover $X$. 
Put $L:=\max_j\lvert\alpha_j\rvert$. 
If $\lvert\beta\rvert\geq L$, choose $j$ with 
$Z(\beta)\cap Z(\alpha_j)\neq\emptyset$. 
Then $\beta=\alpha_j\zeta$ for some finite path $\zeta$. 
Setting $\gamma:=\delta_j\zeta$, we have 
\[ 
Z(\gamma,\beta)\subset Z(\delta_j,\alpha_j)\subset\cG, 
\qquad Z(\gamma)\subset U. 
\] 
For finitely many $U_i$, take the maximum of the corresponding 
bounds. The last assertion follows from 
$Z(\gamma_i)\subset U_i$. 
\end{proof}

\begin{lemma}\label{lem:prefixtransport} 
Let $n\geq3$, and let $(\gamma_i)_{i=0}^{n-1}$ and 
$(\gamma'_i)_{i=0}^{n-1}$ be two families of pairwise incomparable 
finite paths, each with a common terminal vertex. 
Suppose that $\gamma_i\sim_{\cG}\gamma'_i$ for every $i$ and 
\[ 
\bigsqcup_iZ(\gamma_i)\neq X, 
\qquad\bigsqcup_iZ(\gamma'_i)\neq X. 
\] 
Then there exists $g\in\sD(\cG)$ such that 
$g(\gamma_i z)=\gamma'_i z$ for every $i$ and every 
$z\in Z(t(\gamma_i))$. 
In particular, for any subgroup $H\subset\sF(\cH)$ containing 
$\sD(\cG)$, 
\[ 
\sA(\gamma_0,\dots,\gamma_{n-1})\subset H 
\iff 
\sA(\gamma'_0,\dots,\gamma'_{n-1})\subset H. 
\] 
\end{lemma} 

\begin{proof} 
The set 
\[ 
U:=\bigsqcup_{i=0}^{n-1}Z(\gamma'_i,\gamma_i) 
\] 
is a compact open bisection in $\cG$, and both its source 
and its range are proper subsets of $X$. 
By Lemma~\ref{elementofDG}, there exists a full bisection 
$V\subset\cG$ containing $U$ such that $g:=\theta_V\in\sD(\cG)$. 
For each $i$, we have $g(\gamma_i z)=\gamma'_i z$, so 
$g(Z(\gamma_i))=Z(\gamma'_i)$, and $g$ also maps the complement 
of $\bigsqcup_iZ(\gamma_i)$ onto that of $\bigsqcup_iZ(\gamma'_i)$. 
Consequently, conjugation by $g$ carries each even permutation 
of the first family of prefixes to the same permutation of the second. 
This proves the last assertion. 
\end{proof}

We now fix a subgroup $H$ with $\sD(\cG)\subset H\subset\sF(\cH)$. 
The next definition records when, after every sufficiently long 
common extension, representatives of the resulting 
$\sim_{\cG}$-classes can be linked by a three-cycle in $H$. 

\begin{definition}\label{def:approx} 
For $\mu,\nu\in\Omega$, we write $\mu\approx\nu$ if 
$t(\mu)=t(\nu)$ and there exists $m\geq0$ such that the 
following holds for every $\eta\in\Omega$ with 
$t(\mu)=i(\eta)$ and $\lvert\eta\rvert\geq m$: 
there are pairwise incomparable finite paths 
$\gamma_0,\gamma_1,\gamma_2$ satisfying 
\[ 
\gamma_0\sim_{\cG}\mu\eta,\qquad 
\gamma_1\sim_{\cG}\gamma_2\sim_{\cG}\nu\eta, 
\] 
\[ 
Z(\gamma_0)\sqcup Z(\gamma_1)\sqcup Z(\gamma_2)\neq X, 
\qquad\sA(\gamma_0,\gamma_1,\gamma_2)\subset H. 
\] 
\end{definition} 

The bound $m$ may depend on $\mu$ and $\nu$. 
The condition is required for all $\lvert\eta\rvert\geq m$, 
and no length restriction is imposed on the $\gamma_i$. 
By Lemma~\ref{lem:prefixtransport}, once the condition holds 
for one triple representing the specified classes, it holds 
for every such triple with nonempty complement.

\begin{proposition}\label{prop:approx} 
The relation $\approx$ is an equivalence relation on $\Omega$ 
containing $\sim_{\cG}$. 
Moreover, if $t(\mu)=t(\nu)$, then 
\[ 
\mu\approx\nu 
\iff 
\mu e\approx\nu e\quad\text{for every }e\in t(\mu)E. 
\] 
\end{proposition} 

\begin{proof} 
For each $k\in\{3,4,5\}$, choose mutually disjoint sets 
$U_0,\dots,U_{k-1}\in\cCO(X)$ whose union is not $X$. 
Lemma~\ref{lem:prefixplacement} provides a bound $L_k$ such 
that any $k$ finite paths of length at least $L_k$ have 
pairwise incomparable representatives in their respective 
$\sim_{\cG}$-classes, with nonempty complement. 
Repetitions among the prescribed paths are allowed. 

Suppose first that $\mu\sim_{\cG}\nu$. 
For every common extension $\eta$, we have 
$\mu\eta\sim_{\cG}\nu\eta$ by 
Lemma~\ref{lem:subgroupoidrelation} (1). 
If $\lvert\eta\rvert\geq L_3$, choose representatives 
$\gamma_0,\gamma_1,\gamma_2$ for 
$(\mu\eta,\nu\eta,\nu\eta)$ as above. 
They are all $\sim_{\cG}$-equivalent, so 
$\sA(\gamma_0,\gamma_1,\gamma_2)\subset\sD(\cG)\subset H$. 
Thus $\mu\approx\nu$. 
This also proves reflexivity. 

To prove symmetry, suppose that $\mu\approx\nu$, and take 
$\eta$ sufficiently long for Definition~\ref{def:approx} 
to apply and for $\lvert\eta\rvert\geq L_4$. 
Choose pairwise incomparable paths $\gamma_0,\dots,\gamma_3$ 
with nonempty complement such that 
\[ 
\gamma_0\sim_{\cG}\gamma_1\sim_{\cG}\mu\eta, 
\qquad 
\gamma_2\sim_{\cG}\gamma_3\sim_{\cG}\nu\eta. 
\] 
By Lemma~\ref{lem:prefixtransport}, 
\[ 
\sA(\gamma_0,\gamma_2,\gamma_3)\subset H, 
\qquad 
\sA(\gamma_1,\gamma_2,\gamma_3)\subset H. 
\] 
Let $g$ be the element corresponding to the cycle $(0\ 2\ 3)$ 
on these four paths. 
It lies in the first group and fixes $Z(\gamma_1)$ pointwise. 
Hence 
\[ 
g\,\sA(\gamma_1,\gamma_2,\gamma_3)g^{-1} 
=\sA(\gamma_1,\gamma_3,\gamma_0)\subset H. 
\] 
Reordering the last triple gives a witness for $\nu\approx\mu$. 
Since the argument applies to every sufficiently long $\eta$, 
symmetry follows. 

For transitivity, suppose that $\mu\approx\nu$ and 
$\nu\approx\lambda$. 
Take $\eta$ sufficiently long for both relations and for 
$\lvert\eta\rvert\geq L_5$, and choose pairwise incomparable 
paths $\gamma_0,\dots,\gamma_4$ with nonempty complement such that 
\[ 
\gamma_0\sim_{\cG}\mu\eta,\qquad 
\gamma_1\sim_{\cG}\gamma_2\sim_{\cG}\nu\eta,\qquad 
\gamma_3\sim_{\cG}\gamma_4\sim_{\cG}\lambda\eta. 
\] 
By Lemma~\ref{lem:prefixtransport}, the three groups 
\[ 
\sA(\gamma_0,\gamma_1,\gamma_2),\qquad 
\sA(\gamma_1,\gamma_3,\gamma_4),\qquad 
\sA(\gamma_2,\gamma_3,\gamma_4) 
\] 
are contained in $H$. 
The product of the elements corresponding to 
$(1\ 3\ 4)$ and $(2\ 3\ 4)$ is an element $g\in H$ 
corresponding to $(1\ 3\ 4)(2\ 3\ 4)=(1\ 3)(2\ 4)$. 
Consequently 
\[ 
g\,\sA(\gamma_0,\gamma_1,\gamma_2)g^{-1} 
=\sA(\gamma_0,\gamma_3,\gamma_4)\subset H. 
\] 
This gives a witness for $\mu\approx\lambda$ for every 
sufficiently long $\eta$, proving transitivity. 

Finally, suppose that $\mu\approx\nu$, with bound $m$ in 
Definition~\ref{def:approx}. 
If $e\in t(\mu)E$ and $i(\eta)=t(e)$ with 
$\lvert\eta\rvert\geq m$, apply that definition to $e\eta$. 
The resulting triple witnesses the condition for 
$\mu e\approx\nu e$. 
Conversely, suppose that $\mu e\approx\nu e$ for every 
$e\in t(\mu)E$, and let $m_e$ be the corresponding bounds. 
Set $m:=1+\max_{e\in t(\mu)E}m_e$. 
Every path $\eta$ with $i(\eta)=t(\mu)$ and 
$\lvert\eta\rvert\geq m$ has the form $e\zeta$, where 
$\lvert\zeta\rvert\geq m_e$. 
The condition for $\mu e\approx\nu e$, applied to $\zeta$, 
therefore gives the required triple for $\mu\approx\nu$. 
\end{proof} 

Since $\sim_{\cG}$ is contained in $\approx$, the latter induces 
an equivalence relation on $\Delta_{\cG}^0=\Omega/{\sim_{\cG}}$.

\begin{lemma}\label{lem:approxlocal} 
Let $h\in H$ fix a nonempty open subset of $X$ pointwise. 
Then $h(\mu)\approx\mu$ for every $\mu\in\Omega(h)$. 
\end{lemma} 

\begin{proof} 
Choose $U\in\cCO(X)$ which is fixed pointwise by $h$. 
By Proposition~\ref{prop:approx} and 
Lemma~\ref{lem:subgroupoidpartialaction} (2), 
it suffices to prove the assertion after replacing $\mu$ 
by all its extensions of some fixed length. 
We may therefore assume that $Z(\mu)\subset U$ or 
$Z(\mu)\cap U=\emptyset$. 
In the first case, $Z(h(\mu),\mu)$ is an open subset of 
$\Iso(\cH)$. Since $\cH$ is effective, this bisection is contained 
in $X$, so $h(\mu)=\mu$. 

Suppose that $Z(\mu)\cap U=\emptyset$. 
Since $h$ fixes $U$ pointwise, it preserves $X\setminus U$. 
Choose disjoint nonempty clopen subsets $U_1,U_2\subset U$ 
with $U_1\cup U_2\neq U$. 
For every sufficiently long path $\eta$ with $i(\eta)=t(\mu)$, 
Lemma~\ref{lem:prefixplacement} provides paths $\gamma_1,\gamma_2$ 
such that 
\[ 
\gamma_1\sim_{\cG}\gamma_2\sim_{\cG}\mu\eta, 
\qquad Z(\gamma_j)\subset U_j\quad(j=1,2). 
\] 
The paths $\mu\eta,\gamma_1,\gamma_2$ are pairwise incomparable, 
and their basic alternating group is contained in $\sD(\cG)$. 
Conjugating by $h$ and using $h(\mu\eta)=h(\mu)\eta$, we obtain 
\[ 
h\,\sA(\mu\eta,\gamma_1,\gamma_2)h^{-1} 
=\sA(h(\mu)\eta,\gamma_1,\gamma_2)\subset H. 
\] 
The support on the right is disjoint from 
$U\setminus(U_1\cup U_2)$. 
Thus these triples witness $h(\mu)\approx\mu$. 
\end{proof}

We now use primeness to show that the intermediate subgroupoid 
defined by $\approx$ is all of $\cH$ whenever $H\not\subset N$. 

\begin{proposition}\label{prop:approxprime} 
Suppose that $\cG\subsetneq\cH$ is prime and 
$H\not\subset\Nor\left(\sF(\cH),\sF(\cG)\right)$. 
Then, for all $\mu,\nu\in\Omega$, 
\[ 
\mu\approx\nu\iff t(\mu)=t(\nu). 
\] 
\end{proposition} 

\begin{proof} 
In the quotient of $\Delta_{\cG}$ by $\approx$, the vertex 
$[\mu]_{\approx}$ lies over $t(\mu)$ and has exactly one outgoing 
edge labelled $e$ for each $e\in t(\mu)E$, ending at 
$[\mu e]_{\approx}$. 
These assignments are well defined by Definition~\ref{def:approx} 
and the forward implication in Proposition~\ref{prop:approx}, 
so the projection to $(V,E)$ is right-covering. 
The reverse implication gives reducedness. 
The root section is $v\mapsto[v]_{\approx}$, and every vertex 
$[\mu]_{\approx}$ is reached by lifting $\mu$ from its root. 
By Proposition~\ref{prop:subgroupoidgraphquotient}, it therefore 
determines a wide open subgroupoid 
\[ 
\cG\subset\cK\subset\cH, 
\qquad 
\mu\sim_{\cK}\nu\iff\mu\approx\nu. 
\] 
Lemma~\ref{lem:approxlocal} and 
Proposition~\ref{prop:subgroupoidnormalizer} (1) imply that 
every element of $H$ fixing a nonempty open subset pointwise 
belongs to $\sF(\cK)$. 

Let $\Sigma$ be the set of all $\sigma\in\sD(\cG)$ 
which fix a nonempty open subset pointwise. 
The full bisections $U_\sigma$, $\sigma\in\Sigma$, generate $\cG$. 
Indeed, let $g\in\cG$ have distinct source and range. 
Choosing a third point in its $\cG$-orbit, we obtain bisections 
$U,V\in\cCO(\cG)$ with $g\in U$ and $s(U)=s(V)$ such that 
$s(U)$, $r(U)$ and $r(V)$ are mutually disjoint and their union 
is not all of $X$. 
By Lemma~\ref{threecycle} (1) and (2), applied to $\cG$, 
we have $\tau_V\tau_U\in\Sigma$ and $g\in U_{\tau_V\tau_U}$. 
Every isotropy arrow is a product of two arrows with distinct 
source and range. 

For $h\in H$ and $\sigma\in\Sigma$, 
the element $h\sigma h^{-1}\in H$ fixes a nonempty open subset, 
so it belongs to $\sF(\cK)$. 
Consequently 
\[ 
U_h\cG U_h^{-1}\subset\cK\qquad(h\in H). 
\] 
If $\cK=\cG$, applying this inclusion to both $h$ and $h^{-1}$ 
would give $H\subset N$, contrary to the assumption. 
Hence $\cG\subsetneq\cK$, and primeness gives $\cK=\cH$. 
The asserted description of $\approx$ follows. 
\end{proof}

\begin{theorem}\label{thm:minimalovergroup} 
Suppose that $\cG\subsetneq\cH$ is prime. 
If $\sD(\cG)\subset H\subset\sF(\cH)$ and 
$H\not\subset\Nor\left(\sF(\cH),\sF(\cG)\right)$, 
then $\sD(\cH)\subset H$. 
\end{theorem} 

\begin{proof} 
By the generation statement recalled above, it suffices to prove 
\[ 
\sA(\mu_0,\mu_1,\mu_2)\subset H 
\] 
for pairwise incomparable paths $\mu_0,\mu_1,\mu_2$ 
with common terminal vertex $v$ and 
$\bigsqcup_{i=0}^2Z(\mu_i)\neq X$. 

Choose mutually disjoint nonempty clopen sets $U_1,U_2,U_3$ 
in the complement of $\bigsqcup_{i=0}^2Z(\mu_i)$. 
For every sufficiently long path $\eta$ with $i(\eta)=v$, 
Lemma~\ref{lem:prefixplacement} provides $\gamma_3,\gamma_4$ 
such that 
\[ 
\gamma_3\sim_{\cG}\gamma_4\sim_{\cG}\mu_2\eta, 
\qquad Z(\gamma_3)\subset U_1,\quad Z(\gamma_4)\subset U_2. 
\] 
Put $\gamma_i:=\mu_i\eta$ for $i=0,1,2$. 
These five paths are pairwise incomparable, and the union 
of their cylinders is disjoint from $U_3$. 

Proposition~\ref{prop:approxprime} gives 
$\mu_0\approx\mu_2$ and $\mu_1\approx\mu_2$. 
For every sufficiently long $\eta$, 
Definition~\ref{def:approx} and Lemma~\ref{lem:prefixtransport} 
therefore give 
\[ 
\sA(\gamma_0,\gamma_3,\gamma_4)\subset H, 
\qquad 
\sA(\gamma_1,\gamma_3,\gamma_4)\subset H. 
\] 
Moreover, 
$\sA(\gamma_2,\gamma_3,\gamma_4)\subset\sD(\cG)\subset H$. 
The same calculation as in the proof of Proposition~\ref{prop:approx} 
gives an element $g\in H$ acting as 
$(1\ 3\ 4)(2\ 3\ 4)=(1\ 3)(2\ 4)$. 
Hence 
\[ 
g\,\sA(\gamma_0,\gamma_3,\gamma_4)g^{-1} 
=\sA(\gamma_0,\gamma_1,\gamma_2)\subset H. 
\] 

Choose $m$ large enough that this holds for every $\eta$ 
with $i(\eta)=v$ and $\lvert\eta\rvert=m$. 
Let $c$ be the basic three-cycle on $(\mu_0,\mu_1,\mu_2)$ 
and let $c_\eta$ be the corresponding three-cycle on 
$(\mu_0\eta,\mu_1\eta,\mu_2\eta)$. 
The latter have mutually disjoint supports, and 
Lemma~\ref{lem:basicbisection} (2) gives 
\[ 
c=\prod_{\substack{i(\eta)=v\\\lvert\eta\rvert=m}}c_\eta\in H. 
\] 
This proves $\sA(\mu_0,\mu_1,\mu_2)\subset H$ 
and completes the proof. 
\end{proof}

We now describe the normalizer and deduce maximality criteria. 
Put 
\[ 
\Out_{\cH}(\cG):=N/\sF(\cG). 
\] 
For $h\in N$, conjugation by $U_h$ restricts to an automorphism 
of $\cG$. Its class modulo automorphisms induced by full bisections 
of $\cG$ is trivial exactly when $h\in\sF(\cG)$: equality of two 
such automorphisms implies equality of their restrictions to $X$. 
Thus $\Out_{\cH}(\cG)$ identifies with the subgroup of outer automorphisms of 
$\cG$ induced by elements of $\sF(\cH)$. 
Since $\sD(\cG)$ is characteristic in $\sF(\cG)$, it is normal 
in $N$, and the quotient $N/\sD(\cG)$ fits into the exact sequence 
\begin{equation}\label{eq:minimalnormalizerquotient} 
\xymatrix@M=8pt{ 
1 \ar[r] & H_1(\sF(\cG)) \ar[r] & N/\sD(\cG) \ar[r] & 
\Out_{\cH}(\cG) \ar[r] & 1. 
} 
\end{equation} 
Here $H_1(\sF(\cG))=\sF(\cG)/\sD(\cG)$ is the group abelianization. 

Let 
\[ 
\pi:\sF(\cH)\longrightarrow H_1(\sF(\cH)) 
\] 
be the abelianization homomorphism. 
Since $\sD(\cG)\subset\sD(\cH)=\Ker\pi$, its restriction to $N$ 
induces a homomorphism 
\[ 
\bar\pi:N/\sD(\cG)\longrightarrow H_1(\sF(\cH)). 
\] 
On the subgroup $H_1(\sF(\cG))$ in 
\eqref{eq:minimalnormalizerquotient}, this is the homomorphism 
induced by the inclusion $\sF(\cG)\subset\sF(\cH)$.

\begin{theorem}\label{thm:minimalnormalizermaximal} 
Suppose that $\cG\subsetneq\cH$ is prime, and let 
$\sD(\cH)\subset\Gamma\subset\sF(\cH)$. 
Put $N_\Gamma:=\Gamma\cap\Nor\left(\sF(\cH),\sF(\cG)\right)$. 
Then $N_\Gamma$ is a proper subgroup of $\Gamma$, and the following 
conditions are equivalent: 
\begin{enumerate} 
\item $N_\Gamma$ is a maximal subgroup of $\Gamma$. 
\item $N_\Gamma\sD(\cH)=\Gamma$. 
\item $\pi(N_\Gamma)=\pi(\Gamma)$. 
\end{enumerate} 
\end{theorem} 

\begin{proof} 
We first show that $\sD(\cH)\not\subset N$. 
By Lemma~\ref{threecycle} (3), $\sD(\cG)$ acts minimally on $X$ 
and is therefore nontrivial. 
If $\sD(\cH)\subset N$, the subgroup $\sD(\cG)\subset\sD(\cH)$ 
would be normal in $\sD(\cH)$. 
By Theorem~\ref{niceproperties} (1), this would give 
$\sD(\cG)=\sD(\cH)$. 
By \cite[Lemma 4.9 (2)]{NO19IJM}, each of $\cG$ and $\cH$ is 
covered by the full bisections associated with its commutator 
subgroup. 
The equality of these subgroups would therefore imply 
$\cG=\cH$, a contradiction. 
Thus $N_\Gamma$ is proper in $\Gamma$, and 
\[ 
N_\Gamma\subsetneq N_\Gamma\sD(\cH)\subset\Gamma. 
\] 
Consequently (1) implies (2). 

Suppose (2) holds, and let $N_\Gamma\subsetneq L\subset\Gamma$. 
Then $\sD(\cG)\subset L$ and $L\not\subset N$, so 
Theorem~\ref{thm:minimalovergroup} gives $\sD(\cH)\subset L$. 
Hence $\Gamma=N_\Gamma\sD(\cH)\subset L$, proving (1). 
Finally, (2) and (3) are equivalent because 
$\Ker\pi=\sD(\cH)\subset\Gamma$. 
\end{proof}

The theorem separates the structural and homological parts of 
maximality, as in \cite[Sections 3 and 5]{Ma26maximal1}. 
The normalizer in $\sD(\cH)$ is maximal without an additional 
homological condition; for a larger intermediate group $\Gamma$, 
maximality is governed by its image in the ambient abelianization. 
In the nonminimal cases of Theorem~\ref{thm:primetrichotomy}, 
this image condition is automatic, as for the stabilizers in 
\cite[Remark 3.16]{Ma26maximal2}; see 
Corollary~\ref{cor:primenormalizermaximal}. 
Passing from the normalizer to $\sF(\cG)$ itself also requires 
the outer quotient to be trivial, as the next corollary shows. 

The two extreme choices of $\Gamma$ give the following consequences. 

\begin{corollary}\label{cor:minimalfullmaximal} 
Suppose that $\cG\subsetneq\cH$ is prime. 
\begin{enumerate} 
\item The group $\Nor\left(\sF(\cH),\sF(\cG)\right)$ is maximal 
in $\sF(\cH)$ if and only if 
\[ 
\pi\left(\Nor\left(\sF(\cH),\sF(\cG)\right)\right) 
=H_1(\sF(\cH)). 
\] 
\item $\sF(\cG)$ is a maximal subgroup of $\sF(\cH)$ if and only if 
$\Out_{\cH}(\cG)$ is trivial and either (and hence both) of the maps 
induced by inclusion 
\[ 
H_1(\sF(\cG))\longrightarrow H_1(\sF(\cH)), 
\qquad H_1(\cG)\longrightarrow H_1(\cH) 
\] 
is surjective. 
When these conditions hold, $\Gamma\cap\sF(\cG)$ is a maximal 
subgroup of $\Gamma$ for every 
$\sD(\cH)\subset\Gamma\subset\sF(\cH)$. 
\end{enumerate} 
\end{corollary} 

\begin{proof} 
(1)\: 
This is Theorem~\ref{thm:minimalnormalizermaximal} 
with $\Gamma=\sF(\cH)$. 

(2)\: 
We first compare the two surjectivity conditions. 
Theorem~\ref{Li}, naturality of the index map, and the description 
of $\zeta$ by transpositions give a commutative diagram 
with exact rows: 
\[ 
\xymatrix@M=8pt{ 
H_0(\cG,\Z/2) \ar[r] \ar[d] & H_1(\sF(\cG)) \ar[r] \ar[d] & 
H_1(\cG) \ar[r] \ar[d] & 0 \\ 
H_0(\cH,\Z/2) \ar[r] & H_1(\sF(\cH)) \ar[r] & 
H_1(\cH) \ar[r] & 0. 
} 
\] 
All vertical maps are induced by inclusion. 
The left vertical map is surjective because $\cG$ is wide: 
the classes $[1_B]$ of clopen subsets $B\subset X$ generate 
$H_0(\cH,\Z/2)$, and each is the image of the same class for $\cG$. 
A diagram chase therefore shows that the middle vertical map 
is surjective if and only if the right vertical map is surjective. 

Suppose first that $\sF(\cG)$ is maximal in $\sF(\cH)$. 
Since $\sF(\cG)\subset N\subsetneq\sF(\cH)$, we have 
$N=\sF(\cG)$, so $\Out_{\cH}(\cG)$ is trivial. 
Assertion (1) now gives surjectivity on abelianizations, and hence 
on first groupoid homology. 

Conversely, suppose that $\Out_{\cH}(\cG)$ is trivial and the map 
on first groupoid homology is surjective. 
Then $N=\sF(\cG)$ and the map on abelianizations is surjective, 
so (1) implies that $\sF(\cG)$ is maximal in $\sF(\cH)$. 

Finally, surjectivity gives $\sF(\cG)\sD(\cH)=\sF(\cH)$. 
For any $\Gamma$ in the statement, this implies 
\[ 
\left(\Gamma\cap\sF(\cG)\right)\sD(\cH)=\Gamma. 
\] 
Since $N=\sF(\cG)$, Theorem~\ref{thm:minimalnormalizermaximal} 
proves the last assertion. 
\end{proof}

\begin{corollary}\label{cor:minimalcommutatormaximal} 
Suppose that $\cG\subsetneq\cH$ is prime. 
\begin{enumerate} 
\item $\Nor\left(\sD(\cH),\sD(\cG)\right)$ is a maximal subgroup 
of $\sD(\cH)$. 
\item $\sD(\cG)$ is a maximal subgroup of $\sD(\cH)$ if and only if 
the homomorphism 
\[ 
\bar\pi:\Nor\left(\sF(\cH),\sF(\cG)\right)/\sD(\cG) 
\longrightarrow H_1(\sF(\cH)) 
\] 
is injective. 
\end{enumerate} 
\end{corollary} 

\begin{proof} 
By Lemma~\ref{lem:groupnormalizers}, 
$N\cap\sD(\cH)=\Nor\left(\sD(\cH),\sD(\cG)\right)$. 
For $\Gamma=\sD(\cH)$, condition (3) of 
Theorem~\ref{thm:minimalnormalizermaximal} holds automatically, 
proving (1). 
Since 
\[ 
\sD(\cG)\subset N\cap\sD(\cH)\subsetneq\sD(\cH), 
\] 
assertion (1) shows that $\sD(\cG)$ is maximal exactly when 
$N\cap\sD(\cH)=\sD(\cG)$. 
This is equivalent to injectivity of $\bar\pi$, 
since $\Ker\bar\pi=\left(N\cap\sD(\cH)\right)/\sD(\cG)$. 
\end{proof}

We finally drop the standing assumption that $\cG$ is minimal and 
combine the results of Sections~\ref{sec:trichotomy}, 
\ref{sec:exceptional} and~\ref{sec:minimal}. 

\begin{corollary}\label{cor:primenormalizermaximal} 
Let $\cG\subsetneq\cH$ be a prime wide open subgroupoid, and let 
$\Gamma$ be a group with $\sD(\cH)\subset\Gamma\subset\sF(\cH)$. 
Put $N_\Gamma:=\Gamma\cap\Nor\left(\sF(\cH),\sF(\cG)\right)$, 
and let $\pi:\sF(\cH)\to H_1(\sF(\cH))$ be the abelianization map. 
Then $N_\Gamma$ is a proper subgroup of $\Gamma$, and the following 
conditions are equivalent: 
\begin{enumerate} 
\item $N_\Gamma$ is a maximal subgroup of $\Gamma$. 
\item $N_\Gamma\sD(\cH)=\Gamma$. 
\item $\pi(N_\Gamma)=\pi(\Gamma)$. 
\end{enumerate} 
If $\cG$ is not minimal, these conditions hold automatically. 
In particular, 
$\sD(\cH)\cap\Nor\left(\sF(\cH),\sF(\cG)\right)$ 
is a maximal subgroup of $\sD(\cH)$. 
\end{corollary} 

\begin{proof} 
If $\cG$ is minimal, the assertions follow from 
Theorem~\ref{thm:minimalnormalizermaximal}. 
Otherwise, case \textup{(2)} or \textup{(3)} of 
Theorem~\ref{thm:primetrichotomy} holds. 
In case \textup{(2)}, Proposition~\ref{prop:clopenpartitionnormalizer} 
with $n=2$ shows that $N_\Gamma$ is a proper maximal subgroup. 
In case \textup{(3)}, $\cG$ is the $Y$-preserving subgroupoid, 
as recalled at the beginning of Section~\ref{sec:exceptional}, 
so Theorem~\ref{thm:exceptionalstabilizer} gives the same conclusion. 

In case \textup{(2)}, choose a nonempty clopen set $W\subset X_1$; 
in case \textup{(3)}, choose one in $X\setminus Y$. 
In either case, 
$\sF(\cH)_W\subset\sF(\cG)\subset\Nor\left(\sF(\cH),\sF(\cG)\right)$. 
Since $\sD(\cH)\subset\Gamma$, Lemma~\ref{localabelianization} gives 
$N_\Gamma\sD(\cH)=\Gamma$. 
This is equivalent to $\pi(N_\Gamma)=\pi(\Gamma)$ because 
$\Ker\pi=\sD(\cH)$. 
Thus all three conditions hold in the nonminimal cases. 
The last assertion follows by taking $\Gamma=\sD(\cH)$. 
\end{proof}

We now describe the fully amplified maximal subgroups of $\sF(\cH)$. 

\begin{corollary}\label{cor:amplemaximal} 
A subgroup $M\subset\sF(\cH)$ is fully amplified and maximal 
in $\sF(\cH)$ if and only if it has one of the following forms. 
\begin{enumerate} 
\renewcommand{\theenumi}{\alph{enumi}} 
\item $M=\St(\cH,A)$ for a nonempty proper clopen set $A\subset X$ 
such that $2[1_A]\neq[1_X]$ in $H_0(\cH)$. 
\item $M=\St(\cH,Y)$ for a nonempty closed nowhere dense set 
$Y\subset X$ such that the regular reduction $\cH|Y$ is minimal. 
\item $M=\sF(\cG)$ for a minimal prime wide open subgroupoid 
$\cG\subsetneq\cH$ with $\Out_{\cH}(\cG)=1$ such that the map 
$H_1(\cG)\to H_1(\cH)$ induced by inclusion is surjective. 
\end{enumerate} 
These three alternatives are mutually exclusive. 
\end{corollary} 

\begin{proof} 
Let $M$ be a fully amplified maximal subgroup and put $\cG:=\cG_M$. 
Then $M=\sF(\cG)$, so $\cG\neq\cH$, and 
Proposition~\ref{prop:maximalgerms} shows that $\cG\subset\cH$ 
is prime. 
By Corollary~\ref{cor:primenormalizermaximal}, the normalizer 
$N:=\Nor\left(\sF(\cH),M\right)$ is proper in $\sF(\cH)$. 
Since $M\subset N$, maximality gives $M=N$. 

In case \textup{(2)} of Theorem~\ref{thm:primetrichotomy}, 
we have $M=\St(\cH,X_1)$. 
Since elements of $\sF(\cH)$ preserve the classes of clopen sets 
in $H_0(\cH)$, Theorem~\ref{niceproperties} (2) shows that an element 
exchanging $X_1$ and $X_2$ exists if and only if 
$[1_{X_1}]=[1_{X_2}]$. 
Such an element would belong to $N\setminus M$, so these classes 
are distinct, or equivalently $2[1_{X_1}]\neq[1_X]$. 
In case \textup{(3)}, $\cG$ is the $Y$-preserving subgroupoid, 
as recalled at the beginning of Section~\ref{sec:exceptional}, 
so Theorem~\ref{thm:exceptionalstabilizer} gives $M=\St(\cH,Y)$. 
In case \textup{(1)}, Corollary~\ref{cor:minimalfullmaximal} (2) 
gives the conditions in \textup{(c)} of the present statement. 

Conversely, a subgroup in \textup{(a)} is the normalizer of 
$\sF(\cH|A\sqcup\cH|(X\setminus A))$, since the homology condition 
excludes an element exchanging the two parts. 
It is maximal by Proposition~\ref{prop:clopenpartitionnormalizer}. 
Maximality in \textup{(b)} and \textup{(c)} follows from 
Theorem~\ref{thm:exceptionalstabilizer} and 
Corollary~\ref{cor:minimalfullmaximal} (2), respectively. 
All these groups are fully amplified, being topological full groups 
of wide open subgroupoids. 
Finally, their actions on $X$ have respectively two, one and no 
nonempty proper closed invariant subsets, so the alternatives 
are mutually exclusive. 
\end{proof} 

The groups in case \textup{(c)} act minimally on $X$. 
A maximal subgroup which is not fully amplified has full amplification 
$\sF(\cH)$ and therefore also acts minimally on $X$, since 
full amplification does not change orbits. 
Consequently, cases \textup{(a)} and \textup{(b)} are exactly the maximal 
subgroups of $\sF(\cH)$ which do not act minimally on $X$, in accordance 
with \cite[Theorems 6.6 and 6.19]{GV26ETDS} and \cite{Ma26maximal2}.

%%%%%%%%%%%%%%%%%%%%%%%%%%%%%%%%%%%%%%%%%%%%%%%%%%%%%%%%%%%%
\section{Uncountably many maximal subgroups}\label{sec:uncountable} 
%%%%%%%%%%%%%%%%%%%%%%%%%%%%%%%%%%%%%%%%%%%%%%%%%%%%%%%%%%%%

The maximal subgroups in \cite[Proposition 4.16]{Ma26maximal2} 
stabilize nowhere dense Cantor sets. 
The groups constructed in this section act minimally on the unit space. 
In $\sF(\cH)$ itself, the resulting maximal subgroups are fully 
amplified and belong to case \textup{(c)} of 
Corollary~\ref{cor:amplemaximal}. 
We first describe the construction of 
\cite[Theorem 7.5]{BBQS25TAMS} in terms of subgroupoid graphs. 
We then modify the ambient graph to obtain uncountably many 
pairwise nonisomorphic maximal subgroups for every irreducible SFT 
groupoid considered in this paper. 

Let $\mathbf p=(p_n)_{n\geq1}$ be a strictly increasing sequence 
of primes. 
Consider the graph $\Lambda_{\mathbf p}$ with vertex set 
\[ 
W_{\mathbf p}:=\{\mathbf X,\mathbf Y\}\sqcup 
\{\mathbf Z_{n,k}\mid n\geq0,\ 0\leq k\leq p_{n+1}\}. 
\] 
There are two outgoing edges at each vertex, labelled $0$ and $1$. 
Writing $Pj$ for the terminal vertex of the edge labelled $j$ 
with initial vertex $P$, the edges are specified by 
\begin{equation}\label{eq:BBQSgraph} 
\begin{aligned} 
\mathbf X0&=\mathbf Z_{0,0},& 
\mathbf X1&=\mathbf Y,& 
\mathbf Y0&=\mathbf X,& 
\mathbf Y1&=\mathbf Y,\\ 
\mathbf Z_{n,k}0&=\mathbf Z_{n+1,0},& 
\mathbf Z_{n,k}1&=\mathbf Z_{n,k+1} 
&&&(0\leq k<p_{n+1}),\\ 
\mathbf Z_{n,p_{n+1}}0&=\mathbf X,& 
\mathbf Z_{n,p_{n+1}}1&=\mathbf Z_{n,0}. 
\end{aligned} 
\end{equation} 
Thus the vertices $\mathbf Z_{n,k}$ in the $n$th block form 
a $1$-labelled cycle of length $p_{n+1}+1$. 
Its $0$-labelled edges lead to the first vertex of the next block, 
except for the edge from its last vertex, which returns to $\mathbf X$. 
The root is $\mathbf X$, and the edge labels give a right-covering 
projection onto the graph with one vertex and two loops. 

The graph $\Lambda_{\mathbf p}$ is reduced in the sense of 
\eqref{eq:subgroupoidgraphreduced}: the only distinct 
vertices with the same $1$-successor are $\mathbf X$ and $\mathbf Y$, 
and their $0$-successors differ. 
It is strongly connected, since every vertex can reach $\mathbf X$ 
and every vertex is reachable from $\mathbf X$. 
Let $\cH_2$ denote the SFT groupoid of the full shift on two symbols. 
Propositions~\ref{prop:subgroupoidgraphconverse} 
and~\ref{prop:subgroupoidgraphmodel} give a proper wide open subgroupoid 
$\cK_{\mathbf p}\subset\cH_2$ satisfying 
\[ 
\cK_{\mathbf p}\cong\cG_{\Lambda_{\mathbf p}}|Z(\mathbf X). 
\] 
By Lemma~\ref{lem:graphpureinfiniteness}, $\cK_{\mathbf p}$ 
is minimal and purely infinite. 

We record the two combinatorial features of $\Lambda_{\mathbf p}$ 
which will also be used below. 
First, the only equivalence relations $\equiv$ on $W_{\mathbf p}$ 
satisfying 
\[ 
P\equiv P'\implies 
P0\equiv P'0\ \text{ and }\ P1\equiv P'1 
\] 
are equality and the universal relation. 
To see this, suppose first that $\mathbf X\equiv\mathbf Y$. 
Applying $0$ and then powers of $1$ identifies every vertex in the 
zeroth block with $\mathbf X$, and induction does the same 
for all subsequent blocks. 
It therefore remains to show that any nontrivial identification 
implies $\mathbf X\equiv\mathbf Y$. 
If $\mathbf X$ is identified with a block vertex, applying a full 
turn of its $1$-cycle gives $\mathbf X\equiv\mathbf Y$. 
If $\mathbf Y$ is identified with a block vertex, powers of $1$ 
and then $0$ identify $\mathbf X$ with a block vertex. 
Finally, two distinct block vertices can be followed by a word 
$1^l0$ so that one reaches $\mathbf X$ and the other reaches 
the first vertex of a block. 
Indeed, the nonnegative times at which the two vertices reach 
the last vertices of their respective blocks form arithmetic 
progressions. 
These progressions have different common differences for distinct 
blocks, and different residues modulo the common cycle length 
for distinct vertices in the same block. 
They are therefore unequal, and we choose $l$ in their symmetric 
difference, so that exactly one vertex reaches the last vertex 
before applying $0$. 
Thus any nontrivial identification implies 
$\mathbf X\equiv\mathbf Y$. 
By Corollary~\ref{cor:primegraph}, $\cK_{\mathbf p}\subset\cH_2$ is prime. 

Second, every label-preserving automorphism of $\Lambda_{\mathbf p}$ 
is the identity. 
The vertex $\mathbf Y$ is the unique vertex fixed by the transition $1$, 
so it is fixed, as is $\mathbf X=\mathbf Y0$. 
Since every vertex is reachable from $\mathbf X$, all vertices are fixed. 
In a strongly connected graph, repeating a nonempty closed path 
at a root before following a rooted path gives representatives 
of every vertex of arbitrarily large length. 
Consequently, by Lemma~\ref{lem:groupnormalizers} and 
Proposition~\ref{prop:subgroupoidnormalizer}, each 
$h\in\Nor\left(\sF(\cH_2),\sF(\cK_{\mathbf p})\right)$ 
induces an automorphism of $\Lambda_{\mathbf p}$ via 
$[\mu]_{\cK_{\mathbf p}}\mapsto[h(\mu)]_{\cK_{\mathbf p}}$. 
This is well-defined on all vertices because the domain of the partial 
action is cofinite; its range is also cofinite, which gives surjectivity. 
It preserves the edge labels because $h(\mu j)=h(\mu)j$. 
Rigidity and Proposition~\ref{prop:subgroupoidnormalizer} (1) therefore 
give 
\[ 
\Nor\left(\sF(\cH_2),\sF(\cK_{\mathbf p})\right)=\sF(\cK_{\mathbf p}). 
\] 

We compute the homology using Proposition~\ref{prop:graphalgebras} (4) 
and invariance under full open reduction. 
Let $B_{\mathbf p}$ be the adjacency matrix of $\Lambda_{\mathbf p}$. 
In the cokernel presentation, write $x,y,z_{n,k}$ for the vertex 
generators and put $a_n:=z_{n,0}$. 
The relations at $\mathbf Y$, $\mathbf X$ and the block vertices give 
\[ 
x=0,\qquad a_0=-y,\qquad 
z_{n,k}=a_n-ka_{n+1},\qquad p_{n+1}a_{n+1}=0. 
\] 
Eliminating the other generators leaves precisely these relations. 
Since the root is $\mathbf X$, we obtain 
\begin{equation}\label{eq:BBQShomology} 
H_0(\cK_{\mathbf p}) 
\cong\Z\oplus\bigoplus_{n\geq1}\Z/p_n, 
\qquad [1_{\cK_{\mathbf p}^{(0)}}]=0, 
\qquad H_1(\cK_{\mathbf p})=0. 
\end{equation} 
For the last equality, let $u=B_{\mathbf p}^tu$. 
The equation at $\mathbf Y$ gives $u_{\mathbf X}=0$, and 
the coordinates in the $n$th block have a common value $c_n$. 
The equation at $\mathbf Z_{n+1,0}$ then gives $p_{n+1}c_n=0$, 
so all $c_n$ vanish; the equation at $\mathbf X$ gives 
$u_{\mathbf Y}=0$. 

Since $H_1(\cH_2)=0$, the map 
$H_1(\cK_{\mathbf p})\to H_1(\cH_2)$ is surjective. 
Corollary~\ref{cor:minimalfullmaximal} (2), in its groupoid-homology 
form, therefore shows that $\sF(\cK_{\mathbf p})$ is maximal 
in $\sF(\cH_2)=V_2$. 
Different infinite sets of primes give nonisomorphic zeroth 
homology groups and hence nonisomorphic topological full groups by 
\cite[Theorem 3.10]{Ma15crelle}. 
This recovers the uncountable family in \cite[Theorem 7.5]{BBQS25TAMS}. 
These subgroups act minimally and therefore stabilize no nonempty 
finite subset of the unit space. 

We also recover the additional assertion in 
\cite[Theorem 7.5]{BBQS25TAMS} that these subgroups preserve no 
circular order on any dense invariant subset of the unit space. 
Indeed, choose a nontrivial three-cycle 
$g\in\sF(\cK_{\mathbf p})$ with proper clopen support. 
Any dense invariant subset $D$ meets both the support and its 
complement, so $g|D$ is nontrivial and fixes some $d\in D$. 
If $g|D$ preserved a circular order, taking $d$ as the initial 
point would give a $g$-invariant linear order on $D\setminus\{d\}$. 
An order-preserving permutation of a linearly ordered set with 
finite order is the identity: a strict inequality $x<g(x)$ 
or $g(x)<x$ contradicts $g^3(x)=x$. 
This contradicts the nontriviality of $g|D$. 

We now extend this construction to an arbitrary irreducible SFT 
groupoid. 
The additional vertices in the proof preserve the isomorphism class 
of the ambient groupoid, while separating its homology from the 
prime-dependent part of \eqref{eq:BBQShomology}. 

\begin{theorem}\label{thm:uncountableminimal} 
Let $(V,E)$ be a finite directed graph whose adjacency matrix is 
irreducible and is not a permutation matrix, and put 
$\cH:=\cG_{(V,E)}$. 
There exists a family $(\cG_t)_{t\in T}$, with $\#T=2^{\aleph_0}$, 
of proper prime minimal purely infinite wide open subgroupoids of 
$\cH$ such that, for every subgroup 
\[ 
\sD(\cH)\subset\Gamma\subset\sF(\cH), 
\] 
the groups $\Gamma\cap\sF(\cG_t)$, $t\in T$, are pairwise 
nonisomorphic maximal subgroups of $\Gamma$. 
In particular, this holds for $\Gamma=\sF(\cH)$ and 
$\Gamma=\sD(\cH)$. 
\end{theorem} 

\begin{proof} 
Let $A$ be the adjacency matrix of $(V,E)$. 
Fix $v\in V$ and let $\epsilon_v$ be its standard basis column vector. 
Adjoin three vertices $a,b,c$ and form the graph 
$(\widehat V,\widehat E)$ with adjacency matrix 
\begin{equation}\label{eq:uncountableambientmatrix} 
\widehat A:= 
\begin{bmatrix} 
A&0&\epsilon_v&\epsilon_v\\ 
0^t&2&1&1\\ 
0^t&0&2&1\\ 
\epsilon_v^t&1&0&0 
\end{bmatrix}, 
\qquad 
\widehat V=V\sqcup\{a,b,c\}. 
\end{equation} 
Thus $a$ has exactly two loops, which we denote by $e_0,e_1$. 
Even after deleting these two edges, the graph remains strongly 
connected; see Figure~\ref{fig:uncountableambientgraph}. 
Put $\widehat\cH:=\cG_{(\widehat V,\widehat E)}$. 

\begin{figure}[htbp] 
\centering 
\begin{tikzpicture}[ 
  x=1cm,y=1cm, 
  line width=0.55pt, 
  >={Stealth[length=4.6pt,width=3.5pt]}, 
  every node/.style={inner sep=3pt}, 
  every loop/.style={looseness=7} 
] 
  % The original graph is retained inside the box. 
  \draw[rounded corners=8pt] (-3.0,-1.15) rectangle (0.65,1.15); 
  \node[anchor=north west] at (-2.85,1.02) {$(V,E)$}; 
  \node[font=\Large] (v) at (0,0) {$v$}; 
  \node[font=\Large] (a) at (2.35,0) {$a$}; 
  \node[font=\Large] (b) at (4.70,0) {$b$}; 
  \node[font=\Large] (c) at (7.05,0) {$c$}; 

  % Single edges between the added vertices. 
  \draw[->] (a) -- (b); 
  \draw[->] (b) -- (c); 

  % The old vertex v is attached to b and c. 
  \draw[->] (v.north east) 
    .. controls (0.85,2.10) and (4.20,2.10) .. (b.north west); 
  \draw[->] (c.north) 
    .. controls (7.05,3.40) and (0,3.40) .. (v.north); 
  \draw[->] (v.south) 
    .. controls (0,-3.35) and (7.05,-3.35) .. (c.south); 

  % The two directions between a and c are shown separately. 
  \draw[->] (a.south east) 
    .. controls (3.05,-0.95) and (6.35,-0.95) .. (c.south west); 
  \draw[->] (c.south west) 
    .. controls (6.35,-1.95) and (2.35,-1.95) .. (a.south); 

  % Marked loops at a; a multiplicity label for the loops at b. 
  \draw[->] (a) to[out=145,in=100,loop] 
    node[above left,inner sep=1pt] {$e_0$} (a); 
  \draw[->] (a) to[out=80,in=35,loop] 
    node[above right,inner sep=1pt] {$e_1$} (a); 
  \draw[->] (b) to[out=80,in=35,loop] 
    node[above right,inner sep=1pt] {$2$} (b); 
  % Tight bounds around the curves rather than their control points. 
  \pgfresetboundingbox 
  \path[use as bounding box] (-3.02,-2.58) rectangle (7.25,2.62); 
\end{tikzpicture} 
\caption{The graph $(\widehat V,\widehat E)$. 
The box represents $(V,E)$, whose edges are omitted. 
All added edges are shown; the label $2$ denotes two loops.} 
\label{fig:uncountableambientgraph} 
\end{figure}
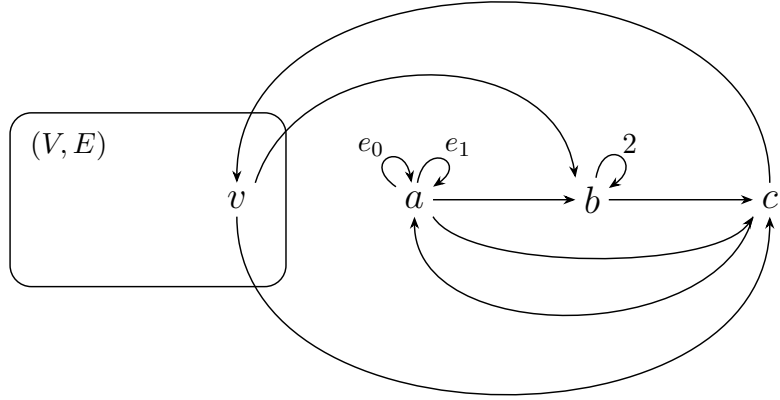 

We first verify that $\widehat\cH\cong\cH$. 
By Proposition~\ref{prop:graphalgebras} (4), the zeroth homology 
and unit classes are computed from the vertex generators. 
Use the vertex names also for their classes in 
$\Coker(I-\widehat A^t)$. 
The relations at the new vertices are 
\[ 
a=2a+b+c,\qquad b=2b+c,\qquad c=v+a, 
\] 
and hence $a=0$, $b=-v$ and $c=v$. 
The extra summands $b+c$ in the relation at $v$ cancel. 
Therefore the old vertex generators identify 
$\Coker(I-\widehat A^t)$ with $\Coker(I-A^t)$, and 
$a+b+c=0$ shows that this identification preserves the unit class. 
The bottom right block of $I-\widehat A$ is 
\[ 
D:= 
\begin{bmatrix} 
-1&-1&-1\\ 
0&-1&-1\\ 
-1&0&1 
\end{bmatrix}, 
\qquad 
\det D=1,\qquad 
D^{-1} 
\begin{bmatrix}0\\0\\1\end{bmatrix} 
=\begin{bmatrix}0\\-1\\1\end{bmatrix}. 
\] 
Its Schur complement is $I-A$, so 
$\det(I-\widehat A)=\det(I-A)$. 
The classification of irreducible SFT groupoids by Matsumoto and 
the author \cite[Theorem 3.6]{MM14Kyoto} now gives 
$\widehat\cH\cong\cH$. 
It therefore suffices to construct the required subgroupoids in 
$\widehat\cH$. 

Fix a prime sequence $\mathbf p$ as above. 
Define a graph $\Delta_{\mathbf p}$ with vertex set 
$W_{\mathbf p}\times\widehat V$ and right-covering projection 
$\rho_{\mathbf p}$ onto $(\widehat V,\widehat E)$ as follows. 
For each $P\in W_{\mathbf p}$, the lifts of $e_0,e_1$ are 
\[ 
(P,a)\longrightarrow(P0,a), 
\qquad 
(P,a)\longrightarrow(P1,a), 
\] 
respectively. 
For each $e\in\widehat E\setminus\{e_0,e_1\}$, its lift from 
$(P,i(e))$ goes to $(P,t(e))$. 
Take the root section 
$\varsigma_{\mathbf p}(w):=(\mathbf X,w)$ for $w\in\widehat V$. 
Paths avoiding $e_0,e_1$ connect all vertices with the same first 
coordinate. 
The strong connectivity of $\Lambda_{\mathbf p}$ then shows that 
$\Delta_{\mathbf p}$ is strongly connected. 
It is reduced, since every vertex of $\widehat V$ has an outgoing 
edge other than $e_0,e_1$, and such an edge preserves the first 
coordinate. 
Propositions~\ref{prop:subgroupoidgraphconverse} 
and~\ref{prop:subgroupoidgraphmodel} consequently give a proper wide open 
subgroupoid $\cG_{\mathbf p}\subset\widehat\cH$ and an isomorphism 
\begin{equation}\label{eq:uncountablegraphmodel} 
\cG_{\mathbf p}\cong\cG_{\Delta_{\mathbf p}}|S_{\mathbf p}, 
\qquad 
S_{\mathbf p}:=\bigsqcup_{w\in\widehat V}Z((\mathbf X,w)). 
\end{equation} 
Here $S_{\mathbf p}\in\cCO(X_{\Delta_{\mathbf p}})$ is full. 
Lemma~\ref{lem:graphpureinfiniteness} implies that 
$\cG_{\mathbf p}$ is minimal and purely infinite. 

We next prove primeness. 
An intermediate wide open subgroupoid determines an equivalence 
relation on $W_{\mathbf p}\times\widehat V$ compatible with the 
projection and right extensions. 
For $w\in\widehat V$, let $\equiv_w$ be the induced relation on 
$W_{\mathbf p}\times\{w\}$, regarded as a relation on $W_{\mathbf p}$. 
Paths avoiding $e_0,e_1$ in both directions show that 
$\equiv_w$ is independent of $w$. 
At $a$, the edges $e_0,e_1$ show that this common relation is 
preserved by the two transitions of $\Lambda_{\mathbf p}$. 
The combinatorial observation above implies that it is either 
equality or the universal relation. 
Thus the quotient is either $\Delta_{\mathbf p}$ itself or the 
ambient graph, proving primeness by Corollary~\ref{cor:primegraph}. 

Likewise, a graph automorphism of $\Delta_{\mathbf p}$ over 
$(\widehat V,\widehat E)$ has the form 
$(P,w)\mapsto(\phi_w(P),w)$. 
Commutation with paths avoiding $e_0,e_1$ makes $\phi_w$ 
independent of $w$, and the two marked edges make the common 
permutation commute with $0$ and $1$. 
Rigidity of $\Lambda_{\mathbf p}$ therefore implies that this 
automorphism is the identity. 
As in the argument for $\cK_{\mathbf p}$, every vertex has rooted 
representatives of arbitrarily large length, so every normalizing 
topological full group element induces such an automorphism. 
Proposition~\ref{prop:subgroupoidnormalizer} gives 
\begin{equation}\label{eq:uncountablenormalizer} 
\Nor\left(\sF(\widehat\cH),\sF(\cG_{\mathbf p})\right)=\sF(\cG_{\mathbf p}). 
\end{equation} 

It remains to calculate the homology and the map induced by inclusion. 
Put 
\[ 
D_0:=\Coker(I-A^t)=H_0(\cH), 
\qquad D_1:=\Ker(I-A^t)=H_1(\cH). 
\] 
For the cokernel presentation of $H_0(\cG_{\mathbf p})$, write 
$w_P$ for the generator at $(P,w)$ and put 
$\alpha_P:=a_P$, $\beta_P:=b_P$, $\gamma_P:=c_P$. 
The relations at $b$ and $c$ give 
\[ 
\beta_P+\gamma_P=0,\qquad 
\gamma_P=v_P+\alpha_P. 
\] 
After eliminating $\beta_P,\gamma_P$, the old vertices satisfy 
exactly the relations of $D_0$ separately for each $P$, while 
the relation at $a$ becomes 
$\alpha_P=\alpha_{P0}+\alpha_{P1}$. 
These two sets of relations are independent. 
By \eqref{eq:BBQShomology} and \eqref{eq:uncountablegraphmodel}, 
\begin{equation}\label{eq:uncountableH0} 
H_0(\cG_{\mathbf p}) 
\cong D_0^{(W_{\mathbf p})}\oplus\Z 
\oplus\bigoplus_{n\geq1}\Z/p_n. 
\end{equation} 

Let $\widetilde A_{\mathbf p}$ be the adjacency matrix of 
$\Delta_{\mathbf p}$. 
For the kernel calculation, write a finitely supported vector 
as $(z_P,\alpha_P,\beta_P,\gamma_P)_{P\in W_{\mathbf p}}$, 
where $z_P\in\Z^V$. 
The equation $u=\widetilde A_{\mathbf p}^tu$ is equivalent to 
\[ 
\begin{aligned} 
(I-A^t)z_P&=\epsilon_v\gamma_P,\\ 
\beta_P&=(z_P)_v+\alpha_P+2\beta_P,\\ 
\gamma_P&=(z_P)_v+\alpha_P+\beta_P,\\ 
\alpha_P&=\sum_{R\in W_{\mathbf p}}B_{\mathbf p}(R,P)\alpha_R 
+\gamma_P. 
\end{aligned} 
\] 
The middle two equations give $\gamma_P=0$. 
The last equation and $\Ker(I-B_{\mathbf p}^t)=0$ give 
$\alpha_P=0$, and the remaining equations give 
$z_P\in D_1$ and $\beta_P=-(z_P)_v$. 
For the finite ambient graph, the same calculation gives 
\[ 
\Ker(I-\widehat A^t) 
=\{(z,0,-z_v,0)\mid z\in D_1\}. 
\] 
Thus 
\begin{equation}\label{eq:uncountableH1} 
H_1(\cG_{\mathbf p})\cong D_1^{(W_{\mathbf p})}, 
\qquad H_1(\widehat\cH)\cong D_1. 
\end{equation} 
The projection $\rho_{\mathbf p}$ sums vertex coordinates over $P$ 
and intertwines $\widetilde A_{\mathbf p}^t$ with $\widehat A^t$. 
Under \eqref{eq:uncountablegraphmodel}, its restriction is the 
inclusion $\iota_{\mathbf p}:\cG_{\mathbf p}\to\widehat\cH$. 
Since inclusion of a full open reduction induces an isomorphism 
on homology, naturality of \eqref{eq:countablegraphhomology} shows that 
the induced map, in the coordinates of \eqref{eq:uncountableH1}, is 
\begin{equation}\label{eq:uncountableH1map} 
H_1(\iota_{\mathbf p}): 
D_1^{(W_{\mathbf p})}\longrightarrow D_1, 
\qquad (z_P)_P\longmapsto\sum_Pz_P. 
\end{equation} 
In particular, it is surjective. 

By Corollary~\ref{cor:minimalfullmaximal} (2), 
\eqref{eq:uncountablenormalizer} and the surjectivity above imply 
that, for every 
$\sD(\widehat\cH)\subset\Gamma\subset\sF(\widehat\cH)$, the group 
$M_{\mathbf p}:=\Gamma\cap\sF(\cG_{\mathbf p})$ is maximal in $\Gamma$. 

Finally, let $\mathcal P_A$ be the finite set of primes dividing 
the order of the torsion subgroup of the finitely generated group $D_0$. 
For each infinite set $\mathcal S$ of primes outside $\mathcal P_A$, 
use its increasing enumeration as $\mathbf p$. 
For a prime $p\notin\mathcal P_A$, \eqref{eq:uncountableH0} has 
an element of order $p$ if and only if $p\in\mathcal S$. 
Distinct choices of $\mathcal S$ therefore give nonisomorphic groups 
$H_0(\cG_{\mathbf p})$. 
Moreover, 
\[ 
\sD(\cG_{\mathbf p})\subset M_{\mathbf p}\subset\sF(\cG_{\mathbf p}). 
\] 
By \cite[Theorem 3.9]{Ma15crelle}, an isomorphism between two 
of the groups $M_{\mathbf p}$ would therefore induce an isomorphism 
between their groupoids. 
The distinct zeroth homology groups rule this out. 
There are $2^{\aleph_0}$ such sets $\mathcal S$. 
Transporting the construction through $\widehat\cH\cong\cH$ 
completes the proof. 
\end{proof}

%%%%%%%%%%%%%%%%%%%%%%%%%%%%%%%%%%%%%%%%%%%%%%%%%%%%%%%%%%%%
\section{Type systems, embeddings, realizations and examples} 
\label{sec:BBQS} 
%%%%%%%%%%%%%%%%%%%%%%%%%%%%%%%%%%%%%%%%%%%%%%%%%%%%%%%%%%%%

We compare the results above with those of Belk, Bleak, Quick and 
Skipper \cite{BBQS25TAMS}, using the groupoid interpretation of 
type systems developed in Section~\ref{sec:subgroupoidgraph}. 
After this comparison, we discuss embeddings with wide open images 
and realizations as reductions to nowhere dense Cantor sets. 
We then discuss several examples defined by finite graphs and examine the 
outer quotient and homological conditions for maximality.

%%%%%%%%%%%%%%%%%%%%%%%%%%%%%%%%%%%%%%%%%%%%%%%%%%%%%%%%%%%%
\subsection{Comparison of the main results}\label{subsec:BBQScomparison} 
%%%%%%%%%%%%%%%%%%%%%%%%%%%%%%%%%%%%%%%%%%%%%%%%%%%%%%%%%%%%

We use the dictionary between type systems and wide open subgroupoids 
of the full $2$-shift groupoid recorded in 
Remark~\ref{rem:subgroupoidgraphBBQS}. 
The following corollary collects the consequences of our general 
criteria for Thompson's group $V$. 

\begin{corollary}\label{cor:typesystems} 
Let $\cH_2$ be the full $2$-shift groupoid, so that 
$\sF(\cH_2)=V_2$ is Thompson's group $V$. 
\begin{enumerate} 
\item For every simple type system $\mathcal P$, finite or infinite, 
$\operatorname{Stab}(\mathcal P)$ is a maximal subgroup of $V_2$. 
\item For a simple type system $\mathcal P$, the subgroup 
$\Fix(\mathcal P)$ is maximal in $V_2$ if and only if 
$\Fix(\mathcal P)=\operatorname{Stab}(\mathcal P)$. 
\item A maximal subgroup $M$ of $V_2$ is fully amplified 
if and only if 
$M=\Fix(\mathcal P)=\operatorname{Stab}(\mathcal P)$ for some 
simple type system $\mathcal P$. 
\end{enumerate} 
\end{corollary} 

\begin{proof} 
Let $\mathcal P$ be simple, and let $\cG\subsetneq\cH_2$ be 
the corresponding prime wide open subgroupoid. 
By Remark~\ref{rem:subgroupoidgraphBBQS}, 
$\Fix(\mathcal P)=\sF(\cG)$ and 
$\operatorname{Stab}(\mathcal P)=\cN_{\cH_2}(\cG)$. 
The trichotomy, Lemma~\ref{lem:groupnormalizers} and 
Theorem~\ref{thm:exceptionalstabilizer} give 
\[ 
\operatorname{Stab}(\mathcal P) 
=\Nor\left(\sF(\cH_2),\sF(\cG)\right). 
\] 
Since $V_2$ is perfect, Corollary~\ref{cor:primenormalizermaximal} 
proves \textup{(1)}. 
As $\Fix(\mathcal P)\subset\operatorname{Stab}(\mathcal P)\subsetneq V_2$, 
assertion \textup{(2)} follows. 

For \textup{(3)}, let $M$ be a fully amplified maximal subgroup. 
Proposition~\ref{prop:maximalgerms} gives 
$M=\sF(\cG_M)$ with $\cG_M\subsetneq\cH_2$ prime. 
The corresponding type system $\mathcal P$ is simple and satisfies 
$M=\Fix(\mathcal P)$. 
By \textup{(1)}, its stabilizer is proper in $V_2$, so maximality 
of $M$ gives $M=\operatorname{Stab}(\mathcal P)$. 
Conversely, $\Fix(\mathcal P)=\sF(\cG)$ is fully amplified, 
proving the claim. 
\end{proof} 

Table~\ref{tab:BBQScomparison} summarizes the principal correspondences. 
The numbers in the left column refer to \cite{BBQS25TAMS}. 

\begin{table}[htbp] 
\caption{Comparison with the main results of \cite{BBQS25TAMS}.} 
\label{tab:BBQScomparison} 
\centering 
\begingroup 
\small 
\renewcommand{\arraystretch}{1.2} 
\begin{tabular}{@{}p{0.44\textwidth}p{0.52\textwidth}@{}} 
\hline 
BBQS & Present paper \\ 
\hline 
Corollary 2.3: finite sets in one tail class & 
Theorem~\ref{thm:exceptionalstabilizer} \\ 
Theorem 4.8: finite simple type systems & 
Theorem~\ref{thm:primetrichotomy} \\ 
Theorems 5.12, 5.13 and 5.15: maximality & 
Corollary~\ref{cor:typesystems} (1) \\ 
Proposition 6.3: finite clopen partitions & 
Proposition~\ref{prop:clopenpartitionnormalizer} \\ 
Theorem 6.14: finite nuclear graph model & 
Proposition~\ref{prop:subgroupoidgraphmodel} \\ 
Theorem 7.5: uncountably many maximal subgroups & 
Section~\ref{sec:uncountable}, 
especially Theorem~\ref{thm:uncountableminimal} \\ 
\hline 
\end{tabular} 
\endgroup 
\end{table} 

For a finite simple type system, the nuclear, atomic binuclear and 
atomic quasinuclear alternatives of \cite[Theorem 4.8]{BBQS25TAMS} 
correspond, respectively, to the minimal, clopen bipartition and 
nowhere dense alternatives in Theorem~\ref{thm:primetrichotomy}. 
That theorem gives a dynamical trichotomy without a finiteness 
assumption on the subgroupoid graph. 
In particular, it gives a structural decomposition for the infinite 
simple type systems considered in \cite[Question 5]{BBQS25TAMS}, 
while leaving their classification within each alternative open. 

Corollary~\ref{cor:typesystems} (1) treats all simple type systems 
uniformly. 
The simple nuclear case is covered by 
Theorem~\ref{thm:minimalnormalizermaximal}. 
A simple multinuclear type system with more than one nucleus is 
atomic binuclear and is covered by 
Proposition~\ref{prop:clopenpartitionnormalizer}. 
The atomic quasinuclear cases are covered by 
Theorem~\ref{thm:exceptionalstabilizer}; see 
Remark~\ref{rem:exceptionalBBQS} for this correspondence. 
This recovers \cite[Theorems 5.12, 5.13 and 5.15]{BBQS25TAMS}. 
It also covers infinite simple type systems which are neither 
multinuclear nor quasinuclear, such as the full $2$-shift cases in 
Examples~\ref{ex:exceptionalAF} and~\ref{ex:aperiodicpoint}. 
Their stabilizers are closed-set stabilizers, whose maximality 
also follows from \cite[Theorems 1.2 and 1.3]{GV26ETDS}. 

\begin{remark}\label{rem:exceptionalBBQS} 
We describe the graph structure underlying the atomic quasinuclear 
case. Return to the setting of Section~\ref{sec:exceptional}, with 
$\Delta$, $R$ and $C$ as defined there, and suppose that $\Delta|C$ 
contains a closed path. 
Then $\Delta|C$ has a unique terminal strongly connected component 
$Q$. Every closed path in $\Delta|C$ lies in $Q$, and every 
infinite path in this graph eventually enters $Q$. 
Moreover, there is $m\geq0$ such that $[\mu]_{\cG}\in Q\cup R$ 
whenever $\lvert\mu\rvert\geq m$; in particular, $C\setminus Q$ 
is finite. A closed path necessarily exists if $\Delta$ is finite, 
since $\Delta|C$ is nonempty and has no sinks. 

Indeed, choose a vertex $P_0$ on a closed path in $\Delta|C$. 
Cofinality from Lemma~\ref{lem:exceptionalgraph} (2), applied to 
the infinite repetition of this path, shows that every vertex of 
$C$ can reach $P_0$. 
Its strongly connected component $Q$ is therefore terminal: 
any vertex reachable from $Q$ within $C$ can return to $P_0$. 
Cofinality also gives a path from $P_0$ to a vertex on any infinite 
path in $\Delta|C$, so every such path enters $Q$. 
This proves uniqueness and the assertion about closed paths. 
Every lift from the root section eventually enters the forward 
closed set $Q\cup R$. 
The cylinders whose lifts have entered this set cover $X$; 
compactness gives the uniform bound $m$. 
Every vertex outside $Q\cup R$ is represented by a path of length 
less than $m$, proving finiteness of $C\setminus Q$. 

This component gives a smaller graph model for the reduction: 
$X_{\Delta|Q}$ is full and clopen in $X_{\Delta|C}$, and 
\[ 
\cG_{\Delta|C}|X_{\Delta|Q}=\cG_{\Delta|Q}. 
\] 
Indeed, $Q$ is forward closed in $\Delta|C$, and every infinite 
path eventually enters $Q$. 
Let $S$ be as in Proposition~\ref{prop:exceptionalgraphmodel}. 
A nonempty compact open bisection of the minimal groupoid 
$\cG_{\Delta|C}$ with source in $S$ and range in $X_{\Delta|Q}$ 
identifies full clopen reductions, so $\cG_{\Delta|Q}$ is 
Kakutani equivalent to $\cH|Y$. 

Branching is measured in $\Delta|Q$; edges from $Q$ to $R$ do not 
count. If there is no branching vertex, strong connectivity 
forces $\Delta|Q$ to be a single directed cycle. 
The uniform bound above then shows that $Y$ is finite and 
consists of eventually periodic points. 
If there is a branching vertex, every cylinder meeting $Y$ 
contains at least two points of $Y$, so $Y$ is a Cantor set. 
In particular, an infinite $Q$ always gives the Cantor case. 

For the full $2$-shift, $R$ consists of one vertex, and this is 
an atomic quasinuclear type system in the terminology of 
\cite[Definition 4.3]{BBQS25TAMS}. 
Their branching condition is precisely the existence of 
a vertex of outdegree two in $\Delta|Q$. 
Thus Theorem~\ref{thm:exceptionalstabilizer} recovers the stabilizer 
descriptions and maximality results in 
\cite[Theorems 5.12, 5.13 and Proposition 6.7]{BBQS25TAMS}, 
as well as the equality of $\Fix(\mathcal P)$ and 
$\operatorname{Stab}(\mathcal P)$ in \cite[Remark 6.8]{BBQS25TAMS}. 
\end{remark} 

The finite-set result \cite[Corollary 2.3]{BBQS25TAMS} is also 
covered by Theorem~\ref{thm:exceptionalstabilizer}. 
Indeed, let $Y$ be a nonempty finite subset of a single orbit 
of an irreducible SFT groupoid $\cH$. 
As explained in \cite[Section 3.1]{Ma26maximal2}, every 
arrow with source and range in $Y$ is regular. 
Thus $\cH|Y$ is transitive and hence minimal, and 
$\St(\Gamma,Y)$ is maximal for every 
$\sD(\cH)\subset\Gamma\subset\sF(\cH)$. 
This applies whether or not the points are eventually periodic. 
Corollary~\ref{cor:exceptionalclassification} describes the two 
possible reductions. 
In the eventually periodic case, sufficiently long prefixes have 
singleton follower sets whose successive tails run around a cycle, 
so the component $Q$ in Remark~\ref{rem:exceptionalBBQS} is a 
directed cycle. Otherwise $\Delta|C$ is acyclic. 
For the full $2$-shift, the choices $Y=\{0^\infty\}$ and 
$Y=\{(01)^\infty,(10)^\infty\}$ recover 
\cite[Examples 3.10 and 6.6]{BBQS25TAMS}; the corresponding 
components $Q$ are a loop and a two-cycle, respectively. 
Example~\ref{ex:aperiodicpoint} gives the contrasting case of 
a non-eventually-periodic singleton, with $\Delta|C$ an infinite ray. 

For general irreducible SFT groupoids, 
Corollary~\ref{cor:primenormalizermaximal} characterizes maximality 
of prime subgroupoid normalizers in all groups between $\sD(\cH)$ 
and $\sF(\cH)$ by their images in the ambient abelianization. 
In the minimal case, Corollary~\ref{cor:minimalfullmaximal} (2) 
characterizes maximality of $\sF(\cG)$ itself by triviality of 
$\Out_{\cH}(\cG)$ and surjectivity of $H_1(\cG)\to H_1(\cH)$; 
the examples in Section~\ref{subsec:normalizerexamples} illustrate 
these conditions and the role of the unit class. 

Finite clopen partitions also give maximal normalizers outside 
the prime setting. 
Proposition~\ref{prop:clopenpartitionnormalizer} recovers 
\cite[Proposition 6.3]{BBQS25TAMS} and gives the precise 
$H_0(\cH)$ condition for a general SFT groupoid. 
A partition into at least three parts gives a non-prime inclusion, 
as noted in Remark~\ref{rem:finitepartitionBBQS}. 

Since $H_0(\cH_2)=0$, the clopen case of 
Corollary~\ref{cor:amplemaximal} does not occur for $V_2$. 
By Corollary~\ref{cor:typesystems} (3), a maximal subgroup of $V_2$ 
preserving no nontrivial type system is not fully amplified, 
and its full amplification is $V_2$. 
The converse fails: the normalizer in 
Example~\ref{ex:paritynormalizers} (2) and the stabilizer of a finite 
clopen partition with at least two parts are maximal and not fully 
amplified, but preserve nontrivial type systems. 
For the partition stabilizer, its germ subgroupoid is $\cH_2$, 
since it permutes the parts transitively and contains the 
topological full group of each clopen reduction. 
Thus failure to be fully amplified does not characterize the maximal 
subgroups considered in \cite[Question 1]{BBQS25TAMS}. 

Let $\mathcal P$ be the type system corresponding to a wide open 
subgroupoid $\cG\subset\cH_2$. 
We now explain the algebraic invariants used in \cite{BBQS25TAMS}. 
Let $T(\cG)$ denote the type semigroup of $\cG$, including zero 
(see \cite[Definition 1.1]{ABBL23ETDS}). 
There is a canonical identification 
$S(\mathcal P)\cong T(\cG)\setminus\{0\}$, 
under which the generator associated with the type $[\mu]_{\cG}$ 
corresponds to $[1_{Z(\mu)}]$. 
Under this identification, the s-type of $U\in\cCO(\cH_2^{(0)})$ 
is $[1_U]\in T(\cG)$. 
Thus two nonempty clopen sets have the same s-type if and only if 
they are the source and range of a compact open bisection of $\cG$. 
The Grothendieck group of $T(\cG)$ is $H_0(\cG)$; 
passing to this group can lose information. 
In the nuclear case, however, $S(\mathcal P)$ is already a group 
by \cite[Lemma 5.1 (ii)]{BBQS25TAMS}, 
and hence identifies with $H_0(\cG)$. 

Proposition~\ref{prop:subgroupoidgraphmodel} identifies $\cG$ with 
a full compact open reduction of the graph groupoid of its entire 
type graph, without a finiteness or nuclearity assumption. 
For a finite nuclear type system, the root cylinders can be moved 
by prefix replacements to disjoint cylinders in the nucleus, 
as in the proof of \cite[Theorem 6.14]{BBQS25TAMS}. 
This gives the reduction to the nucleus graph used there. 
Thus the graph model extends that result at the groupoid level, 
and the preceding semigroup identification explains its 
homological conclusion. 

Finally, Section~\ref{sec:uncountable} first translates the 
construction in \cite[Theorem 7.5]{BBQS25TAMS} into subgroupoid 
graphs and computes its homology, including the unit class and 
$H_1$. 
Theorem~\ref{thm:uncountableminimal} then constructs a single 
family $(\cG_t)_{t\in T}$, with $\#T=2^{\aleph_0}$, 
in any fixed irreducible SFT groupoid $\cH$ such that, 
for every fixed 
$\sD(\cH)\subset\Gamma\subset\sF(\cH)$, the groups 
$\Gamma\cap\sF(\cG_t)$ are pairwise nonisomorphic maximal 
subgroups of $\Gamma$. 
The extension concerns the ambient groupoid and all these 
intermediate groups; the construction of \cite{BBQS25TAMS} 
already uses infinite type systems. 

Theorem~\ref{thm:minimalovergroup} provides the generation step 
for the prime minimal inclusions studied here: a subgroup 
containing $\sD(\cG)$ and an element outside 
$\Nor\left(\sF(\cH),\sF(\cG)\right)$ must contain $\sD(\cH)$. 
The general generation criteria based on cone transitivity 
and swap-primitivity in \cite[Theorems 8.5 and 8.18]{BBQS25TAMS} 
are not treated here. 
Nor do we answer \cite[Questions 1--4 and 6]{BBQS25TAMS}, 
which concern, among other things, automorphic images of $T$ 
and maximal subgroups of $F$ and $T$.

%%%%%%%%%%%%%%%%%%%%%%%%%%%%%%%%%%%%%%%%%%%%%%%%%%%%%%%%%%%%
\subsection{Embeddings into the full two-shift groupoid} 
\label{subsec:binaryrealization} 
%%%%%%%%%%%%%%%%%%%%%%%%%%%%%%%%%%%%%%%%%%%%%%%%%%%%%%%%%%%%

We consider embeddings of groupoids into the full two-shift 
groupoid whose images are wide open subgroupoids. 

For the full two-shift groupoid, the following result gives 
a converse, at the level of isomorphism classes, to the graph model 
in Proposition~\ref{prop:subgroupoidgraphmodel}. 
It is a special case of Nyland and Ortega's embedding theorem 
\cite[Theorem 11.16]{NO19IJM}. 
For completeness, we give a direct proof for the row-finite 
graphs considered here, following the prefix-code construction 
in \cite[Lemma 11.1]{NO19IJM}. 
Their construction was inspired by Brownlowe and S\o rensen's 
work on embeddings of Leavitt path algebras 
\cite[Proposition 5.1]{BS16JA}. 
We retain the notation $\cH_2$ for the full $2$-shift groupoid 
and put $C:=\{0,1\}^{\N}$. 

\begin{proposition}\label{prop:binarygraphrealization} 
Let $(W,F)$ be a countable row-finite directed graph without sinks, 
and let $Y\in\cCO(X_{(W,F)})$ be nonempty with no isolated points. 
Then $\cG_{(W,F)}|Y$ is isomorphic to a wide open subgroupoid 
of $\cH_2$. 
\end{proposition} 

\begin{proof} 
Write 
\[ 
Y=\bigsqcup_{j=1}^r Z(\mu_j), 
\] 
where the finite paths $\mu_1,\dots,\mu_r$ are pairwise incomparable. 
Let $W_Y$ be the set of vertices reachable from some $t(\mu_j)$. 
Every infinite path starting in $W_Y$ passes through vertices 
of outdegree at least two infinitely often. 
Otherwise, adjoining a suitable prefix from $Y$ would give 
an isolated point of $Y$. 

For each $w\in W_Y$, choose a finite binary word $c(f)$ 
for every $f\in wF$ so that the cylinders $c(f)C$, $f\in wF$, 
partition $C$. 
Such a complete prefix code exists for every positive number of 
words: starting with the empty word, replace one word $a$ by 
$a0,a1$ until the desired number is reached. 
In particular, when $w$ has outdegree one, its edge receives 
the empty word. 
Extend $c$ to finite paths starting in $W_Y$ by concatenation, 
with $c(w)$ equal to the empty word for vertices. 
For $w\in W_Y$, define 
\[ 
h_w:Z(w)\longrightarrow C, 
\qquad h_w(f_1f_2\cdots)=c(f_1)c(f_2)\cdots. 
\] 
The preceding observation ensures that the word on the right 
is infinite. 
Successive decoding gives a unique inverse, and every finite 
initial segment of the code is determined by a finite path. 
Thus $h_w$ is a continuous bijection between compact Hausdorff 
spaces and is a homeomorphism. 

Choose a complete binary prefix code $\{p_1,\dots,p_r\}$ and set 
\[ 
h(\mu_j z):=p_jh_{t(\mu_j)}(z). 
\] 
This defines a homeomorphism $h:Y\to C$. 
For a finite path $\mu=\mu_j\eta$ extending one of the $\mu_j$, 
put $b(\mu):=p_jc(\eta)$. 
Then 
\[ 
h(\mu z)=b(\mu)h_{t(\mu)}(z),\qquad z\in Z(t(\mu)). 
\] 
Consequently, when $\mu,\nu$ extend the chosen root paths and 
$t(\mu)=t(\nu)$, conjugation by $h$ carries the prefix replacement 
$\mu z\mapsto\nu z$ to 
\[ 
b(\mu)\omega\longmapsto b(\nu)\omega, 
\qquad\omega\in C. 
\] 
This is the homeomorphism induced by the compact open bisection 
$Z(b(\nu),b(\mu))$ of $\cH_2$. 
By simultaneous extension of the two prefixes, the bisections 
$Z(\nu,\mu)$ with these properties form a basis for 
$\cG_{(W,F)}|Y$. 

The reduction $\cG_{(W,F)}|Y$ is effective. 
Indeed, a basic bisection of nonunit isotropy would force its 
source cylinder in $Y$ to have a unique infinite continuation, 
contrary to the hypothesis on $Y$. 
Since $\cH_2$ is also effective, both groupoids identify with 
the groupoids of germs of their bisections; 
see \cite[Section 3.1]{Ma26maximal2}. 
Conjugation by $h$ therefore induces an injective homomorphism 
\[ 
\Phi:\cG_{(W,F)}|Y\longrightarrow\cH_2 
\] 
which maps each of the above bisections homeomorphically onto 
$Z(b(\nu),b(\mu))$. 
Hence $\Phi$ is an open embedding. 
Its image is wide because $h(Y)=C$. 
\end{proof} 

In particular, every SFT groupoid with Cantor unit space is 
isomorphic to a wide open subgroupoid of $\cH_2$, by taking 
$Y=X_{(W,F)}$. 
This is the finite-graph case of \cite[Theorem 11.3]{NO19IJM}. 
The same conclusion holds for every AF groupoid with Cantor 
unit space \cite[Corollary 11.26]{NO19IJM}. 
Indeed, represent it by a Bratteli diagram with a single root 
\cite[Theorem 3.9]{GPS04ETDS} and view the diagram as a graph. 
Two finite paths from the root with the same terminal vertex 
have the same length, so the reduction to the root cylinder is 
the given AF groupoid \cite[Proposition 11.25]{NO19IJM}. 
Proposition~\ref{prop:binarygraphrealization} therefore applies. 

\begin{remark}\label{rem:binaryrealizationscope} 
The embeddings described above impose no primeness condition 
on their images and do not by themselves imply maximality of 
the associated topological full groups or normalizers. 
They include the embeddings underlying 
\cite[Examples 6.16--6.18]{BBQS25TAMS}. 

The choice of the full $2$-shift is essential for this generality. 
Let $\cK$ be its canonical AF core, described in 
Example~\ref{ex:AFcoresubgroupoidgraph}, and let $\cH_3$ be the 
full $3$-shift groupoid. 
Their zeroth homology groups with unit classes are 
\[ 
\left(H_0(\cK),[1_{\cK^{(0)}}]\right) 
\cong\left(\Z[1/2],1\right), 
\qquad 
\left(H_0(\cH_3),[1_{\cH_3^{(0)}}]\right) 
\cong\left(\Z/2,1\right). 
\] 
Every homomorphism $\Z[1/2]\to\Z/2$ is zero. 
Since a wide open inclusion induces a homomorphism on $H_0$ 
preserving the unit class, $\cK$ cannot be embedded into $\cH_3$ 
with wide open image. 
Nevertheless, by \cite[Corollary 4.13]{Ma26maximal2}, there is a 
nowhere dense $\cH_3$-\'etale Cantor set $Y\subset\cH_3^{(0)}$ 
such that $\cK\cong\cH_3|Y$, the full restriction to $Y$. 
\end{remark}

%%%%%%%%%%%%%%%%%%%%%%%%%%%%%%%%%%%%%%%%%%%%%%%%%%%%%%%%%%%%
\subsection{Stabilizers and strong realizations} 
\label{subsec:BBQSexamples} 
%%%%%%%%%%%%%%%%%%%%%%%%%%%%%%%%%%%%%%%%%%%%%%%%%%%%%%%%%%%%

We use the terminology of \cite[Definition 4.7]{Ma26maximal2}: 
an effective ample groupoid $\cK$ with compact unit space is 
realizable in $\cH$ if there exists a closed nowhere dense set 
$Y\subset\cH^{(0)}$ with Hausdorff effective reduction such that 
$\cK\cong\cH{\downarrow}Y$. 
It is strongly realizable if $\cK\cong\cH|Y$ for a closed nowhere 
dense $\cH$-\'etale set $Y$, so that this is a full restriction. 
When $\cK^{(0)}$ is a Cantor set, so is $Y$. 

We consider stabilizers of nowhere dense Cantor sets. 
The first example describes a regular reduction to a set that is 
not $\cH_2$-\'etale. 
We then use proper subgraphs to construct strong realizations 
for general SFT groupoids. 

\begin{example}[{\cite[Example 6.9]{BBQS25TAMS}}] 
\label{ex:BBQSCantorstabilizer} 
Let $\cH_2$ be the full $2$-shift groupoid, and put 
\[ 
Y:=\{00,1\}^{\N}\subset\{0,1\}^{\N}. 
\] 
Thus $Y$ consists of all infinite concatenations of $00$ and $1$. 
This is a nowhere dense Cantor set. 
Indeed, concatenation gives a homeomorphism from a full two-symbol 
shift space onto $Y$, and every finite concatenation can be extended 
by $01$ to obtain a cylinder disjoint from $Y$. 

For a finite binary word $\mu$, the set 
$\{z\in\{0,1\}^{\N}\mid\mu z\in Y\}$ is 
\[ 
\begin{cases} 
Y&\text{if }\mu\in\{00,1\}^*,\\ 
0Y&\text{if }\mu\in\{00,1\}^*0,\\ 
\emptyset&\text{otherwise}. 
\end{cases} 
\] 
These three sets are distinct, and a prefix replacement preserves 
$Y$ precisely when its two prefixes have the same follower set. 
Thus the subgroupoid $\cG$ defined by 
\eqref{eq:exceptionalsubgroupoid} has three vertices $A,B,R$, 
root $A$, and transitions 
\[ 
\begin{array}{c|cc} 
P&P0&P1\\ \hline 
A&B&A\\ 
B&A&R\\ 
R&R&R 
\end{array}. 
\] 
Deleting $R$ leaves the graph $\Lambda$ with edges 
$A\xrightarrow{0}B$, $A\xrightarrow{1}A$ and 
$B\xrightarrow{0}A$. 
Every finite path in $\Lambda$ starting at $A$ is uniquely 
determined by its binary label. 
The follower-set calculation shows that the basic prefix 
replacements preserving $Y$ and meeting $Y$ are precisely 
those whose two prefixes lift to paths ending at the same 
vertex of $\Lambda$. 
The label map therefore identifies the corresponding basic 
bisections and gives 
\[ 
\cH_2|Y\cong\cG_{\Lambda}|Z(A), 
\qquad 
M_{\Lambda}=\begin{bmatrix}1&1\\1&0\end{bmatrix}, 
\] 
where $M_{\Lambda}$ is the adjacency matrix of $\Lambda$. 
Since $\Lambda$ is strongly connected, this reduction is minimal. 
Thus Theorem~\ref{thm:exceptionalstabilizer} applies and gives 
\[ 
\sF(\cG)=\Nor\left(V_2,\sF(\cG)\right)=\St(\cH_2,Y), 
\] 
which is a maximal subgroup of $V_2$. 
In fact, coding the two first-return paths $00$ and $1$ at $A$ 
identifies $\cG_{\Lambda}|Z(A)$ with $\cH_2$. 
Thus $\cH_2$ is realizable in itself via $Y$. 
This describes the regular reduction explicitly, in addition 
to recovering the stabilizer in \cite[Example 6.9]{BBQS25TAMS}. 

The set $Y$ is not $\cH_2$-\'etale. 
Indeed, the arrow at $0^\infty$ represented by $z\mapsto 0z$ 
has both endpoints in $Y$, but is not regular for $Y$: 
\[ 
(00)^k1^\infty\in Y,\qquad 
0(00)^k1^\infty\notin Y, 
\] 
and $(00)^k1^\infty\to 0^\infty$. 
Hence no restriction of this prefix replacement to a neighborhood 
of $0^\infty$ preserves $Y$. 
The realization via this particular $Y$ is therefore not strong. 
An AF variant, in which the regular reduction is also minimal 
and splits an orbit of the full restriction, is given in 
\cite[Example 3.18]{Ma26maximal2}. 
\end{example} 

\begin{example}[Proper subgraphs]\label{ex:subgraphstabilizer} 
Let $\cH=\cG_{(V,E)}$ be an irreducible SFT groupoid, and let 
$(W,F)$ be a proper subgraph of $(V,E)$ whose adjacency matrix 
is irreducible and is not a permutation matrix. 
Put $Y:=X_{(W,F)}\subset X_{(V,E)}$. 
Then $Y$ is a Cantor set and is closed. 
It is nowhere dense: irreducibility of $(V,E)$ allows every 
finite path in $(W,F)$ to be extended to a path using an edge 
outside $F$. 
Moreover, every arrow of $\cH$ with source and range in $Y$ is regular, 
and $\cH|Y=r^{-1}(Y)\cap s^{-1}(Y)=\cG_{(W,F)}$. 
Indeed, it can be represented by a prefix replacement between 
two paths in $(W,F)$ with the same terminal vertex, and such 
a replacement locally preserves $Y$. 
Thus $Y$ is $\cH$-\'etale and $\cG_{(W,F)}$ is strongly 
realizable in $\cH$. 
The regular reduction is therefore minimal, and 
Theorem~\ref{thm:exceptionalstabilizer} shows that 
$\St(\Gamma,Y)$ is maximal in every group 
$\sD(\cH)\subset\Gamma\subset\sF(\cH)$. 

The subgroupoid graph can also be read directly from the subgraph. 
For the subgroupoid $\cG$ in \eqref{eq:exceptionalsubgroupoid}, 
its vertices are $Q_w$ for $w\in W$ and $R_v$ for $v\in V$. 
Both $Q_w$ and $R_w$ project to $w$, and the transitions are 
\[ 
Q_we=\begin{cases} 
Q_{t(e)}&e\in F,\\ 
R_{t(e)}&e\notin F, 
\end{cases} 
\qquad 
R_ve=R_{t(e)}. 
\] 
The root at $v$ is $Q_v$ if $v\in W$ and $R_v$ otherwise. 
The vertices $Q_w$ record paths which have remained in $(W,F)$; 
the vertices $R_v$ record paths whose cylinders miss $Y$. 
Deleting the latter vertices leaves precisely $(W,F)$. 
Thus this construction gives SFT examples of the configuration 
in Section~\ref{sec:exceptional}. 
By \cite[Theorem 4.3 (2)]{Ma26maximal2}, the isomorphism 
class of $\cG_{(W,F)}$ is an invariant of the resulting stabilizer. 

For instance, let $\cH_d$ be the full $d$-shift groupoid, 
with $d\geq3$, and take 
$Y=\{0,\dots,m-1\}^{\N}$ for $2\leq m<d$. 
This gives a strong realization of the full $m$-shift groupoid 
in $\cH_d$. 
Here the choices of $\Gamma$ are limited. 
For even $d$, one has $\sF(\cH_d)=\sD(\cH_d)$, so this is 
the only possible $\Gamma$. 
For odd $d$, one has $\sF(\cH_d)/\sD(\cH_d)\cong\Z/2$, 
so $\Gamma$ is either $\sD(\cH_d)$ or $\sF(\cH_d)$; 
see \cite[Corollary 6.24]{Ma15crelle}. 
For each fixed choice of $\Gamma$, varying $m$ gives 
pairwise nonisomorphic maximal subgroups 
$\St(\Gamma,Y)$, since the zeroth homology groups of their 
regular reductions are $\Z/(m-1)$. 
\end{example} 

The subgraph construction above leads to the following general 
strong realization result. 

\begin{proposition}\label{prop:SFTstrongrealization} 
Let $\cH$ be an irreducible SFT groupoid with Cantor unit space, 
and let $\cK$ be any SFT groupoid with Cantor unit space. 
Then $\cK$ is strongly realizable in $\cH$: 
there exists a nowhere dense $\cH$-\'etale Cantor set 
$Y\subset\cH^{(0)}$ such that 
\[ 
\cK\cong\cH|Y=r^{-1}(Y)\cap s^{-1}(Y) 
\] 
with the relative topology on the full restriction. 
\end{proposition} 

\begin{proof} 
We first strongly realize $\cK$ in $\cH_2$. 
Write $\cK=\cG_{(W,F)}$, and let $B$ be the adjacency matrix 
of $(W,F)$. 
Choose an odd integer $n\geq\#W$ and a positive integral 
$n\times n$ matrix $P$ whose upper left $\#W\times\#W$ block 
dominates $B$ entrywise. 
Set 
\[ 
C:=\begin{bmatrix}P&P\\I_n&2I_n\end{bmatrix}. 
\] 
The graph of $C$ is strongly connected and contains $(W,F)$ 
as a proper subgraph. 
Let $\cG$ be its SFT groupoid. 
Multiplication by a unimodular matrix gives 
\[ 
\begin{bmatrix}I_n&-P\\0&I_n\end{bmatrix}(I_{2n}-C) 
=\begin{bmatrix}I_n&0\\-I_n&-I_n\end{bmatrix}. 
\] 
Hence 
\[ 
H_0(\cG)=\Coker(I_{2n}-C^t)=0, 
\qquad \det(I_{2n}-C)=(-1)^n=-1. 
\] 
The unit class is necessarily zero, so 
\cite[Theorem 3.6]{MM14Kyoto} gives $\cG\cong\cH_2$. 
The path space of $(W,F)$ is a nowhere dense $\cG$-\'etale 
Cantor subset, and its full restriction is $\cK$. 
Indeed, the full-restriction argument in 
Example~\ref{ex:subgraphstabilizer} does not require the subgraph 
to be irreducible. 
Thus $\cK$ is strongly realizable in $\cH_2$. 

Next, let $A$ be the adjacency matrix defining $\cH$. 
The matrix $\widehat A$ in \eqref{eq:uncountableambientmatrix} 
defines an SFT groupoid $\widehat\cH\cong\cH$, as shown in 
the proof of Theorem~\ref{thm:uncountableminimal}. 
Its vertex $a$ and the two loops $e_0,e_1$ form a proper subgraph 
whose groupoid is $\cH_2$. 
Example~\ref{ex:subgraphstabilizer} therefore strongly realizes 
$\cH_2$ in $\widehat\cH$, and hence in $\cH$. 
The conclusion follows from transitivity of strong realizability, 
as established in \cite[Section 4.2]{Ma26maximal2}. 
\end{proof}

%%%%%%%%%%%%%%%%%%%%%%%%%%%%%%%%%%%%%%%%%%%%%%%%%%%%%%%%%%%% 
\subsection{Finite subgroupoid graphs}\label{subsec:finitegraphexamples} 
%%%%%%%%%%%%%%%%%%%%%%%%%%%%%%%%%%%%%%%%%%%%%%%%%%%%%%%%%%%% 

We now recover three examples from \cite{BBQS25TAMS} by checking 
primeness and computing homology from their finite subgroupoid graphs. 
In the group-theoretic conclusions below we write normalizers in 
the usual group notation, using Lemma~\ref{lem:groupnormalizers}. 

Let $\Lambda$ be a finite strongly 
connected reduced graph with one outgoing edge of each label 
$0,1$ at every vertex, and choose a root $q$. 
Propositions~\ref{prop:subgroupoidgraphconverse} 
and~\ref{prop:subgroupoidgraphmodel} give a wide open subgroupoid 
\[ 
\cK\subset\cH_2,\qquad 
\cK\cong\cG_{\Lambda}|Z(q). 
\] 
It is minimal and purely infinite. 
In each example, we verify primeness 
using Corollary~\ref{cor:primegraph} and check that every 
label-preserving automorphism of $\Lambda$ is trivial, without 
requiring it to fix the root. 
As in Section~\ref{sec:uncountable}, every vertex is represented 
by arbitrarily long rooted paths, so 
Proposition~\ref{prop:subgroupoidnormalizer} associates such an 
automorphism to every element of $\Nor\left(V_2,\sF(\cK)\right)$. 
Triviality of these automorphisms therefore gives 
$\Nor\left(V_2,\sF(\cK)\right)=\sF(\cK)$. 
Since $H_1(\cH_2)=0$, a proper prime inclusion with this property 
gives a maximal subgroup $\sF(\cK)$ of $V_2$ by 
Corollary~\ref{cor:minimalfullmaximal} (2). 

If $M$ is the adjacency matrix of $\Lambda$, then 
Proposition~\ref{prop:graphalgebras} (4) gives 
\[ 
H_0(\cK)\cong\Coker(I-M^t),\qquad 
[1_{\cK^{(0)}}]=[\epsilon_q],\qquad 
H_1(\cK)\cong\Ker(I-M^t). 
\] 
The group $\sF(\cK)$ is of type $F_\infty$ by 
\cite[Theorem 6.21]{Ma15crelle}. 
Its abelianization is given by 
\[ 
H_1(\sF(\cK))\cong 
\left(H_0(\cK)\otimes\Z/2\right)\oplus H_1(\cK), 
\] 
by \cite[Corollary 6.24]{Ma15crelle}. 

\begin{example}[{Higman--Thompson groups; 
\cite[Example 6.16]{BBQS25TAMS}}]\label{ex:BBQSHigmanThompson} 
Fix $n\geq3$, put $m:=n-1$, and consider the graph with vertices 
$1,\dots,m$ and transitions 
\[ 
(1)0=m,\qquad (1)1=1,\qquad 
(i)0=1,\qquad (i)1=i-1\quad(2\leq i\leq m). 
\] 
Here $(i)j$ denotes the successor of vertex $i$ along label $j$. 
Choose $r\in\{1,\dots,m\}$ as the root, and denote the resulting 
subgroupoid by $\cK_{n,r}$. 
The graph is strongly connected and reduced. 
If a quotient identifies $i<j$, applying $1^{i-1}$ identifies 
$1$ with $j-i+1>1$; applying $0$ then identifies $m$ with $1$. 
Successive applications of $1$ identify every vertex with $1$. 
Thus $\cK_{n,r}\subset\cH_2$ is prime. 
The vertex $1$ is the unique vertex fixed by the transition $1$, 
so every label-preserving automorphism fixes it and hence every 
vertex. 
Consequently $\sF(\cK_{n,r})$ is maximal in $V_2$. 

Writing $a_i$ for the vertex generators of $H_0$, the relations 
are $a_i=a_1+a_{i-1}$ for $i>1$ and $a_1=a_m+a_1$. 
Hence 
\[ 
H_0(\cK_{n,r})\cong\Z/(n-1),\qquad 
[1_{\cK_{n,r}^{(0)}}]=r,\qquad H_1(\cK_{n,r})=0. 
\] 
There are exactly $n$ first-return paths from vertex $1$ to itself: 
their labels are $1$, $01^k0$ for $0\leq k\leq m-2$, and 
$01^{m-1}$. 
Every infinite path visits $1$ infinitely often. 
There are also exactly $r$ paths from the root $r$ which first 
reach $1$, allowing the path of length zero when $r=1$. 
After simultaneous extension, every basic prefix replacement 
can be taken to end at $1$. 
Coding the initial paths by $r$ roots and the first-return paths 
by $n$ symbols therefore identifies these replacements with 
those defining the Higman--Thompson group $V_{n,r}$. 
Thus 
\[ 
\sF(\cK_{n,r})\cong V_{n,r}, 
\] 
recovering the maximal subgroups of 
\cite[Example 6.16]{BBQS25TAMS}. 
\end{example} 

\begin{example}[{\cite[Example 6.17]{BBQS25TAMS}}] 
\label{ex:BBQSnonHigmanThompson} 
Take the root $A$ and the following transitions on five vertices: 
\[ 
\begin{array}{c|cc} 
P&P0&P1\\ \hline 
A&B&C\\ 
B&D&C\\ 
C&B&E\\ 
D&C&A\\ 
E&A&B 
\end{array}. 
\] 
The graph is strongly connected and reduced. 
Any pair of distinct vertices other than $\{A,B\}$ is sent to 
$\{A,B\}$ by one of the words $0,1,10,010,110$. 
An identification $A\equiv B$ then gives successively 
$B\equiv D$, $C\equiv A$ and $C\equiv E$. 
Thus every nontrivial quotient identifies all vertices. 
Moreover, $A$ is the unique vertex of indegree two, so every 
label-preserving automorphism fixes $A$ and hence every vertex. 
The resulting subgroupoid $\cK\subset\cH_2$ is therefore prime, 
and $\sF(\cK)$ is maximal in $V_2$. 

Let $a,b,c,d,e$ be the vertex generators of $H_0(\cK)$. 
The relations give 
\[ 
a=b+c,\qquad d=b-c,\qquad e=c-b, 
\qquad 3b=3c=0, 
\] 
with no further relations. 
Consequently 
\[ 
H_0(\cK)\cong\Z/3\oplus\Z/3,\qquad 
[1_{\cK^{(0)}}]=(1,1),\qquad H_1(\cK)=0. 
\] 
It follows that $\sF(\cK)=\sD(\cK)$ is a simple group of 
type $F_\infty$. 
It is not isomorphic to the commutator subgroup of any 
Higman--Thompson group: by \cite[Theorem 3.10]{Ma15crelle}, 
such an isomorphism would induce an isomorphism of the 
corresponding groupoids, whereas the zeroth homology group 
of a Higman--Thompson groupoid is cyclic. 
\end{example} 

\begin{example}[{\cite[Example 6.18]{BBQS25TAMS}}] 
\label{ex:BBQSinfiniteabelianization} 
Take the root $A$ and the transitions 
\[ 
\begin{array}{c|cc} 
P&P0&P1\\ \hline 
A&A&B\\ 
B&A&C\\ 
C&B&C 
\end{array}. 
\] 
Again the graph is strongly connected and reduced. 
An identification of any two vertices forces all three to be 
identified: $A\equiv B$ gives $B\equiv C$ by applying $1$, 
and either $A\equiv C$ or $B\equiv C$ gives $A\equiv B$ 
by applying $0$. 
The vertex $A$ is the unique vertex fixed by $0$, so the graph 
has no nontrivial label-preserving automorphism. 
Hence the associated subgroupoid $\cK\subset\cH_2$ is prime, 
and $\sF(\cK)$ is maximal in $V_2$. 

The homology relations are 
$a=a+b$, $b=a+c$ and $c=b+c$. 
They reduce to $b=0$ and $c=-a$, and hence 
\[ 
H_0(\cK)\cong\Z,\qquad [1_{\cK^{(0)}}]=1, 
\qquad H_1(\cK)\cong\Z. 
\] 
For the last assertion, $\Ker(I-M^t)$ is generated by 
$(1,0,-1)^t$. 
Thus $\sF(\cK)$ is of type $F_\infty$ and 
\[ 
H_1(\sF(\cK))\cong\Z/2\oplus\Z. 
\] 
Its commutator subgroup is simple and finitely generated 
by \cite[Theorem 4.16 and Corollary 6.25]{Ma15crelle}, 
but has infinite index. 
In particular, $\sF(\cK)$ is not virtually simple. 
The nontrivial abelianization of $\sF(\cK)$ is compatible with 
its maximality in $V_2$: the criterion in 
Corollary~\ref{cor:minimalfullmaximal} concerns the map to 
$H_1(\cH_2)=0$, which is automatically surjective. 
\end{example}

%%%%%%%%%%%%%%%%%%%%%%%%%%%%%%%%%%%%%%%%%%%%%%%%%%%%%%%%%%%% 
\subsection{Normalizers and homological obstructions} 
\label{subsec:normalizerexamples} 
%%%%%%%%%%%%%%%%%%%%%%%%%%%%%%%%%%%%%%%%%%%%%%%%%%%%%%%%%%%% 

The final examples distinguish the two conditions in 
Corollary~\ref{cor:minimalfullmaximal} (2). 
They include both maximal normalizers strictly containing 
$\sF(\cG)$ and a self-normalizing $\sF(\cG)$ that is not maximal. 
Whenever such a normalizer is maximal and strictly contains 
$\sF(\cG)$, it is not fully amplified: its germ subgroupoid strictly 
contains $\cG$, so primeness forces it to be $\cH$, and its full 
amplification is $\sF(\cH)$. 

\begin{example}[Parity cocycles]\label{ex:paritynormalizers} 
Let $X:=\cH_2^{(0)}$ and $E=\{0,1\}$, and take either 
$E_1=\{0,1\}$ or $E_1=\{1\}$ in Example~\ref{ex:cocyclesubgroupoidgraph}. 
Thus $\xi:\cH_2\to\Z/2$ counts, modulo two, the change in 
the number of symbols from $E_1$ under a prefix replacement. 
Put $\cG:=\Ker\xi$ and write $P_0,P_1$ for the two parity 
classes, with root $P_0$. 
Both graphs below are strongly connected and reduced, so 
$\cG$ is minimal and purely infinite. 
Since their only nontrivial quotient identifies the two vertices, 
the inclusions $\cG\subset\cH_2$ are prime. 
In both cases the label-preserving automorphism group has order two. 
By Proposition~\ref{prop:subgroupoidnormalizer}, the normalizer 
acts on this graph with kernel $\sF(\cG)$. 

\textup{(1)} Suppose $E_1=\{0,1\}$. 
Then $\cG=c^{-1}(2\Z)$ for the canonical cocycle 
$c(x,n,y)=n$, and the transitions are 
$P_i0=P_i1=P_{i+1}$, with indices taken modulo two. 
The adjacency matrix and homology are 
\[ 
M=\begin{bmatrix}0&2\\2&0\end{bmatrix},\qquad 
H_0(\cG)\cong\Z/3,\qquad H_1(\cG)=0. 
\] 
Writing $p_i$ for the vertex class of $P_i$, we have 
$p_1=-p_0$ and $[1_X]=p_0$, where $p_0$ generates $\Z/3$. 
A normalizer element interchanging the vertices would send 
each cylinder defined by a sufficiently long prefix of parity 
$i$ to a cylinder of parity $i+1$, and hence act by negation 
on $H_0(\cG)$. 
This is impossible, since it must fix $[1_X]$. 
Therefore 
\[ 
\Out_{\cH_2}(\cG)=1,\qquad 
\Nor\left(V_2,\sF(\cG)\right)=\sF(\cG), 
\] 
and this group is maximal in $V_2$ by 
Corollary~\ref{cor:minimalfullmaximal} (2). 
For the type system in \cite[Example 3.12]{BBQS25TAMS}, one 
therefore has $\Fix(\mathcal P)=\operatorname{Stab}(\mathcal P)$. 
This corrects the assertion in \cite[Example 3.12]{BBQS25TAMS} 
that $\Fix(\mathcal P)$ has index two in 
$\operatorname{Stab}(\mathcal P)$; 
see also the prefix-set argument below. 

The obstruction also has an elementary prefix-set formulation. 
If a complete binary prefix set has $e$ words of even length 
and $o$ words of odd length, then $e-o\equiv1\pmod3$: 
the singleton consisting of the empty word has this property, 
and splitting one word into its two children changes $e-o$ 
by $3$ or $-3$. 
A prefix replacement changing every length by an odd number 
would interchange the even and odd counts of its source and 
range prefix sets, contradicting this congruence. 

\textup{(2)} Suppose $E_1=\{1\}$. 
Now $P_i0=P_i$ and $P_i1=P_{i+1}$, and 
\[ 
M=\begin{bmatrix}1&1\\1&1\end{bmatrix},\qquad 
H_0(\cG)=H_1(\cG)=0. 
\] 
The involution $\tau\in V_2$ given by 
$\tau(0z)=1z$ and $\tau(1z)=0z$ has $\xi=1$ on $U_\tau$. 
The cocycle identity shows that conjugation by $U_\tau$ preserves 
$\Ker\xi$, and $\tau$ induces the nontrivial graph automorphism. 
Consequently 
\[ 
\Out_{\cH_2}(\cG)\cong\Z/2,\qquad 
\Nor\left(V_2,\sF(\cG)\right)=\sF(\cG)\rtimes\langle\tau\rangle. 
\] 
This normalizer is maximal in $V_2$ by 
Corollary~\ref{cor:minimalfullmaximal} (1), since $V_2$ is perfect. 
The subgroup $\sF(\cG)$ has index two in the normalizer and 
is therefore not maximal in $V_2$. 
\end{example} 

\begin{example}[Cyclic covers and semidirect products] 
\label{ex:cycliccovers} 
We relate the examples in \cite[Section 6.2]{Ma26maximal1} 
to the normalizer criteria of Section~\ref{sec:minimal}. 
Let $p$ be an odd prime, let $\cK=\cG_{(V,E)}$, and write 
$E=\bigsqcup_{a\in\Z/p}E_a$. 
Assume that $(V,E_0)$ is irreducible and that $E_a\neq\emptyset$ 
for some $a\neq0$. 
Use the labels $\ell(e)=a$ for $e\in E_a$ to define 
$\xi:\cK\to\Z/p$ as in 
Example~\ref{ex:cocyclesubgroupoidgraph}, and put 
\[ 
\cG:=\cK\times_\xi\Z/p,\qquad 
\lambda_c(g,b):=(g,b+c),\qquad 
\cH:=\cG\rtimes_\lambda\Z/p. 
\] 
The covering graph is strongly connected: the edges in $E_0$ 
connect each layer, and any nonzero label generates $\Z/p$. 
It has a vertex with at least two outgoing edges, so $\cG$ is 
minimal and purely infinite by Lemma~\ref{lem:graphpureinfiniteness}. 
If $\mathcal R_p$ is the pair groupoid on $\Z/p$, then 
\[ 
\cH\cong\cK\times\mathcal R_p, 
\qquad ((g,b),c)\longmapsto(g,b,b+\xi(g)-c). 
\] 
This is again an irreducible SFT groupoid. 
One graph model replaces each vertex $v$ by a chain 
$(v,0)\to\cdots\to(v,p-1)$ and each edge $e$ by an edge 
from $(i(e),p-1)$ to $(t(e),0)$. 

The inclusion $\cG\subset\cH$ is prime, and 
\[ 
N:=\Nor\left(\sF(\cH),\sF(\cG)\right) 
=\sF(\cG)\rtimes_\lambda\Z/p, 
\qquad \Out_{\cH}(\cG)\cong\Z/p, 
\] 
by \cite[Lemmas 5.3 and 5.4]{Ma26maximal1}. 
Thus these examples have a nontrivial outer quotient. 
Since $p$ is odd, each deck transformation is an even permutation 
of the $p$ clopen layers and hence a product of three-cycles. 
Each of these is a commutator of transpositions in $\sF(\cH)$, 
so the deck transformations belong to $\sD(\cH)$. 
Consequently, Theorem~\ref{thm:minimalnormalizermaximal} 
and the proof of Corollary~\ref{cor:minimalfullmaximal} give 
\[ 
N\text{ is maximal in }\sF(\cH) 
\iff H_1(\cG)\to H_1(\cH) 
\text{ is surjective}. 
\] 
Here and below the homomorphisms are induced by inclusion. 
On the other hand, Corollary~\ref{cor:minimalcommutatormaximal} (1) 
always makes $N\cap\sD(\cH)$ maximal in $\sD(\cH)$. 
As 
\[ 
N\cap\sD(\cH) 
=\left(\sF(\cG)\cap\sD(\cH)\right)\rtimes_\lambda\Z/p, 
\] 
this maximal subgroup equals $\sD(\cG)\rtimes_\lambda\Z/p$ 
if and only if $H_1(\sF(\cG))\to H_1(\sF(\cH))$ is injective. 
Hence $\sD(\cG)\rtimes_\lambda\Z/p$ is maximal in $\sD(\cH)$ 
exactly when this map is injective. 
These observations recover the two maximality criteria evaluated by finite 
matrices in \cite[Section 6.2 and Proposition 6.4]{Ma26maximal1}. 

For example, take $p=3$, $\cK=\cH_2$, and label the two loops 
by $0$ and $1$. The covering graph has adjacency matrix 
\[ 
B=\begin{bmatrix}1&1&0\\0&1&1\\1&0&1\end{bmatrix}. 
\] 
Since $H_0(\cG)=0$ and $\det(I-B)=-1$, 
\cite[Theorem 3.6]{MM14Kyoto} gives $\cG\cong\cH_2$. 
The prefix set $\{0,10,11\}$ identifies 
$\cH\cong\cH_2\times\mathcal R_3$ with $\cH_2$. 
Hence $N\cong V_2\rtimes_\lambda\Z/3$ is a maximal subgroup 
of $V_2$, recovering \cite[Example 6.5]{Ma26maximal1}. 
Under $(z,i)\mapsto\rho_i z$, with 
$(\rho_0,\rho_1,\rho_2)=(0,10,11)$, the subgroupoid $\cG$ 
becomes the kernel of the cocycle counting $1$'s modulo three. 
Its type rules are $P_i0=P_i$ and $P_i1=P_{i+1}$. 
The prefix cycle $(0\ 10\ 11)$ implements the deck transformation 
and cyclically permutes these types. 
Thus $\sF(\cG)$ and $N$ correspond respectively to 
$\Fix(\mathcal P)$ and $\operatorname{Stab}(\mathcal P)$, 
and the maximality also follows from 
\cite[Theorem 5.12]{BBQS25TAMS}. 
\end{example} 

\begin{example}[A homological obstruction] 
\label{ex:selfnormalizingnonmaximal} 
The two conditions in Corollary~\ref{cor:minimalfullmaximal} (2) 
are independent. 
In Example~\ref{ex:paritynormalizers} (2), the homology map is 
surjective and the outer quotient is nontrivial. 
Here we construct an example with trivial outer quotient for which 
$H_1(\cG)\to H_1(\cH)$ is not surjective, by restricting the 
cyclic cover in Example~\ref{ex:cycliccovers} to one layer. 

Let $\cH$ be the SFT groupoid of the two-vertex graph with 
labelled adjacency matrices 
\[ 
A_0=\begin{bmatrix}0&1\\1&0\end{bmatrix},\quad 
A_1=\begin{bmatrix}2&0\\0&2\end{bmatrix},\quad A_2=0, 
\quad A=A_0+A_1+A_2=\begin{bmatrix}2&1\\1&2\end{bmatrix}. 
\] 
Let $\xi:\cH\to\Z/3$ be the associated cocycle and put 
$\cG:=\Ker\xi$. 
By Example~\ref{ex:cocyclesubgroupoidgraph}, $\cG$ is the 
full clopen reduction of $\cH\times_\xi\Z/3$ to 
$X\times\{0\}$, where $X=\cH^{(0)}$. 
The covering graph has adjacency matrix 
\[ 
B=\begin{bmatrix} 
A_0&A_1&0\\0&A_0&A_1\\A_1&0&A_0 
\end{bmatrix}. 
\] 
It is strongly connected and has a vertex with at least two 
outgoing edges, so $\cG$ is minimal and purely infinite by 
Lemma~\ref{lem:graphpureinfiniteness}. 
The inclusion $\cG\subsetneq\cH$ is prime. 
Indeed, any nontrivial compatible equivalence relation on the 
covering graph identifies two layers over one base vertex. 
The labelled edges then identify all three layers over both 
base vertices, giving the base graph. 

The homology computation gives 
\[ 
H_1(\cG)\cong\Ker(I-B^t) 
=\Z(1,-1,1,-1,1,-1)^t,\qquad 
H_1(\cH)\cong\Z(1,-1)^t. 
\] 
As in \cite[Section 6.2]{Ma26maximal1}, the map induced by 
inclusion sums the three layer coordinates. 
It is therefore multiplication by three, and is not surjective. 

To compute the outer quotient, write $a_i,b_i$ for the vertex 
classes in $H_0(\cG)\cong\Coker(I-B^t)$, with $i\in\Z/3$. 
The relations are 
\[ 
a_i=b_i+2a_{i+1},\qquad b_i=a_i+2b_{i+1}. 
\] 
Eliminating the $b_i$ leaves $4(a_{i+2}-a_{i+1})=0$. 
Consequently, with $t_j:=a_j-a_0$ for $j=1,2$, we have 
\[ 
H_0(\cG)\cong\Z a_0\oplus(\Z/4)t_1\oplus(\Z/4)t_2, 
\qquad [1_X]=a_0+b_0=2t_1. 
\] 
The automorphisms of the covering graph over the base graph 
are exactly the three layer translations. 
The two nontrivial translations send $[1_X]$ to 
$2t_1+2t_2$ and $2t_2$, respectively. 
Neither fixes $[1_X]$, so neither can be induced by an element 
of $\Nor\left(\sF(\cH),\sF(\cG)\right)$. 
Proposition~\ref{prop:subgroupoidnormalizer} therefore gives 
$\Out_{\cH}(\cG)=1$, $\Nor\left(\sF(\cH),\sF(\cG)\right)=\sF(\cG)$. 
Nevertheless, $\sF(\cG)$ is not maximal in $\sF(\cH)$, 
by Corollary~\ref{cor:minimalfullmaximal} (2). 
Thus restriction to one layer removes the outer symmetry of 
the cyclic cover while retaining the homological obstruction 
to maximality. 
\end{example}

\section*{Acknowledgements}
The author used ChatGPT (GPT-6 Astra) and Claude (Opus 5.5) 
to assist with mathematical discussions, checking arguments, 
literature searches, and the drafting and revision of the manuscript. 
The mathematical arguments and references were independently 
verified by the author, who is responsible for the final content.

\bibliography{../zzz} 
\bibliographystyle{amsplain}

\end{document}